\documentclass[11pt, a4paper, reqno]{amsart}%
\usepackage{epsfig,subfigure}
\usepackage[utf8]{inputenc}
\usepackage{amssymb}
\usepackage{amsfonts}
\usepackage{amsthm}
\usepackage{amsmath,multirow,rotating}
\usepackage[round]{natbib}
\usepackage{bm}
\usepackage[colorlinks=true,linkcolor=blue,citecolor=blue,pdfborder={0 0 0}]{hyperref}
\usepackage{color}
\usepackage{paralist}

\usepackage{mathrsfs}
\usepackage{dsfont}

\usepackage{enumitem}
\usepackage{graphicx}%

\newtheorem{theorem}{Theorem}[section]

\newtheorem{lem}[theorem]{Lemma}
\newtheorem{cor}[theorem]{Corollary}
\newtheorem{thm}[theorem]{Theorem}

\theoremstyle{definition}

\newtheorem{cond}[theorem]{Condition}

\theoremstyle{remark}

\newtheorem{ex}[theorem]{Example}
\newtheorem{rem}[theorem]{Remark}

\numberwithin{equation}{section} \numberwithin{theorem}{section}

\newcommand{\Exp}{\operatorname{E}}		% Expected Value
\newcommand{\Cov}{\operatorname{Cov}}

\newcommand{\Var}{\operatorname{Var}}

\newcommand{\R}{\mathbb R}
\newcommand{\N}{\mathbb N}
\newcommand{\Z}{\mathbb Z}

\newcommand{\eps}{\varepsilon}
\newcommand{\Jc}{\mathcal{J}}
\newcommand{\scs}{\scriptscriptstyle}
\newcommand{\ip}[1]{\lfloor #1 \rfloor}

\newcommand{\I}{\mathds{1}}
\newcommand{\ind}{\mathds{1}}

\newcommand{\wto}{\stackrel{d}{\longrightarrow}}

\newcommand{\Nor}{\mathcal{N}}
\newcommand{\op}{o_\Prob(1)}
\newcommand{\ho}{^}
\newcommand{\intne}{\int_{0}\ho{1}}
\newcommand{\intnu}{\int_{0}\ho{\infty}}

\newcommand{\sumki}{\sum_{i=1}\ho{k_n}}

\newcommand{\Pro}{\mathbb{P}}
\newcommand{\Prob}{\mathbb{P}}

\newcommand{\di}{\mathrm{d}}

\newcommand{\slb}{{{\operatorname{sb}}}} 		% Sliding Blocks
\newcommand{\djb}{{{\operatorname{db}}}}		% Disjoint Blocks
\newcommand{\mb}{{{\operatorname{mb}}}}		% Generic Block Method

\newcommand{\pbh}{\hat{\bar{p}}}

\newcommand{\hhutsl}{\hat H_{n}^{\slb}}
\newcommand{\kn}{\mathcal{K}_n}

\newcommand{\Bc}{\mathcal{B}}
\newcommand{\Fc}{\mathcal{F}}

\allowdisplaybreaks

\begin{document}

\title[New estimators for the cluster size distribution]{Statistical analysis for stationary time series at extreme levels: new estimators for the limiting cluster size distribution}

\author{Axel B\"{u}cher}
\thanks{\textit{Corresponding author}: 
Axel Bücher (\href{mailto:axel.buecher@hhu.de}{axel.buecher@hhu.de})}
\author{Tobias Jennessen}

\address{Heinrich-Heine-Universit\"at D\"usseldorf, Mathematisches Institut, Universit\"atsstr.~1, 40225 D\"usseldorf, Germany.}
\email{axel.buecher@hhu.de}
\email{tobias.jennessen@hhu.de}

\date{\today}

\begin{abstract}
A measure of primal importance for capturing the serial dependence of a stationary time series at extreme levels is provided by the limiting cluster size distribution. New estimators based on a blocks declustering scheme are proposed and analyzed both theoretically and by means of a large-scale simulation study. A sliding blocks version of the estimators is shown to outperform a disjoint blocks version. In contrast to some competitors from the literature, the estimators only depend on one unknown parameter to be chosen by the statistician.

\medskip

\end{abstract}

\maketitle

%\begin{keyword}[class=AMS]
%\kwd[Primary ]{62G32} %	Nonparametric Inference, Statistics of extreme values; tail inference
%\kwd{62E20} %Distribution theory: Asymptotic distribution theory
%\kwd{62M09} %	Inference from stochastic processes: Non-Markovian processes: estimation
%\kwd[; secondary ]{60G70}	%Stochastic processes:Extreme value theory; extremal processes
%\kwd{62G20.} %Nonparametric Inference: Asymptotic Properties
%\end{keyword}

\section{Introduction} %\label{sec:intro}

The serial dependence of a stationary time series $(X_t)_{t\in\Z}$ at extreme levels may be described by various, partially interrelated limiting objects. 
The most traditional approach consists of studying the point process of exceedances and its weak convergence (see \citealp{HsiHusLea88}, or Section 10.3 in \citealp{BeiGoeSegTeu04}). Two characterizing objects show up in the limit: the extremal index $\theta\in[0,1]$ and the limiting cluster size distribution $\pi$, a probability distribution on the positive integers with $\pi(m)$ approximately representing the probability that extreme observations occur in a temporal cluster of size $m$. Under mild additional assumptions, the extremal index is in fact the reciprocal of the expectation of the limiting cluster size distribution \citep{Lea83}. 

A recent alternative object for assessing the serial dependence is given by the tail process $(Y_t)_{t\in\Z}$ (or the spectral process $(\Theta_t)_{t\in\Z}$) that may be associated with a suitably standardized version of $(X_t)_{t\in\Z}$ \citep{BasSeg09}.   Heuristically, the law of those processes on $\R^\Z$ provides a more detailed description of the serial dependence.  In fact, relying on results from \cite{KulSou20}, it can be shown that the limiting cluster size distribution $\pi$ may be expressed as a functional of the tail process under mild additional conditions, see Remark~\ref{rem:tp} below.

Estimating the above mentioned objects based on a finite stretch of observations has received a lot of attention in recent years. For instance, estimators for the extremal index have been studied in \cite{Hsi93, SmiWei94, FerSeg03, Suv07, Nor15, BucJen20}, among many others. Estimators for $\pi$ have been studied in \cite{Hsi91, Fer03, Rob09, Rob09b}. To the best of our knowledge, inference on the law of the tail process has only been studied for selected functionals (note that the above mentioned contributions fall into this category as well). For instance, \cite{DreSegWar15, DavDreSegWar18,DreKne20} investigate estimators for the c.d.f.\ of $Y_t$, at a fixed lag $t$, which are based on making sophisticated use of the time change formula. \cite{CisKul20} consider sliding blocks versions of peak-over-threshold estimators for a general class of functionals, including the extremal index and the limiting cluster size distribution. It worthwhile to mention that asymptotic theory for many of the afore-mentioned estimators may be (non-trivially) derived from high level results in \cite{DreRoo10} on empirical processes for cluster functionals, see also \cite{KulSouWin19}.

The present paper is motivated by the apparently little amount of well-studied estimators for the limiting cluster size distribution $\pi$. Inspired by recent contributions on the estimation of  the  extremal index, we study an estimator that is based on a (disjoint or sliding) blocks declustering method. The sliding blocks estimator is shown to be more efficient than the disjoint blocks version.
Moreover, by extensive Monte Carlo simulations, they are shown to exhibit very good finite-sample behavior in comparison to the competitors from \cite{Hsi91, Fer03, Rob09}.

The remaining parts of this paper are organized as follows: mathematical preliminaries, including precise definitions of the limiting objects described above, are provided in Section~\ref{sec:math}. In that section, we also define the new estimators. Regularity conditions needed to derive asymptotic normality are collected in Section~\ref{sec:cond}, with the respective theoretical results given in Section~\ref{sec:main}. Section~\ref{sec:sim} contains results from a large scale Monte Carlo simulation study. The main  arguments for the proofs are collected  in Section~\ref{sec:proofs}, with an interesting side result on weak convergence of an empirical process associated with compound probabilities presented in Section~\ref{sec:comp} and proven in Section~\ref{sec:proofcomp}. Finally, all remaining proofs are deferred to a sequence of appendices and additional simulation results are collected in Appendix~\ref{more_figures}.

\section{Mathematical preliminaries and definition of estimators}
\label{sec:math}

Throughout the paper, $(X_t)_{t \in \Z}$ denotes a stationary time series with marginal cumulative distribution function (c.d.f.) $F$. 
The sequence is assumed to have an \textit{extremal index} $\theta \in (0,1]$, i.e., we assume that, for any $\tau >0$, there exists a sequence $(u_n(\tau))_{n \in \N}$ such that 
$\lim_{n \to \infty} n \bar{F}(u_n(\tau))=\tau$ and 
\begin{align} 
\lim_{n \to \infty} \Pro(M_{1:n} \leq u_n(\tau)) = e\ho{-\theta \tau}, \label{ei_def}
\end{align} 
where $\bar{F} = 1-F$ and $M_{1:n} = \max\{X_1,\ldots,X_n\}$. Some thoughts reveal that, if the extremal index exists, then the convergence  in \eqref{ei_def} holds for any sequence $u_n(\tau)$ such that $\lim_{n \to \infty} n \bar{F}(u_n(\tau))=\tau$ (see, e.g., the beginning of Section 5 in \citealp{HsiHusLea88}) and that
we may always choose
$u_n(\tau)=F^{\leftarrow}(1-\tau/n)$ (see the proof of Theorem 1.7.13 in \citealp{LeaLinRoo83}). Subsequently, the latter definition is tacitly employed, where $F^{\leftarrow}(p) = \inf\{ x \in \R: F(x) \ge p\}$ denotes the (left-continuous)  generalized inverse of $F$.

The point process of exceedances is defined as 
 \[ 
 N_n\ho{(\tau)}(B) = \sum_{t=1}\ho{n} \I(t/n \in B, X_t > u_n(\tau)),
 \] 
 for any Borel set $B \subset E:=(0,1]$ and $\tau \geq 0$. If the time series is serially independent, then it is well-known that $N_n^{\scs (\tau)}$ converges in distribution to a  homogeneous Poisson process   on $E$ with intensity $\tau$. In the serial dependent case, if the extremal index exists and a certain mixing condition is met, then a necessary and sufficient condition for weak convergence of $N_n^{\scs (\tau)}$ is as follows, see Theorems~4.1 and 4.2 in  \cite{HsiHusLea88}: there exists a $\Delta (u_n(\tau))$-separating sequence $(q_n)_n$ (see Section~\ref{sec:cond} for a definition) such that the following limit exists for all $m\in\N_{\ge 1}$: 
\begin{align} \label{eq:pi}
\pi(m) = \lim_{n \to \infty} \pi_n(m), \quad \pi_n(m) = \Pro(N_{n}\ho{(\tau)}(B_n)=m \mid N_{n}\ho{(\tau)}(B_n) > 0),
\end{align}
where $B_n = (0,q_n/n]$. In that case, the convergence in the last display holds for any $\Delta(u_n(\tau))$-separating sequence $(q_n)_n$ and the weak limit of $N_n^{\scs (\tau)}$, say $N^{\scs (\tau)}$, is a compound poisson process with  intensity $\theta \tau$ and compounding distribution $\pi$, notionally $N^{\scs (\tau)} \sim \mathrm{CPP}(\theta \tau, \pi)$.   If the $\Delta(u_n(\tau))$-condition holds for all $\tau>0$, then $\pi$ does not depend on $\tau$  (\citealp{HsiHusLea88}, Theorem~5.1), which will be tacitly assumed throughout. Motivated by \eqref{eq:pi}, the distribution $\pi$ is commonly referred to as the \textit{(limiting) cluster size distribution}.

\begin{rem} \label{rem:tp}
The limiting cluster size distribution is closely connected to the tail process introduced in \cite{BasSeg09}, see also the monograph \cite{KulSou20}. Since the tail process may only be defined for heavy tailed stationary time series, a standardization is necessary first. For simplicity, we assume that $F$ is continuous. In that case, for any $t \in \Z$, $Z_t = 1/\{1-F(X_t)\}$ is standard Pareto-distributed and the event $X_t > u_n(\tau)$ is (almost surely) equivalent to $Z_t>n/\tau$. Under the assumption that $(Z_t)_{t \in \Z}$ is regularly varying 	(i.e., all vectors of the form $(Z_k,\ldots,Z_\ell)$ are multivariate regularly varying), there exists a process $(Y_t)_{t \in \Z}$, the \textit{tail process of $(Z_t)_{t \in \Z}$}, such that, for every $s,t\in \Z$ with $s \le t$,
\[ 
\Pro \big( x^{-1} (Z_s,\ldots,Z_t) \in \cdot \mid Z_0 >x \big) 
\stackrel{w}{\to} 
\Pro \big( (Y_s,\ldots,Y_t) \in \cdot \big) \quad (x\to\infty),
\] 
see Theorem 2.1 in \cite{BasSeg09}. If we additionally assume that, for the sequence $(q_n)_n$ from \eqref{eq:pi} and for all $x,y>0$,
\begin{align}
	\lim_{m \to \infty} \limsup_{n \to \infty} \Pro \Big( \max_{m \leq |t| \leq q_n} Z_t > n x \mid  Z_0 > n y \Big) = 0, \label{cond_tailpr}
	\end{align}
then $\pi$ may be expressed through the tail process, see Example 6.2.9 in \cite{KulSou20}:
	\begin{align*}
	\pi(m) 
	=& 
	\lim_{n \to \infty} \Pro \Big( \sum_{1 \le t \le q_n} \I(X_t >u_n(\tau)) = m \, \Big| \, \max_{1 \le t \le q_n} X_t > u_n(\tau) \Big) \\ 
	=& \ \Pro \Big( \sum_{t \ge 0} \I(Y_t>1) = m \, \Big| \, \max_{t \leq -1} Y_t \leq 1 \Big), \quad m \in \N_{\ge 1}.
	\end{align*}
	In other words, $\pi(m)$ is the conditional probability that the `number of time points where the tail process exceeds the value 1' equals $m$, conditional on the event that the tail process does not exceed 1 until $t=-1$.
	It is worthwhile to mention that \eqref{cond_tailpr} is for instance satisfied for geometrically ergodic Markov chains, short-memory linear or max-stable processes and m-dependent sequences; see \cite{CisKul20}, page 7, and \cite{KulSou20}, page 151.
	\end{rem}
	
Let $N_E^{\scs (\tau)}$ denote the distributional limit of $N_n^{\scs(\tau)}(E)$. Since the distribution of $N^{\scs(\tau)}$ is $\mathrm{CPP}(\theta \tau, \pi)$, we have the stochastic representation
\[
N_E^{(\tau)} \stackrel{d}{=} \sum_{i=1}^{\eta(\theta \tau)} \xi_i
\]
for independent random variables $\eta(\theta\tau) \sim \mathrm{Poisson}(\theta \tau)$ and $\xi_i \sim \pi$. 
As a consequence,  we have
\begin{align*}
p\ho{(\tau)}(0) &= \Prob(N^{\scs (\tau)}(E)=m) 
= 
e\ho{-\theta \tau}, \\
p\ho{(\tau)}(m) &= \Prob(N^{\scs (\tau)}(E)=m) 
=
\sum_{j=1}\ho{m} \frac{e\ho{-\theta \tau}(\theta \tau )\ho{j}}{j!} \pi\ho{\ast j}(m), \quad m\in\N_{\ge 1},
\end{align*} 
where $\pi\ho{\ast j}$ is the $j$-th convolution of $\pi$. 
As explicitly written down in Equation (1.5) in \cite{Rob09}, the previous equations allow to obtain, for any $\tau>0$, a recursion expressing $\pi(m)$ as a function of $\theta, p^{\scs (\tau)}(1), \dots, p^{\scs (\tau)}(m)$ and $\pi(1), \dots,  \pi(m-1)$. This recursion then allows for estimation of $\pi(m)$ based on estimation of $\theta, p^{\scs (\tau)}(1), \dots, p^{\scs (\tau)}(m)$, which is precisely the approach followed in \cite{Rob09}.

It may be argued that this 	approach suffers from the fact that the obtained recursion is depending on $\tau$, which ultimately implies that the final estimator depends on $\tau$ as well. Hence, the statistician has either to make a choice, or to apply a suitable aggregation scheme. Within the present paper, we propose to instead consider a different recursion based on
\[
\bar p(m) = \int_0^\infty p^{(\tau)}(m) \theta e^{-\theta \tau} \, \mathrm d\tau 
=
\Exp[p^{(Z)}(m)],
\]
where $Z \sim \mathrm{Exponential}(\theta)$. Perhaps surprisingly, and unlike for $p^{\scs (\tau)}(m)$ above, the respective recursion does not even depend on $\theta$, which allows for even simpler estimation. More precisely, a simple calculation shows that
$\bar{p}(0) = \intnu \theta e\ho{-2\theta \tau} \ \mathrm{d}\tau = 1/2$ and 
\begin{align*}
 \bar{p}(m) &= \sum_{j=1}^m \pi^{\ast j} (m) \frac{\theta}{j!} \intnu (\theta \tau)^j e^{-2 \theta \tau} \ \mathrm{d}\tau 
 = \sum_{j=1}^m \frac{1}{2^{j+1}} \pi^{\ast j}(m)
\end{align*}
for $m \in \N_{\ge1}$. As a consequence,
\begin{align*}
 \bar{p}(m) 
 &= \frac{1}{4} \pi(m) + \sum_{j=2}^m \frac{1}{2^{j+1}} \sum_{k=j-1}^{m-1} \pi^{\ast (j-1)}(k) \pi(m-k) \\
 &= \frac{1}{4} \pi(m) + \sum_{k=1}^{m-1} \pi(m-k) \sum_{j=2}^{k+1} \pi^{\ast(j-1)}(k) \frac{1}{2^{j+1}} \\
 &= \frac{1}{4} \pi(m) +\frac{1}{2} \sum_{k=1}^{m-1} \pi(m-k) \bar{p}(k),
\end{align*}
which in turn implies
\begin{equation}
 \pi(m) = 4 \bar{p}(m) - 2 \sum_{k=1}^{m-1} \pi(m-k) \bar{p}(k), \quad m \in \N_{\ge 1}. \label{RecursionPi}
\end{equation}
Obviously, Equation~\eqref{RecursionPi} allows to recursively derive $(\pi(1), \dots, \pi(m))$ from $(\bar p(1), \dots, \bar p(m))$. The plug-in principle hence allows to estimate the former vector based on suitable estimators for the latter vector. 

For the estimation of $(\bar p(1), \dots, \bar p(m))$,  a transformation extensively used in \cite{BerBuc18} and \cite{BucJen20} comes in handy: 
the random variable 
\begin{align*} 
Z_{1:n} = n\{1- F(M_{1:n}) \}
\end{align*}
is asymptotically exponentially distributed with parameter $\theta$, for $n\to\infty$. Indeed, since 
$v_n(\tau)=F^\rightarrow(1-\tau/n)$ with the right-continuous generalized inverse $F^{\rightarrow}$ satisfies $\lim_{n\to\infty} \bar  F(v_n(\tau))=\tau$, whence
\begin{align}
	\Pro(Z_{1:n} \geq \tau) = \Pro(M_{1:n} \leq v_n(\tau)) \to e^{-\theta \tau} \label{z_conv}
\end{align}
for $n\to\infty$ by \eqref{ei_def}.
Next, for motivating our estimator it is instructive to consider, for two independent copies $(X_t)_{t\in \Z}, (\tilde X_t)_{t\in\Z}$, the random variable
\[
N_n^{(\tilde Z_{1:n})}(E) = \sum_{t=1}^n \ind(X_t > u_n(\tilde Z_{1:n}))
\stackrel{a.s.}{=} \sum_{t=1}^n \ind(X_t > \tilde M_{1:n}),
\]
where $\tilde Z_{1:n}=n\{1-F(\tilde M_{1:n})\}$.
Then, conditional on $\tilde Z_{1:n}$, the random variable
$
N_n^{\scs (\tilde Z_{1:n})}(E) 
$
approximately follows a compound poisson distribution with intensity $\theta \tilde Z_{1:n}$ and compounding distribution $\pi$, for sufficiently large $n$. As a consequence, 
\[
\Prob\big(N_n^{(\tilde Z_{1:n})}(E)  = m \mid \tilde Z_{1:n}\big) \approx p^{(\tilde Z_{1:n})}(m),
\]
which readily implies
\begin{align} \label{eq:mm}
\Prob\Big(N_n^{(\tilde Z_{1:n})}(E) = m\Big) 
\approx \Exp\Big[p^{(\tilde Z_{1:n})}(m)\Big]  \approx \bar p(m),
\end{align}
where the second approximation  is due to \eqref{z_conv}. The latter display allows for estimation of 
$\bar p(m)$ based on the method of moments.

More precisely, suppose we observe a finite-stretch of observation from the time series, say $X_1, \dots, X_n$.  Divide the observation period into non-overlapping successive blocks of size $b=b_n$, that is, into blocks $I_i=I_i^{\djb}$ ($\djb$ for `disjoint blocks'),
\[
I_1=\{1, \dots, b\}, \quad I_2=\{b+1, \dots, 2b\}, \quad \dots \quad,
I_{k}=\{(k-1)b+1, \dots, kb\}, 
\]
where $k=k_n=\lfloor n/b_n\rfloor$. A possible remainder block $I_{k+1}^\circ=\{kb+1, \dots, n\}$ of cardinality $|I_{k+1}|<b_n$ will have a negligible influence on the subsequent estimators and will hence be discarded. Asymptotically, $(b_n)_n$ needs to be an intermediate sequence satisfying $b=b_n\to \infty$ and $b_n=o(n)$. Now, by well-known heuristics, cluster functionals (i.e., statistics that depend only the `large observations'  within a specific block $I_j$)
calculated based on disjoint blocks of observations may be considered asymptotically independent, whence \eqref{eq:mm} suggests to estimate $\bar p(m)$ by
\begin{align*} 
\pbh_n(m) = \pbh_n^{\,\djb} (m)
&= 
\frac{1}{k_n(k_n-1)} \sum_{\substack{i,i'=1 \\ i \neq i'}}^{k_n} \I \bigg\{ \sum_{s \in I_{i'}} \I\big(X_s > {M}_{ni}^{\djb} \big)=m \bigg\},
\end{align*} 
where the upper index `$\mathrm{db}$' refers to the fact that the underlying blocks are disjoint and where $M_{ni}^{\djb}=\max\{X_t: t \in I_i\}$.  Following \cite{BerBuc18,BucJen20}, a possibly more efficient version that is based on sliding/overlapping blocks instead of disjoint blocks is given by
\begin{align*} 
\pbh_n^{\,\slb}(m)
&= 
\frac{1}{|D_n|} \sum_{(i,i') \in D_n} \I \bigg\{ \sum_{s \in I_{i'}^{\slb}} \I \Big( X_s >  M_{ni}^\slb \Big) =m \bigg\}, 
\end{align*}
where  $I_{i}^{\slb}=\{i,\ldots,i+b_n-1\}, M_{ni}^\slb = \max\{X_t: t \in I_i^{\slb}\}$ and 
where $D_n$ is the set of all pairs $(i,i')\in \{1, \dots, n-b_n+1\}^2$ such that $I_{i}^\slb \cap I_{i'}^\slb=\varnothing$.
Obviously, since $I_{i}^\slb \cap I_{i'}^\slb=\varnothing$, the same heuristics as in the disjoint blocks case applies: the expectation of each summand is approximately equal to $\bar p(m)$. 

Based on the recursion \eqref{RecursionPi}, the final (disjoint and sliding blocks) estimators for $\pi(m)$, $m \in \N_{\ge1}$, are defined, for $\mb \in \{\mathrm{db}, \mathrm{sb}\}$, by
\begin{align} \label{eq:pihat}
\hat \pi_n^{\mb}(m) = 4\pbh_n^{\,\mb}(m) - 2 \sum_{k=1}^{m-1} \hat \pi^{\mb}_n(m-k) \pbh_n^{\,\mb}(k).
\end{align}

\section{Regularity conditions} \label{sec:cond}
	
This section summarizes technical regularity conditions  which are imposed to derive asymptotic properties for the estimators from the previous section.
First of all, the serial dependence will be controlled via alpha- and beta-mixing coefficients. For two sigma-fields $\Fc_1, \Fc_2$ on a probability space $(\Omega, \Fc, \Prob)$, let
\begin{align*}
 \alpha({\Fc}_1,{\Fc}_2)
	&=	\sup_{A \in {\Fc}_1, B\in {\Fc}_2} |\Prob(A\cap B)-\Prob(A)\Prob(B)|,\\
 \beta(\mathcal{F}_1,\mathcal{F}_2) &= \frac{1}{2} \sum_{i \in I} \sum_{j \in J} \sup |\Pro(A_i \cap B_j) - \Pro(A_i) \Pro(B_j)|,	
\end{align*}
where the last supremum is over all finite partitions $(A_i)_{i \in I} \subset \mathcal{F}_1$ and $(B_j)_{j \in J} \subset \mathcal{F}_2$ of $\Omega$. 
For $-\infty \le p < q \le \infty$ and $\eps\in(0,1]$, let $\Bc_{p:q}^\eps$  denote the sigma algebra generated by $U_s^\eps:=U_s \ind(U_s > 1- \eps)$ with $s\in \{p, \dots, q\}$; here, $U_s = F(X_s)$. Finally, for $\ell\ge 1$, let
 \begin{align*}
 \alpha_{\eps}(\ell) &= \sup_{k \in \N} \alpha(\Bc_{1:k}^\eps, \Bc_{k+\ell:\infty}^\eps), \ \ \ 
 \beta_{\eps}(\ell) = \sup_{k \in \N} \beta(\Bc_{1:k}^\eps, \Bc_{k+\ell:\infty}^\eps).
 \end{align*} 
Conditions on the decay of the mixing coefficients will be imposed below.  

Fix $m\ge 1$ and $\tau_1> \dots > \tau_m>0$. For $1\le p < q \le n$, let
$\Fc_{p:q,n}^{(\tau_1, \dots, \tau_m)}$ denote the sigma-algebra generated by the events $\{X_s > u_n(\tau_j)\}$ for $ s\in \{p, \dots, q\}$ and  $j \in \{1, \dots m\}$. For $\ell \in\{1, \dots, n\}$, define
\begin{multline*}
\alpha_{n,\ell}(\tau_1, \dots, \tau_m)=\sup\{ |\Prob(A\cap B) - \Prob(A) \Prob(B)| : \\
A \in \Fc_{1:s,n}^{(\tau_1, \dots, \tau_m)}, B \in \Fc_{s+\ell:n,n}^{(\tau_1, \dots, \tau_m)}, 1 \le s \le n-\ell\}.
\end{multline*}
The condition $\Delta_n(\{u_n(\tau_j)\}_{1\le j \le m})$ is said to hold if there exists a sequence $(\ell_n)_n$ with $\ell_n=o(n)$ such that $\alpha_{n,\ell_n}(\tau_1, \dots, \tau_m) =o(1)$ as $n\to\infty$. 
A sequence $(q_n)_n$ with $q_n=o(n)$ is said to be $\Delta_n(\{u_n(\tau_j)\}_{1\le j \le m})$-separating if there exists a sequence $(\ell_n)_n$ with $\ell_n=o(q_n)$ such that $\alpha_{n,\ell_n}(\tau_1, \dots,\tau_m) =o(q_n/n)$ as $n\to\infty$. 
If $\Delta_n(\{u_n(\tau_j)\}_{1\le j \le m})$ is met, then such a sequence always exists, simply take $q_n=\ip{\max\{ n\alpha_{n,\ell_n}^{\scs 1/2}, (n\ell_n)^{\scs 1/2}\}}.$ 

As already stated in Section~\ref{sec:math}, by Theorems 4.1 and 4.2 in \cite{HsiHusLea88},
if the extremal index exists and the $\Delta(u_n(\tau))$-condition is met ($m=1$), then a necessary and sufficient condition for weak convergence of $N_n^{\scriptscriptstyle(\tau)}$ 
is the convergence in \eqref{eq:pi} for some $\Delta(u_n(\tau))$-separating sequence $(q_n)_n$.
Moreover, in that case, the convergence in \eqref{eq:pi} holds for any $\Delta(u_n(\tau))$-separating sequence $(q_n)_n$, and the weak limit of $N_n^{\scriptscriptstyle(\tau)}$, say $N^{(\scs \tau)}$, is a compound poisson process $\mathrm{CPP}(\theta \tau, \pi)$. If the $\Delta(u_n(\tau))$-condition holds for any $\tau>0$, then $\pi$ does not depend on $\tau$ (\citealp{HsiHusLea88}, Theorem~5.1). 

A multivariate version of the latter results is stated in \cite{Per94}, see also the summary in \cite{Rob09}, page 278, and the thesis \cite{Hsi84}. Suppose that the extremal index exists and that  the $\Delta(u_n(\tau_1), u_n(\tau_2))$-condition is met for any $\tau_1\ge \tau_2\ge0, \tau_1 \ne0$. Moreover, assume that there exists a family of probability measures $\{\pi_2^{\scs(\sigma)}: \sigma\in [0,1]\}$ on $\Jc = \{(i,j) \in \N_{\ge 0}^2: i \ge j \ge 0, i \ge 1\}$, such that, for all $(i,j) \in \Jc$,
\[
\lim_{n\to\infty} \Prob(N_n^{(\tau_1)} (B_n) = i, N_n^{(\tau_2)} (B_n) = j \mid N_{n}^{(\tau_1)}(B_n) >0) = \pi_2^{(\tau_2/\tau_1)}(i,j),
\]
where $q_n$ is some $\Delta(u_n(\tau_1), u_n(\tau_2))$-separating sequence.  In that case, the two-level point process $\bm N_n^{\scs (\tau_1,\tau_2)}=(N_{\scs n,1}^{\scs (\tau_1,\tau_2)}, N_{\scs n,2}^{\scriptscriptstyle(\tau_1,\tau_2)})$ converges in distribution to a point process $\bm N^{\scs (\tau_1,\tau_2)} = (N_{\scs 1}^{\scs (\tau_1, \tau_2)}, N_{\scs 2}^{ \scs (\tau_1,\tau_2)})$ with characterizing Laplace transform explicitly stated in \cite{Rob09} on top of page 278. Throughout, let
\begin{align*} 
\bm N_E^{(\tau_1,\tau_2)} = (N^{(\tau_1,\tau_2)}_{E,1}, N^{(\tau_1,\tau_2)}_{E,2}) = \bm N^{(\tau_1, \tau_2)}(E),
\end{align*}
whose marginal distributions are equal to $N_E^{\scs (\tau_1)}$ and $N_E^{\scs (\tau_2)}$ and which further allows for the stochastic representation 
\[
\bm N_E^{(\tau_1, \tau_2)} \stackrel{d}{=} \sum_{i=1}^{\eta(\theta \tau_1)} (\xi_{i,1}^{(\tau_2/\tau_1)}, \xi_{i,2}^{(\tau_2/\tau_1)}),
\]
where $\eta(\theta \tau_1) \sim \mathrm{Poisson}(\theta \tau_1)$ is independent of the bivariate i.i.d.\ sequence $(\xi_{i,1}^{\scs (\tau_2/\tau_1)}, \xi_{i,2}^{\scs (\tau_2/\tau_1)}) \sim \pi_2^{\scs (\tau_2/\tau_1)}$. As a consequence, the distribution of $\bm N_{\scs E}^{\scriptscriptstyle(\tau_1,\tau_2)}$ on $\N_{\ge 0}^2$, say 
\[
p_2^{(\tau_1,\tau_2)}(i,j)  = \Prob( \bm N_E^{(\tau_1, \tau_2)}  = (i,j)),
\] 
 is given by 
$p_2^{\scs (\tau_1,\tau_2)}(0,0) = e^{-\theta \tau_1}$, $p_2^{\scs (\tau_1,\tau_2)}(i,j)=0$ for $i<j$ and 
\[ 
p_2^{(\tau_1,\tau_2)}(i,j) = e^{-\theta \tau_1} \sum_{k=1}^{i} \frac{(\theta \tau_1)^k}{k!} \pi_2^{(\tau_2/\tau_1), \ast k}(i,j) , \quad i \geq j \geq 0, \ i \geq 1, 
\]
where $\pi_2^{\scs (\tau_2/\tau_1), \ast k}$ is the $k$-th convolution of $\pi_2^{\scs (\tau_2/\tau_1)}$.

The assumptions needed to derive asymptotic properties for $\pbh_n^{\mb}(m)$ and  $\hat \pi^{\mb}_n(m)$ are collected in the following condition.

\begin{cond} \label{cond}~
	\begin{enumerate}[leftmargin=.8cm]
		\item[(i)] The stationary time series $(X_s)_{s\in \N}$ has an extremal index $\theta \in (0,1]$ and the two-level  point process of exceedances $\bm N_n^{\scs (\tau_1,\tau_2)}$ converges weakly to $\bm N^{\scs (\tau_1, \tau_2)}$. 
		\item[(ii)] 
		There exist constants $\varepsilon_1 \in (0,1), \eta > 2$ and $C >0$ such that 
		\[ 
		\alpha_{\varepsilon_1}(n) \leq C n^{-\eta}  \qquad \forall\ n \in \N.
		\]
		The block size $b_n$ converges to infinity and satisfies 
		\[ 
		k_n = o(b_n^\eta), \ n \to \infty, 
		\] 
		(i.e., a slow decrease of the mixing  coefficients requires large block sizes). 	Further, there exists a sequence 
		$\ell_n \to \infty$ with $\ell_n =o(b_n)$ and $k_n\alpha_{\eps_1}(\ell_n) =o(1)$ as $n \to\infty$. 
		\item[(iii)] For some $c > 1-\eps_1$ with $\eps_1$ from (ii), one has 
		\[ 
		\lim_{n \to \infty} \Pro \Big( \min_{i=1,\ldots,2k_n} N_{ni}' \leq c \Big) =0, 
		\]
		where $N_{ni}'=\max \{ U_s, \ s \in [(i-1)b_n/2+1,\ldots,ib_n/2]\}$ for $i\in\{1,\ldots,2k_n\}$ and $U_s=F(X_s)$.
		\item[(iv)] (Bias.) For any $j \in \N_{\ge 1}$, as $n \to \infty$, 
		\[ 
		\Exp \big[ \varphi_{n,j}(Z_{1:b_n}) \big] = \bar{p}(j) + o\big( k_n\ho{-1/2} \big), 
		\] 
		where $\varphi_{n,j}(z)= \Pro(N_{b_n}^{\scs (z)}=j)$.
	\end{enumerate}
\end{cond}

The conditions are weaker versions of the conditions imposed in \cite{BerBuc18}, which in turn are mostly based on \cite{Rob09}.  In contrast to those papers, no moment condition on the increments of $\tau \mapsto N^{\scs (\tau)}_n(E)$ is needed, which may be explained by the fact that the cluster functionals showing up in the definition of $\hat{\bar p}_n(m)$ are bounded by 1. This also allows for a great simplification of the $\alpha$-mixing condition in comparison to the last-named references. For the treatment of the sliding blocks estimator, we will additionally impose a beta-mixing condition below, which is used for proving tightness of the scaled estimation error of empirical compound probabilities, see Section~\ref{sec:comp}. Further discussions and exemplary time series models meeting the conditions (i)-(iii) are provided in \cite{BerBuc18}.

%%%%%%%%%%%%%%%%%%%%%%%%%%%%%%%%%%%%%%%%%%%%%%%%%%%%%%%%%%%%%%%%%%%%%%%%%%%%%

\section{Main results} \label{sec:main}
In this section we derive asymptotic normality of both the disjoint and sliding blocks estimators from Section \ref{sec:math}. A comparison of the asymptotic variances shows that the sliding blocks version exhibits a smaller asymptotic variance than the disjoint blocks version. Subsequently, for $\mb \in \{\djb, \slb\}$, let
\begin{align}
\nonumber
 	s_{n,j}^\mb 
	&= \sqrt{k_n} \big\{ \hat{\bar{p}}_n^{\,\mb}(j) - \bar p(j) \big\}, \quad j \in \N_{\ge 1}, \\
\label{eq:vmn}
	v_{n,j}^\mb 
	&= \sqrt{k_n} \big\{ \hat \pi_n^{\mb}(j) - \pi(j) \big\}, \quad j \in \N_{\ge 1}.	
\end{align}
For simplicity, we will further assume that $F$ is continuous.

\begin{thm} \label{theo:djb}
	Assume that Condition \ref{cond} is met. Then, for any $m \in \N_{\ge 1}$,
	\[ 
	(s_{n,1}^\djb, \ldots, s^\djb_{n,m}) \wto (s_1^\djb, \ldots, s_m^\djb) \sim \Nor_{m}(0,\Sigma_{m}^\djb)
	\]
	as $n \to \infty$,
	where the covariance matrix $\Sigma_{m}^\djb = (d_{j,j'}^\djb)_{1\leq j,j' \leq m}$ is given by
	\begin{multline} \label{eq:sigmadb}
	d_{j,j'}^\djb = \intnu\intnu \Cov \Big( 
	\I(N_E^{(\tau)}=j)+ p^{(Z)}(j),  \\
	\I(N_E^{(\tau')}=j') + p^{(Z)}(j') \Big) \,\di H(\tau) \,\di H(\tau').
	\end{multline}
	Here, $H$ denotes the c.d.f.\ of the $\mathrm{Exp}(\theta)$-distribution,  and $N_E^{(\tau)}\sim p^{(\tau)}$ and $Z\sim\mathrm{Exp}(\theta)$ are such that
	\[
	\Pro(N_E^{(\tau)}=j, Z>\mu) = 
	\begin{cases}
	p_2^{(\tau,\mu)}(j,0)  & ,\tau \geq \mu \\
	e^{-\theta \mu} \ind(j=0) & ,\tau < \mu
	\end{cases} 
	\qquad(j\in\N_{\ge 0},\mu>0)
	\]
	and
	\[
	\Pro(N_E^{(\tau)}=j, N_E^{(\tau')}=j') = 
	\begin{cases}
	p_2^{(\tau,\tau')}(j,j')  & ,\tau \geq \tau'\\
	p_2^{(\tau',\tau)}(j',j)  & ,\tau < \tau'
	\end{cases}
	\qquad(j,j'\in\N_{\ge 0}).
	\]
\end{thm}

\begin{thm} \label{theo:slb}
	In addition to Condition \ref{cond} assume that $\sqrt{k_n} \beta_{\varepsilon_2}(b_n) =o(1)$ for some $\eps_2>0$. Then,  for any $m \in \N_{\ge 1}$,
	\[ 
	(s_{n,1}^{\slb}, \ldots, s_{n,m}^{\slb}) 
	\wto 
	(s_1^{\slb}, \ldots, s_m^{\slb}) \sim \Nor_{m}(0,\Sigma_{m}^{\slb})
	\]
	as $n \to \infty$, where the covariance matrix $\Sigma_m^{\slb} = (d_{j,j'}^{\slb})_{1 \leq j,j'\leq m}$ is given by 
   \begin{align}
	d_{j,j'}^{\slb} 
	&=
    2 \intne \Big\{  
    \intnu \intnu   \Cov \big( \I(X_{1,\xi}^{(\tau)}=j), \I(Y_{1,\xi}^{(\tau')}=j') \big) \, \di H(\tau) \di H(\tau') \nonumber \\
    &\hspace{2cm} + \intnu \! \!  \Cov\big( \I(X_{3, \xi}^{(\tau)}=j), p^{(Y_{3, \xi})}(j') \big) \, \di H(\tau) \nonumber \\
    & \hspace{2cm} + \intnu \! \! \Cov \big( \I(X_{3, \xi}^{(\tau)}=j'), p^{(Y_{3, \xi})}(j) \big) \, \di H(\tau) \nonumber \\
      & \hspace{2cm}+  \Cov \big( p^{(X_{2, \xi})}(j), p^{(Y_{2, \xi})}(j') \big) \Big\} \, \di \xi , \label{eq:sigmasl} 
    \end{align}
    where for $0 \leq \tau \leq \tau'$ and $x,y>0$,
    \begin{align*}
    \Pro \big( X_{1,\xi}^{(\tau)} =j, Y_{1 ,\xi}^{(\tau')}=j' \big) 
    &= 
    \sum_{l=0}^{j} \sum_{r=j-l}^{j'} p^{(\xi \tau)}(l)  p^{(\xi \tau')}(j'-r) \\
    &\hspace{3.5cm}  \times p_2^{((1-\xi)\tau', (1-\xi)\tau)}(r,j-l) , \\
    \Pro \big( X_{2,\xi} >x, Y_{2,\xi} >y \big) 
    &= 
    \exp \big( -\theta \{ (x \wedge y)\xi + (x \vee y) \} \big),\\
    \Pro \big( X_{3,\xi}^{(\tau)} =j, Y_{3,\xi} >x \big)  
    &= 
    e^{-\theta \xi x} \sum_{l=0}^{j}  p^{(\xi \tau)}(l) p_2^{((1-\xi)\tau, (1-\xi)x)}(j-l,0) \I(x \leq \tau) \\
    &\hspace{3.5cm} + e^{-\theta x} p^{(\tau \xi)}(j) \I(x > \tau).
    \end{align*}
\end{thm}
It is worthwhile to mention that $X_{\scs 1,\xi}^{\scs (\tau)}, Y_{\scs 1,\xi}^{\scs (\tau)}, X_{\scs 3,\xi}^{\scs (\tau)}$ are equal in distribution to $N_E^{\scs (\tau)}$  and that $X_{2,\xi},Y_{2,\xi},Y_{3,\xi}$ are exponentially distributed with parameter~$\theta$.

Regarding the estimator $\hat \pi^\mb_n(j)$ from \eqref{eq:pihat}, recall the definition of $v_{n,j}^\mb$ in \eqref{eq:vmn} and of $(s_1^\mb, \dots, s_m^\mb)$ and $\Sigma_m^\mb$ in Theorem~\ref{theo:djb} ($\mb = \djb$) or Theorem~\ref{theo:slb} ($\mb = \slb$).

\begin{cor} \label{Conv_vjn}
Let $\mb \in \{\djb, \slb\}$. Under the conditions of Theorem~\ref{theo:djb} ($\mb = \djb$) or Theorem \ref{theo:slb} ($\mb = \slb$) we have, for any $m \in \N_{\ge 1}$ and as $n \to \infty$,
	\[ 
	(v_{n,1}^\mb,\ldots,v_{n,m}^\mb) \wto (v_1^\mb,\ldots,v_m^\mb) \sim \Nor_{m}(0,\Gamma_m^\mb),\] 
	where  $v_1^\mb = 4 s_1^\mb$ and 
	\[ 
	v_j^\mb = 4s_j^\mb - 2 \sum_{k=1}^{j-1} \pi(j-k)s_k^\mb - 2 \sum_{k=1}^{j-1} \bar p(j-k) v_{k}^\mb, \quad j \ge 2.
	\]
	This recursion allows to write $(v_1^\mb,\ldots,v_m^\mb)^\top = A_m (s_1^\mb,\ldots,s_m^\mb)^\top$ for some matrix $A_m \in \R^{m\times m}$, such that the covariance matrix $\Gamma_m^\mb$ may be written as $\Gamma_m^\mb = A_m \Sigma_{m}^\mb A_m^\top$.
	\end{cor}

In the next theorem it will be shown that the asymptotic variances of the sliding blocks estimators are not larger than the asymptotic variances of their disjoint blocks counterparts. As a consequence, the sliding blocks estimators can be considered at least as efficient and should usually be preferred in practice.

\begin{thm} \label{Var_comp}
	For any $m \in \N$, we have
	\[ 
	\Sigma_m^{\slb} \leq_L \Sigma_m^{\djb} \quad \textrm{and} \quad 
	\Gamma_m^{\slb} \leq_L \Gamma_m^{\djb}, 
	\] 
	where $\leq_L$ denotes the Loewner-ordering between symmetric matrices. In particular, $\Var(s_j^{\slb}) \leq \Var(s_j^{\djb})$ and $\Var(v_j^{\slb}) \leq \Var(v_j^{\djb})$ for any $j\in \N_{\ge 1}$.
\end{thm}

\begin{ex}
	In the case that the time series is serially independent, a simple calculation yields $\pi(i)=\I(i=1)$ and $\pi_2^{(\sigma)}(i,j)=(1-\sigma)\I(i=1,j=0)+\sigma \I(i=1,j=1)$, which implies 
        \[ 
	p^{(\tau)}(1) = \tau e^{-\tau}, \quad 
	p_2^{(\tau',\tau)}(1,0)= (\tau'-\tau)e^{-\tau'},\quad 
	p_2^{(\tau',\tau)}(1,1)=\tau e^{-\tau'} 
	\]
	for $\tau' \geq \tau \geq 0, \tau'\ne 0$. Lengthy computations show that $d_{1,1}^{\djb}=5/108$, such that $\sigma^{2,\djb} =\Var(v_1^\djb)=20/27\approx 0.7407$. Likewise, $\sigma^{2,\slb} =\Var(v_1^\slb)\approx 0.3790$.
	The competing blocks estimator $\hat \pi_n^{\scs (\tau),\mathrm{Rob}}$ from \cite{Rob09} is known  to satisfy
	\[ 
	\sqrt{k_n} \big\{ \hat \pi_n^{(\tau),\mathrm{Rob}}(1) - \pi(1) \big\} \wto \Nor(0,\mu^2(\tau)), \quad \mu^2(\tau) = e^{\tau} (\tau + (1-\tau)^2-e^{-\tau}). 
	\]
	see Corollary 4.2 in that reference or p.\;3300 in \cite{Rob09b}.
	It is worth to mention that $\mu^2$ is strictly increasing with $\sigma^{2, \djb} < \mu^2(\tau)$ iff $\tau > 0.7573$.
\end{ex}

Recall that $\theta= \{ \sum_{j=1}^\infty j \pi(j) \}^{-1}$. As a consequence, following \cite{Hsi91} and \cite{Rob09}, the extremal index $\theta$ may be estimated by
\[
\hat \theta_n^{\mb}(m) = \Big\{ \sum_{j=1}^{m} j \hat \pi_n^{\mb}(j) \Big\}^{-1}, \quad \mb \in \{\djb, \slb\},
\]
for sufficiently large $m$. More precisely, $\hat \theta_n^{\mb}(m)$ should be considered an estimator for the partial sum approximation $\theta(m) =  \{ \sum_{j=1}^m j \pi(j) \}^{-1}$. The following result is an immediate consequence of Corollary~\ref{Conv_vjn}, see also Corollary~4.2 in \cite{Rob09}.

\begin{cor}
Under the conditions of Theorem~\ref{theo:djb} ($\mb = \djb$) or Theorem~\ref{theo:slb} ($\mb = \slb$) we have, for any $m \in \N$ and as $n \to \infty$,
	\[ 
	\sqrt{k_n} \big\{ \hat \theta_n^\mb(m) - \theta(m) \big\} \wto -\Big\{ \sum_{j=1}^{m} j \pi(j) \Big\}^{-2} \ \sum_{j=1}^{m} j v_j^\mb \sim \Nor(0,\sigma_\mb^2(m) ), 
	\] 
	where 
$
	\sigma_\mb^2(m)  = \big\{ \sum_{j=1}^{m} j \pi(j) \big\}^{-4} \ (1,\ldots,m) \Gamma_m^\mb (1,\ldots,m)^\top.
$
\end{cor}

\section{Finite-sample results} \label{sec:sim}

A simulation study was carried out to analyze the finite-sample performance of the introduced estimators and to compare them with estimators from the literature. Results are presented for the following three time series models which were also considered in \cite{Rob09} (with a slightly different ARMAX-model).
\begin{compactitem}%[leftmargin=.8cm]
	\item \textbf{ARMAX-model}:
	\[ X_s = \max \{ \alpha X_{s-1},(1-\alpha) Z_s \}, \qquad s \in \Z, \]
	where $\alpha \in [0,1)$ and $(Z_s)_s$ is an i.i.d. sequence of standard Fréchet random variables. We consider $\alpha = 0.5$ resulting in $\theta = 0.5$ and $\pi(1)=0.5$, $\pi(2)=0.25$, $\pi(3)=0.125$, $\pi(4)=0.0625$ and $\pi(5)=0.03125$ by \cite{Per94}.
	\item  \textbf{Squared ARCH-model}:
	\[
	X_s = (2 \times 10^{-5} + \lambda X_{s-1}) Z_s^2,  \qquad s \in \Z,
	\]
	where $\lambda\in(0,1)$  and where $(Z_s)_s$ denotes an i.i.d.\ sequence of standard normal random variables. We consider $\lambda = 0.5$, for which the simulated values $\theta = 0.727$ and $\pi(1)=0.751$, $\pi(2)=0.168$, $\pi(3)=0.055$, $\pi(4)=0.014$ and $\pi(5)=0.008$ were obtained in \cite{DehResRooVri89}.
	\item \textbf{AR-model}:
	\[ X_s = r^{-1} X_{s-1} + Z_s, \qquad s \in \Z, \]
	where $(Z_s)_s$ is an i.i.d.\ sequence of random variables that  are uniformly distributed on $\{0,1/r,\ldots,(r-1)/r\}$. We consider $r=4$, for which the simulated values $\theta = 0.75$ and $\pi(1)=0.75$, $\pi(2)=0.1875$, $\pi(3)=0.0469$, $\pi(4)=0.0117$ and $\pi(5)=0.0029$ were obtained in \cite{Per94}.
\end{compactitem}

In all scenarios the sample size was fixed to $n=2\,000$ and the block size $b$ was chosen from the set $\{6,8,\ldots,36,38\}$. All results are based on $N=500$ simulation runs each. 

For completeness, and inspired by \cite{Nor15}, a slight modification of the estimators from Section~\ref{sec:math} has been considered as well. For its motivation, note that $X_s > M_{ni}^{\mb}$ iff $\hat{F}_n(X_s) > 1-\hat{Z}^{\mb}_{ni}/b_n$ (a.s.), where $\hat{Z}^{\mb}_{ni}=b_n \{ 1-\hat{F}_n(M_{ni}^{\mb}) \}$ with the empirical c.d.f.\ $\hat F_n(x)=n^{-1} \sum_{i=1}^n \ind(X_i \le x)$. For large block size $b_n$, we further have $\hat{Z}^{\mb}_{ni} \approx \hat Y_{ni}^\mb = -b_n \log \hat{F}_n(M^{\mb}_{ni})$, which suggests to define
\begin{align*} 
\pbh_n^{\,y,\djb} (m)
&= 
\frac{1}{k_n(k_n-1)} \sum_{\substack{i,i'=1 \\ i \neq i'}}^{k_n} \I \bigg\{ \sum_{s \in I_{i'}} \I\big( \hat{F}_n(X_s) > 1-\hat{Y}^{\djb}_{ni}/b_n \big)=m \bigg\}, \\
\pbh_n^{\,y,\slb}(m)
&= 
\frac{1}{|D_n|} \sum_{(i,i') \in D_n} \I \bigg\{ \sum_{s \in I_{i'}^{\slb}} \I \Big( \hat{F}_n(X_s) > 1-\hat{Y}^{\slb}_{ni}/b_n \Big) =m \bigg\}.
\end{align*} 
Finally, let $\hat \pi_n^{\, y, \mb}$ be defined in terms of $\pbh_n^{\,y,\mb}$ as in (\ref{eq:pihat}). For the ease of a unified notation, the estimators from Section~\ref{sec:math} will subsequently be denoted by $\pbh_n^{\,z,\mb}$ and $\hat \pi_n^{\, z, \mb}$.

%%%%%%%%%%%%%%%%%%%%%%%%%%%%%%%%%%%%%%%%%%

\subsection{Comparison of the introduced estimators for $\pi$}

In this section we compare the finite-sample performance of the introduced four estimators $\hat \pi_n^{z,\djb}$, $\hat \pi_n^{z,\slb}$, $\hat \pi_n^{y,\djb}$ and $\hat \pi_n^{y,\slb}$. 

We start with a detailed analysis of the variance, bias and mean squared error (MSE) as a function of the block size parameter $b$. Results are only reported for the squared ARCH-model;  the corresponding figures for the ARMAX- and AR-model show roughly the same qualitative behavior and can be found in Appendix \ref{more_figures}.
The variance is depicted in Figure \ref{Fig:sqARCH_Var}, which can be seen to  be increasing in the block size for all estimators. It is further apparent that the $Z$- and $Y$-versions behave nearly identical, whereas the variance of the sliding blocks estimators is considerably smaller than for the disjoint blocks estimators, uniformly over all block sizes. For the $Z$-version, this is in accordance with the theoretical result from Theorem \ref{Var_comp}. 

The bias is presented in Figure \ref{Fig:sqARCH_Bias} and can be seen to be either  increasing or  decreasing in $b$. The largest absolute value of the bias is mostly attained for small block sizes. The bias curves for the sliding blocks estimators are smoother than for the disjoint blocks versions, which may be explained by the fact that no observations have to be discarded when $b$ is not a divisor of $n$. One can further see that the $Y$-versions exhibit a substantially smaller absolute bias for small block sizes (except for $m=2$); an observation that has also been made in \cite{BucJen20}.
However, we observe that neither of our estimators can be said to be overall superior with regard to the smallest bias. 

The mean squared error is outlined in Figure \ref{Fig:sqARCH_MSE}. In many cases, the MSE-curves show a similar behavior as the variance-curves for large block sizes, since there the variance is dominating over the squared bias. Likewise, the large squared bias for small block sizes can be identified in the MSE-curves as well, eventually resulting in a typical  u-shape. Again, the $Y$-versions perform better for small block sizes (except for $m=2$). Moreover, the sliding blocks estimators outperform the disjoint blocks estimators with regard to the MSE. Since this qualitative behavior holds uniformly over all models under consideration, we omit the disjoint blocks estimators in the subsequent discussion. 

\begin{figure} [t!]
	\begin{center}
		\includegraphics[width=\textwidth]{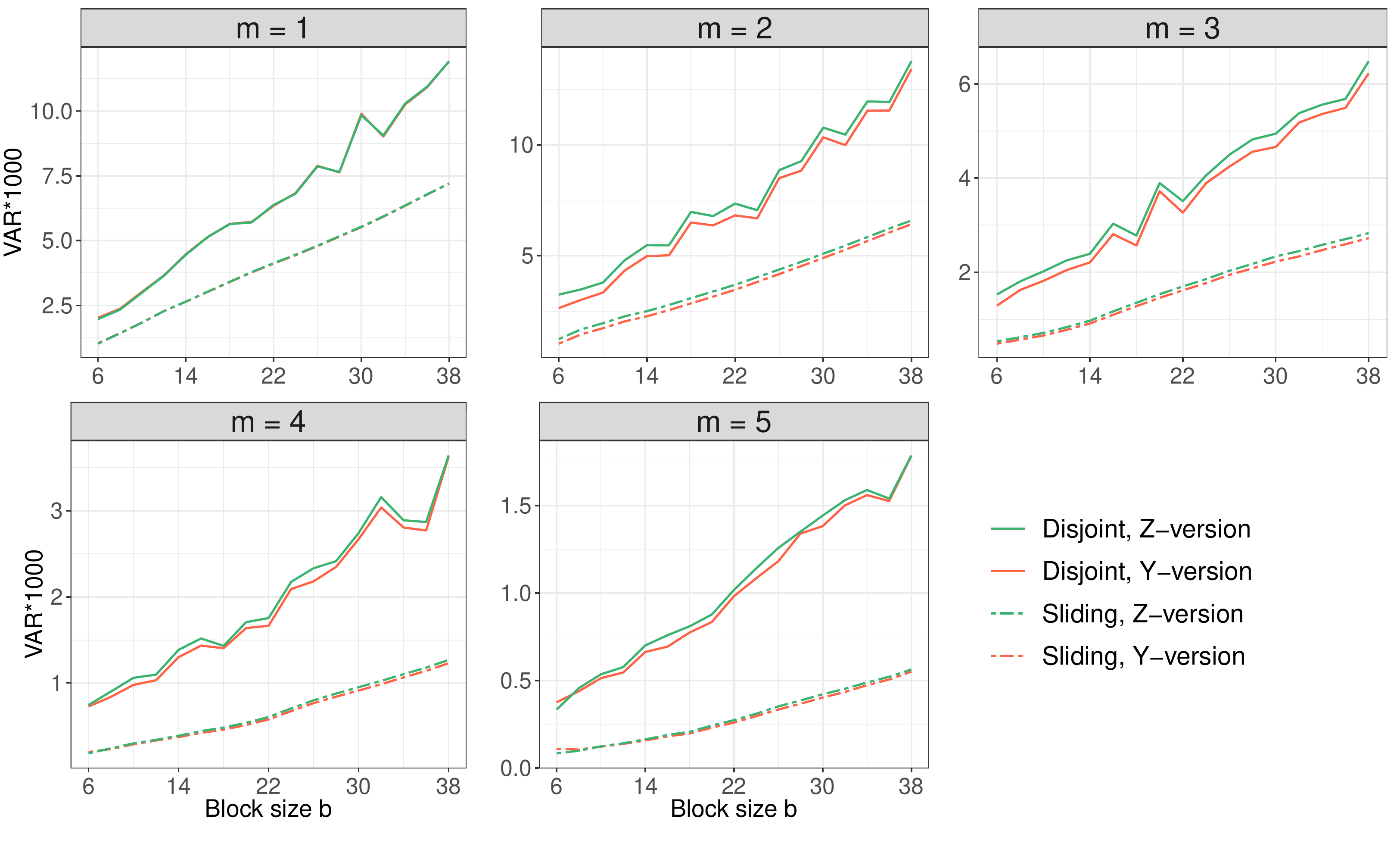} \vspace{-.8cm}
		\caption{Variance multiplied by $10^3$ for the estimation of $\pi(m)$ within the squared ARCH-model for $m=1,\ldots,5$.} 
		\vspace{-.3cm} 
		\label{Fig:sqARCH_Var}
	\end{center}
\end{figure}

\begin{figure} [t!]
	\begin{center}
		\includegraphics[width=\textwidth]{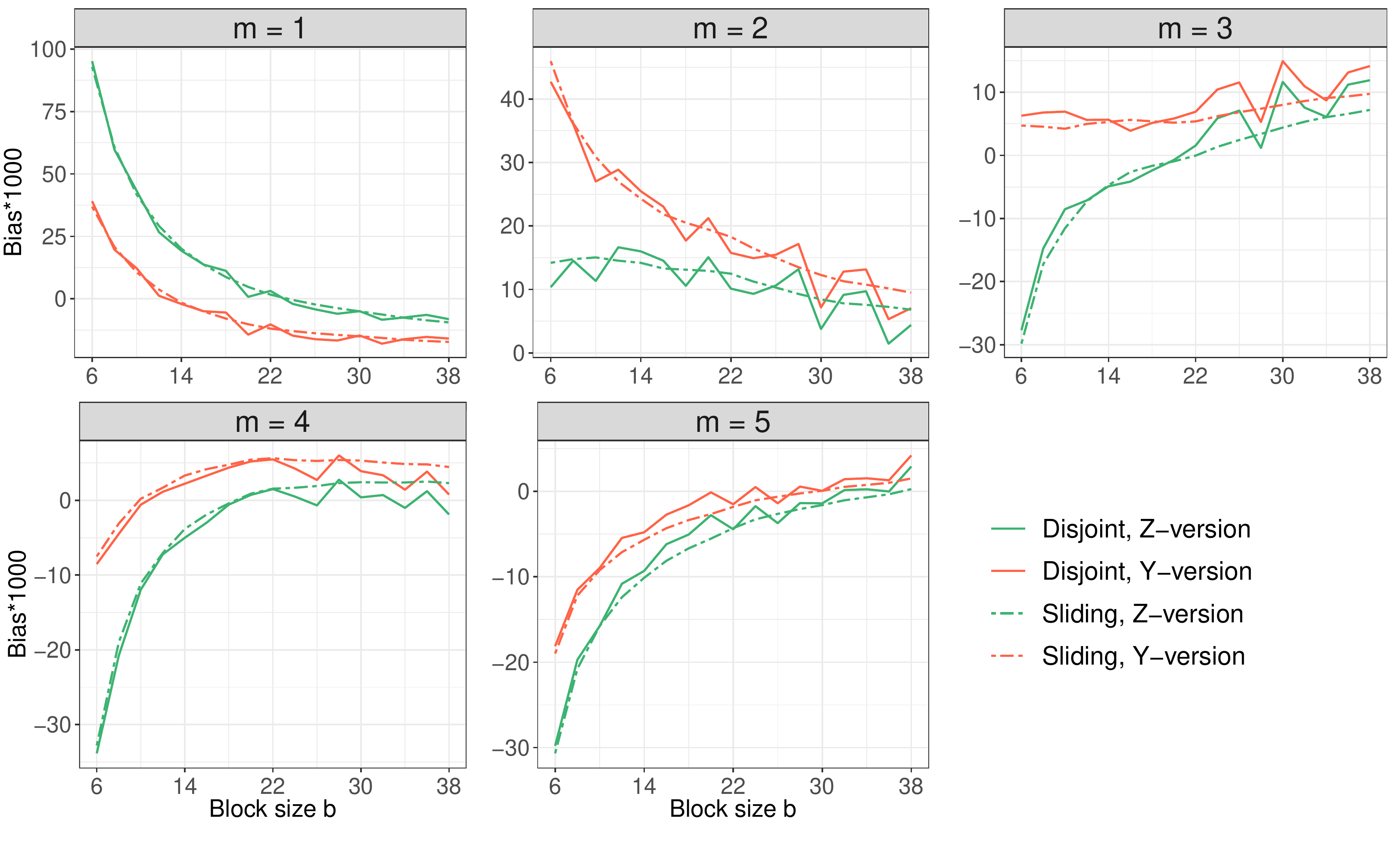} \vspace{-.8cm}
		\caption{Bias multiplied by $10^3$ for the estimation of $\pi(m)$ within the squared ARCH-model for $m=1,\ldots,5$.} 
		\vspace{-.3cm} 
		\label{Fig:sqARCH_Bias}
	\end{center}
\end{figure}

\begin{figure} [t!]
	\begin{center}
		\includegraphics[width=\textwidth]{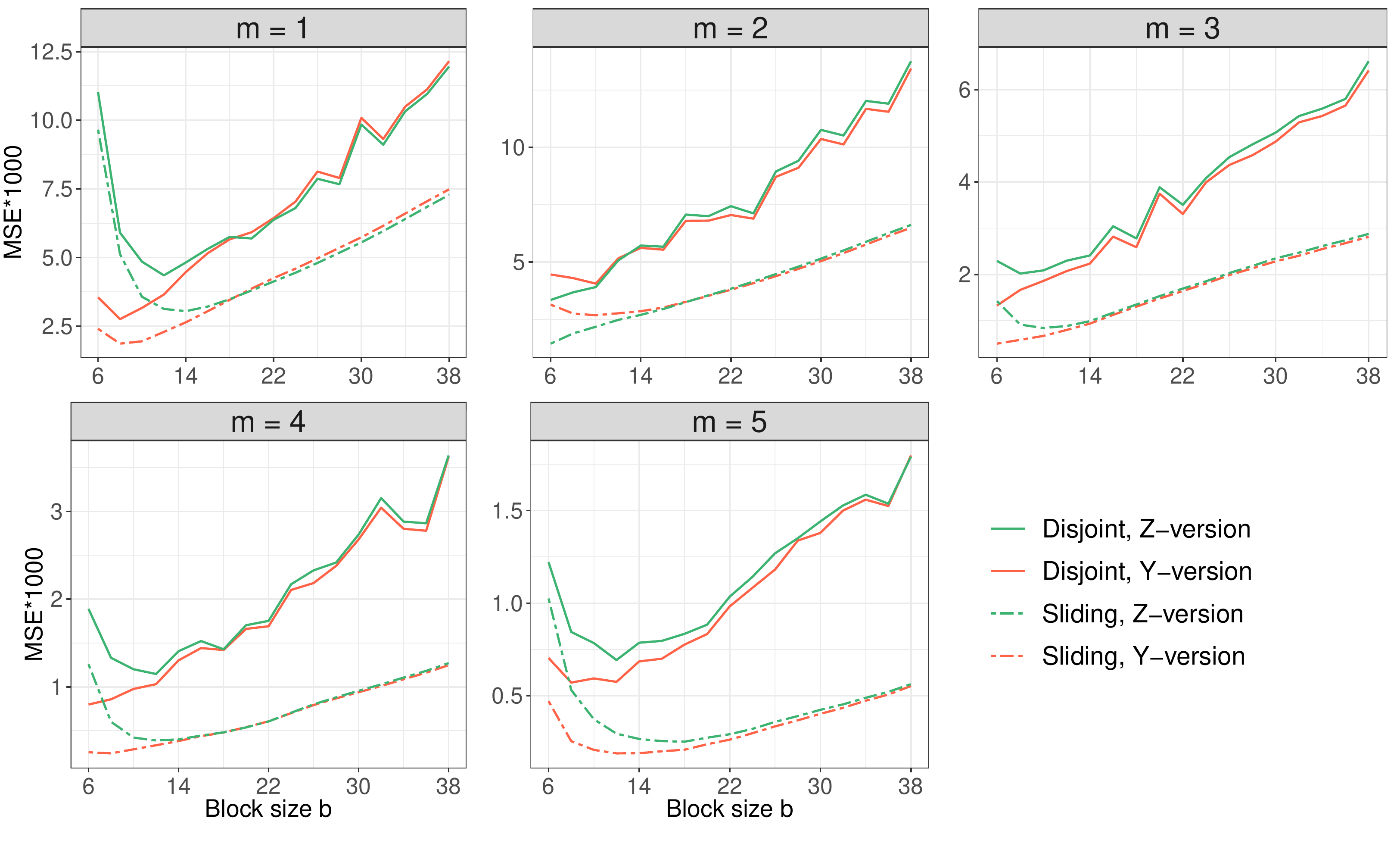} \vspace{-.8cm}
		\caption{Mean squared error multiplied by $10^3$ for the estimation of $\pi(m)$ within the squared ARCH-model for $m=1,\ldots,5$.} 
		\vspace{-.3cm} 
		\label{Fig:sqARCH_MSE}
	\end{center}
\end{figure}

%%%%%%%%%%%%%%%%%%%%%%%%%%%%%%%%%%%%%%%%%%

\subsection{Comparison with competing estimators for $\pi$}

In this section, we compare the performance of our sliding blocks estimators for $\pi(m)$ with the following competitors from the literature: the integrated version of the blocks estimator from \cite{Rob09} with parameters $\sigma = 0.7$ and $\phi=1.3$ (page 276 in that reference), the blocks estimator from \cite{Hsi91} with $v_n=X_{n-\lfloor n/s_n \rfloor:n}$, where $s_n=2(b_n-3)$ (see (1.4) in \citealp{Hsi91} and (1.2) in \citealp{Rob09}, where a similar same choice has been made), and the inter-exceedance times estimator from \cite{Fer03} with $N=3k_n$ (see equation (4.12) in that reference).

In Figure \ref{Fig:sqARCH_MSE_all}, the MSE is plotted as a function of the blocksize in the squared ARCH-model (see Appendix~\ref{more_figures} for other models). We can see that the MSE is mostly decreasing for small blocksizes and tends to increase from an intermediate blocksize onwards, which is due to the common bias-variance-tradeoff. The MSE-curves of our sliding blocks estimators are very smooth compared to the competing estimators and lie uniformly below their MSE-curves in many cases. Generally, the estimator by Robert and our sliding blocks estimators outperform the estimators by Ferro and Hsing in almost all scenarios under consideration. 

The minimum values of the mean squared error (minimum over $b$) are of particular interest. They are presented for all models under consideration in Table \ref{minMSE_table}. The estimator $\hat \pi^{z,\slb}_n$ wins twice, $\hat \pi^{y,\slb}_n$ wins seven times and Robert's estimator five times, while the estimators by Ferro wins once. It is worth to mention that the sliding blocks estimators cover all minimum values within the ARMAX-model, and Robert's estimator seems to perform especially well for large values of $m$. 

\begin{figure} [t!]
	\begin{center}
		\includegraphics[width=\textwidth]{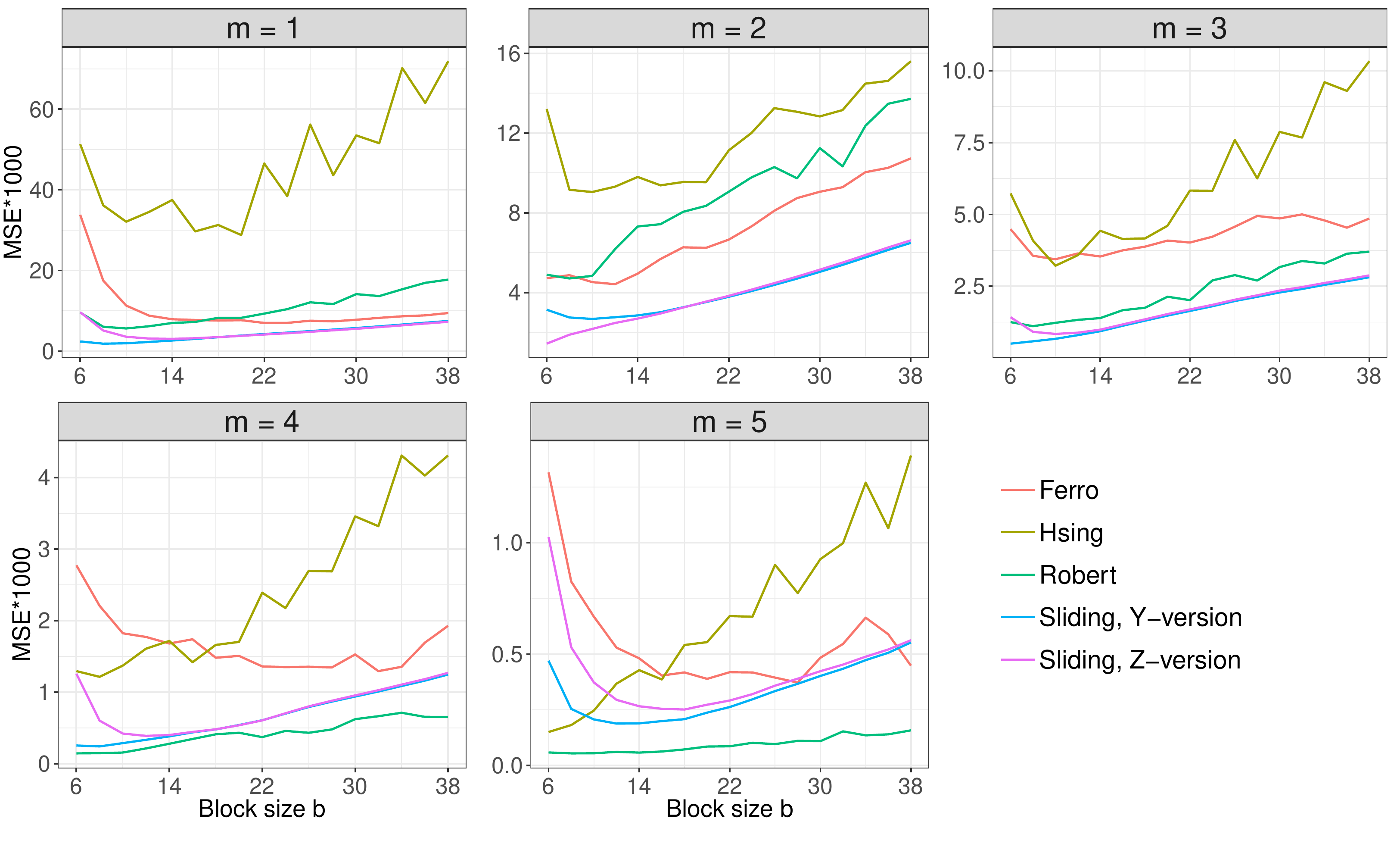} \vspace{-.8cm}
		\caption{Mean squared error multiplied by $10^3$ for the estimation of $\pi(m)$ within the squared ARCH-model for $m=1,\ldots,5$.} 
		\label{Fig:sqARCH_MSE_all}
	\end{center}
\end{figure}

\begin{table}[ht]
	\centering
		\bgroup
		\def\arraystretch{1.2}
	\begin{tabular}{crrrrrrr}
		\hline \hline
		Model & m & $\pi(m)$ & Sliding, Z & Sliding, Y & Robert & Hsing & Ferro \\
		\hline
		AR &    1 & 0.750 & 8.255 & 6.746 & 12.951 & 20.094 & \bf{4.007} \\
		&    2 & 0.188 & 2.374 & \bf{1.636} & 8.732 & 7.352 & 3.683 \\
		&    3 & 0.047 & 1.679 & 1.301 & \bf{1.090} & 2.864 & 1.423 \\
		&    4 & 0.012 & 0.861 & 0.497 & \bf{0.113} & 0.277 & 0.236 \\
		&    5 & 0.003 & 0.159 & 0.088 & \bf{0.008} & 0.017 & 0.035 \\
		ARMAX &    1 & 0.500 & 2.642 & \bf{1.650} & 6.819 & 5.318 & 5.343 \\
		&    2 & 0.250 & 0.495 & \bf{0.434} & 2.177 & 1.586 & 3.460 \\
		&    3 & 0.125 & \bf{0.186} & 0.311 & 1.763 & 1.816 & 2.118 \\
		&    4 & 0.062 & 0.252 & \bf{0.179} & 1.144 & 1.011 & 2.454 \\
		&    5 & 0.031 & 0.206 & \bf{0.086} & 0.474 & 0.390 & 2.350 \\
		sqARCH &    1 & 0.751 & 3.044 & \bf{1.860} & 5.631 & 28.795 & 7.001 \\
		&    2 & 0.168 & \bf{1.436} & 2.677 & 4.706 & 9.043 & 4.418 \\
		&    3 & 0.055 & 0.842 & \bf{0.503} & 1.111 & 3.214 & 3.439 \\
		&    4 & 0.014 & 0.389 & 0.242 & \bf{0.145} & 1.215 & 1.294 \\
		&    5 & 0.008 & 0.251 & 0.188 & \bf{0.055} & 0.150 & 0.372 \\
		\hline \hline
	\end{tabular}
	\egroup
	\caption{Minimal mean squared error multiplied by $10^3$ for the AR-model, the maxAR-model and the squared ARCH-model. The estimator with the row-wise smallest MSE is in boldface.}
	\label{minMSE_table}
\end{table}

\section{Proofs of the main results} \label{sec:proofs}

We start by arguing that we may slightly redefine the estimators, which will greatly simplify the notational complexity. For $m\in\N_{\ge 0}$, let
\begin{align*} 
 \tilde{\bar{p}}_n^{\,\djb}(m)
 &= \frac{1}{k_n^2} \sum_{i,i'=1}^{k_n} \I \bigg\{ \sum_{s \in I_{i'}^\djb} \I\big(X_s > M_{ni}^\djb \big)=m \bigg\}, \\
\tilde{\bar p}_n^{\,\slb}(m)
&= \frac{1}{(n-b_n+1)^2} \sum_{i,i'=1}^{n-b_n+1} \I \bigg\{ \sum_{s \in I_{i'}^{\slb}} \I \big(X_s >  M_{ni}^\slb \big) =m \bigg\}.
\end{align*}
Since $ \sum_{s \in I_{i}^\djb} \I\big(X_s > M_{ni}^\djb \big) = \I(m=0)$ and $|\tilde{\bar{p}}_n^{\, \djb}| \leq 1$, we have, for $m\ge 1$,
	\begin{align*}
		\tilde{\bar{p}}^{\,\djb}_n(m) - \hat{\bar{p}}^{\,\djb}_n(m) 
		&= \Big( 1- \frac{k_n}{k_n-1} \Big) \tilde{\bar{p}}^{\,\djb}_n(m) 
		= o_\Prob(k_n^{-1/2}).
	\end{align*}
As a consequence, throughout the proof, we may redefine $\hat{\bar{p}}^{\,\djb}_n(m) = \tilde{\bar{p}}^{\,\djb}_n(m)$. A similar argument holds for the sliding blocks version, whence we subsequently set $\hat{\bar{p}}^{\,\slb}_n(m) = \tilde{\bar{p}}^{\,\slb}_n(m)$.

Next, we will introduce some additional notation. For $s\in\Z$, let $U_s=F(X_s)$.
For $\tau>0$ and $m\in\N_{\ge 0}$, let
\begin{align*}
p_{n}^{(\tau), \djb} (m) 
&=
\frac1{k_n} \sum_{i=1}^{k_n} \I\big( N_{b_n,i}^{(\tau), \djb}=m\big),  \\
p_{n}^{(\tau), \slb} (m) 
&=
\frac1{n-b_n+1} \sum_{i=1}^{n-b_n+1} \I\big( N_{b_n,i}^{(\tau), \slb}=m\big), 
\end{align*}
where, for $\mb \in \{\djb, \slb\}$,
\begin{align*}
N_{b_n,i}^{(\tau), \mb} 
&=  \sum\nolimits_{s \in I_{i}^{\mb}} \I \Big(U_s > 1-\frac{\tau}{b_n}  \Big).
\end{align*}
Denote the rescaled estimation error by
\begin{align} \label{eq:enj}
e_{n,m}^\mb(\tau) 
= 
\sqrt{k_n} \Big\{ p_n^{(\tau), \mb}(m) - \varphi_{n,m}(\tau) \Big\}, 
\end{align}
where $\varphi_{n,m}$ is defined in Condition~\ref{cond}(iv).
Note that the disjoint blocks version $e_{n,m}^\djb$ has been extensively studied in \cite{Rob09}. Next, let $Z_{ni}^\mb=b_n\{1-F(M_{ni}^\mb)\}$ and, for $x>0$, let 
\[
\hat H_{n}^\djb(x) = \frac{1}{k_n} \sum_{i=1}^{k_n} \I(Z_{ni}^\djb \leq x),
\quad
\hat H_{n}^\slb(x) = \frac{1}{n-b_n+1} \sum_{i=1}^{n-b_n+1} \I(Z_{ni}^\slb \leq x),
\] 
denote the empirical c.d.f.\ of $Z_{n1}^{\djb},\ldots,Z_{nk_n}^\djb$ and $Z_{n1}^{\slb},\ldots,Z_{n,n-b_n+1}^\slb$, respectively. Finally, recall $H(x) = (1-e^{-\theta x})\I(x\geq 0)$, the c.d.f.\ of the exponential distribution with parameter $\theta$. 

\ \\
\textbf{Proof of Theorem \ref{theo:djb}}. 
By continuity of $F$, we have $U_s > 1-Z_{ni}^{\djb}/b_n$ iff $X_s > M_{ni}^{\djb}$ almost surely, whence we may write,  for $j\in\N_{\ge 1}$,
\[
\hat{\bar{p}}_n^{\,\djb}(j) \stackrel{a.s.}{=} k_n^{-1} \sum_{i=1}^{k_n} p_n^{(Z_{ni}^\djb), \djb}(j).
\]
We may thus decompose
\begin{align} \label{dec_sjn}
	s_{n,j}^{\djb} = \sqrt{k_n} \{ \hat{\bar{p}}^{\,\djb}_n(j) - \bar{p}(j) \}  
	\stackrel{a.s.}{=} 
	A_{n1} + A_{n2} + A_{n3},
\end{align}
where
\begin{align*}
	A_{n1}
	&= 
	\frac{1}{\sqrt{k_n}} \sum_{i=1}^{k_n} \bigg\{ \intnu \I\big(N_{b_n,i}^{(\tau), \djb}=j\big) - \varphi_{n,j}(\tau) \ \mathrm{d}H(\tau) \\
	& \hspace{5cm}+ \varphi_{n,j}(Z_{ni}^\djb) -  \Exp[\varphi_{n,j}(Z_{n1}^\djb)] \bigg\}, \\
	A_{n2}
	&= 
	\intnu e_{n,j}^\djb(\tau) \ \mathrm{d}(\hat H_{n}^\djb-H)(\tau),  \qquad
	A_{n3} 
	=
	\sqrt{k_n} \big\{ \Exp[\varphi_{n,j}(Z_{n1}^\djb)] - \bar p(j) \big\}.
\end{align*}
We have $A_{n3}=o(1)$ by Condition \ref{cond}(v) and $A_{n2}=o_{\Prob}(1)$ by Lemma~ \ref{dj_ejn_integral}. Hence, setting 
\begin{multline} \label{eq:wni}
	W_{n,i}^\djb(j) 
	= 
	\intnu \I\big(N_{b_n,i}^{(\tau), \djb}=j\big) - \varphi_{n,j}(\tau)\, \mathrm{d}H(\tau) \\
	+ \varphi_{n,j}(Z_{ni}^\djb) -  \Exp[\varphi_{n,j}(Z_{ni}^\djb)],
\end{multline}
we have $s_{n,j}^\djb = k_n^{-1/2} \sum_{i=1}^{k_n} W_{n,i}^\djb(j) + \op$. The assertion  then follows from 
\[ 
\frac{1}{\sqrt{k_n}} \sum_{i=1}^{k_n} \big(W_{n,i}^\djb(1), \ldots, W_{n,i}^\djb(m) \big) 
\wto 
\Nor_{m}(0,\Sigma_{m}^\djb) 
\] as a consequence of Lemma \ref{dj_weakconv}.
\qed

\ \\
\textbf{Proof of Theorem \ref{theo:slb}}. As in the proof of Theorem~\ref{theo:djb}, we have
\[
	\hat{\bar{p}}_n^{\,\slb}(j) \stackrel{a.s.}= \frac{1}{n-b_n+1} \sum_{i=1}^{n-b_n+1} p_n^{(Z_{ni}^{\slb}),\slb}(j).
\]
Similarly as in (\ref{dec_sjn}) and by using the bias Condition \ref{cond}(vi), we can thus write
\begin{align*}
s_{n,j}^{\slb} &\,=\,  \sqrt{k_n} \big\{ \hat{\bar p}_n^{\,\slb}(j) - \bar p(j) \big\} \\
 &\stackrel{a.s.}= 
 \frac{\sqrt{k_n}}{n-b_n+1} \sum_{i=1}^{n-b_n+1} W_{n,i}^{\slb}(j) + \intnu e_{n,j}^{\slb}(\tau) \ \mathrm{d} (\hhutsl -H)(\tau) + o(1), 
\end{align*}
where $W_{n,i}^{\slb}$ is defined as in \eqref{eq:wni}, but with `$\djb$' replaced by `$\slb$' everywhere. The assertion then follows from $\intnu e_{n,j}^{\slb} \ \mathrm{d} (\hat H_n^{\slb} -H)=o_\Prob(1)$ by Lemma \ref{sl_ejn_integral} and 
 \[ 
 \frac{\sqrt{k_n}}{n-b_n+1} \sum_{i=1}^{n-b_n+1} \big(W^{\slb}_{n,i}(1), \ldots, W^{\slb}_{n,i}(m)\big)
 \wto \Nor_{m}(0,\Sigma_{m}^{\slb})  
 \] by Lemma \ref{sl_weakconv}.
 \qed

\ \\
\textbf{Proof of Theorem \ref{Conv_vjn}}. Throughout, we omit the index $\mb\in\{\djb, \slb\}$. For $j\in\N_{\ge1}$, set $\varphi_j : \R^{2j-1} \to \R, \varphi_j(x) = 4x_j-2 \sum_{k=1}^{j-1} x_{2j-k} x_k$, such that
 \begin{align*}
  \hat \pi_n(j) &= \varphi_j(\hat{\bar p}_n(1), \ldots, \hat{\bar p}_n(j), \hat \pi_n(1),\ldots, \hat \pi_n(j-1)), \\
\pi(j) &= \varphi_j(\bar p(1), \ldots, \bar p(j), \pi(1),\ldots, \pi(j-1)).
 \end{align*}
 By Theorem \ref{theo:djb} and \ref{theo:slb}, we know that $(s_{n,1},\ldots,s_{n,m}) \wto (s_1,\ldots,s_m) \sim \Nor_m(0,\Sigma_m)$. To prove the theorem, we use this result and apply  induction over $m$. First, 
 \[
 v_{n,1} = \sqrt{k_n} \big\{ \hat \pi_n(1) - \pi(1) \big\} = 4 \sqrt{k_n} \big\{ \hat{\bar p}_n(1) - \bar p(1) \big\} = 4s_{n,1},
 \] 
 such that $(s_{n,1},s_{2,n},v_{n,1}) \wto (s_1,s_2,4s_1) = (s_1, s_2, v_1)$. Second, assume we have 
 \begin{equation*}
  (s_{n,1},\ldots,s_{n,m},v_{n,1},\ldots,v_{m-1,n}) \wto (s_1,\ldots,s_m,v_1,\ldots,v_{m-1})
 \end{equation*}
 for $m \geq 2$. Then, the delta-method implies
  \begin{align*}
  v_{n,m} 
  &= \sqrt{k_n} \big\{ \varphi_m(\hat{\bar p}_n(1), \ldots, \hat{\bar p}_n(m), \hat \pi_n(1),\ldots, \hat \pi_n(m-1)) - \nonumber \\ 
  & \hspace{4cm} \varphi_m(\bar p(1), \ldots, \bar p(m), \pi(1),\ldots, \pi(m-1)) \big\} \nonumber \\
  &= \varphi_m'(\bar p(1), \ldots, \bar p(m), \pi(1),\ldots, \pi(m-1)) \\
  & \hspace{4cm} \cdot (s_{n,1},\ldots,s_{n,m},v_{n,1},\ldots,v_{n,m-1})^\top + o_\Prob(1) \\
  &\wto  4s_m - 2 \sum_{k=1}^{m-1} \pi(m-k)s_k - 2 \sum_{k=1}^{m-1} \bar p(m-k) v_{k}=: v_m, \nonumber
 \end{align*}
where $\varphi_m'$ denotes the gradient of $\varphi_m$. We obtain that 
\[
(s_{n,1}, \ldots, s_{n,m},v_{n,1},\ldots,v_{n,m}) \wto (s_1,\ldots,s_m,v_1,\ldots,v_m). 
\]
Since every $v_j$ is a linear function of $(s_1,\ldots,s_m) \sim \Nor_m(0,\Sigma_m)$, the vector $(v_1,\ldots,v_m)$ follows an $m$-dimensional normal distribution as well.
\qed

%%%%%%%%%%%%%%%%%%%%%%%%%%%%%%%%%%%%%%%%%%%%%%%%%%%%%%%%%%%%%%%%%%%%%%%%%%%%%%%%%%%%%%%%%

%%%%%%%%%%%%%%%%%%%%%%%%%%%%%%%%%%%%%%%%%%%%%%%%%%%%%%%%%%%%%%%%%%%%%%%%%%%%%%%%%%%%%%%%%%%%%%%%%%%%

\ \\
\textbf{Proof of Theorem \ref{Var_comp}}. 
We only need to prove $\Sigma_m^{\slb} \leq_L \Sigma_m^\djb$; the assertion regarding $\Gamma_m^{\mb}$ is an immediate consequence.

In the following, we assume for simplicity that $U_s$ and $Z_{ni}^{\slb}$ are measurable with respect to the $\mathcal{B}^{\varepsilon}_{\cdot : \cdot}$-sigma fields; the general case can be treated by multiplication with suitable indicator functions as in the proofs in the appendices. 
Now, $\Sigma_m^{\slb} \leq_L \Sigma_m^\djb$ is equivalent to
\begin{align}
\Var \Big( \sum_{j=1}^{m} a_j s_j^{\slb} \Big) \leq \Var \Big( \sum_{j=1}^{m} a_j s_j^\djb \Big) \label{var_sum}
\end{align}
for any $a=(a_1,\ldots,a_m)^\top \in \R^m$. To prove the latter, we are going to apply Lemma A.10 in \cite{ZouVolBuc19}.
For $j \in\{1,\ldots, m\}$ and $i\in\N_{\ge 1}$, let $S_{n,i} = \sum_{j=1}^{m} a_j V_{n,i}(j)$, where

\begin{align*}
V_{n,i}(j) = \intnu \I \Big( \sum_{s \in J_i} \I \Big(U_s>1-\frac{\tau}{b_n}\Big) = j \Big) \, \di H(\tau) + \varphi_{n,j}\Big(b_n ( 1- \max_{s \in J_i} U_s)\Big) 
\end{align*}
and where $J_i=\{i, i+1, \dots, i+b_n-1\}$.
Note that $I_i^\djb= J_{(i-1)b_n+1}$ for $i\in\{1, \dots, k_n\}$ and that $I_i^{\slb}=J_i$ for $i\in\{1, \dots, n-b_n+1\}$.
By the proofs of Theorem \ref{theo:djb} and \ref{theo:slb} we can write 
\begin{align*}
	\Var \Big( \sum_{j=1}^{m} a_j s_j^{\slb} \Big) &= \lim_{n \to \infty	} \Var \Big( \sqrt{\frac{n}{b_n}} \frac{1}{n} \sum_{i=1}^{n} S_{n,i} \Big), \\
	\Var \Big( \sum_{j=1}^{m} a_j s_j^\djb \Big) &= \lim_{n \to \infty	} \Var \Big( \sqrt{\frac{n}{b_n}}  \frac{b_n}{n} \sum_{i = 1}^{ \lfloor n/b_n \rfloor} S_{n,(i-1)b_n+1} \Big).
\end{align*}
For $ h \in \N_{\ge 0}$, set $\gamma_n(h) = \Cov(S_{n,1}, S_{n,h+1})$; note that $S_{n,1},\ldots, S_{n,n-b_n+1}$ is stationary. 
We obtain 
\begin{align*}
	|\gamma_n(h)| 
	& \leq 
	\sum_{j,j'=1}^{m} |a_j a_{j'}| \big| \Cov (V_{n,1}(j), V_{n,h+1}(j')) \big| \leq 8 \sum_{j,j'=1}^{m} |a_j a_{j'}|,
\end{align*}
such that $\sup_{n \in \N, h \in \N_{\ge 0}} |\gamma_n(h)| < \infty$.
Further, by Lemma 3.9 in \cite{DehPhi02} we have 
\begin{align*}
	|\gamma_n(h+b_n)| & \leq 4 \sum_{j,j'=1}^{m} |a_j a_{j'}| \ ||V_{n,1}(i)||_{\infty} \ ||V_{n,1}(j)||_{\infty} \ \alpha_{\eps_1}(1+h) \lesssim
	 \alpha_{\eps_1}(h)
\end{align*}
with $\varepsilon$ from Condition~\ref{cond}. This implies 
\[ 
\sum_{h=1}^{\infty} |\gamma_n(h+b_n)| 
\lesssim \sum_{h=1}^{\infty} \alpha_{\eps_1}(h) 
\lesssim \sum_{h=1}^{\infty} h^{-\eta} < \infty
\] 
by Condition \ref{cond}(iii).
Relation (\ref{var_sum}) then follows from Lemma A.10 in \cite{ZouVolBuc19}.
\qed

\section{On sliding blocks estimators for compound probabilities} \label{sec:comp}

Throughout this section, we derive an extension of Theorem 4.1 in \cite{Rob09} from the disjoint blocks process $e_{n,m}^{\scs \djb}$ in \eqref{eq:enj} to the sliding blocks version $e_{n,m}^{\scs \slb}$. The result is used for proving Theorem~\ref{theo:slb}, but might in fact be of general interest for statistics for  time series extremes based on sliding blocks. For $m\in\N_{\ge 0}$ and $\tau \geq 0$, let
\[ 
E_{n,m}^{\slb}(\tau) = \big( e_{n,0}^{\slb}(\tau), \ldots, e_{n,m}^{\slb}(\tau)\big).
\] 
For simplicity, we impose the same mixing conditions as needed for the results in Section~\ref{sec:main}.

\begin{thm} \label{weak_conv}
	Suppose that Condition~\ref{cond}(i)--(ii)  is met  and that, additionally, $\sqrt{k_n} \beta_{\varepsilon_2}(b_n) =o(1)$ for some $\eps_2>0$. Then,  for any $m \in \N_{\ge 1}$,
	\begin{align*}
	E_{n,m}^{\slb} \wto E_m^{\slb} \quad \textrm{in} \quad D([0,\infty))^{m+1},
	\end{align*}
	where $E_m^{\slb}(\cdot)=\big( e_{0}^{\slb}(\cdot), \ldots, e_{m}^{\slb}(\cdot)\big)$ is a centered Gaussian process with continuous sample paths, almost surely, and with covariance functional given by, for $0 \leq \tau \leq \tau'$ and $j,j'\in\{0 , \dots,  m\}$,
	\begin{align*}
    \Cov \big( e_{j}^{\slb}(\tau), e_{j'}^{\slb}(\tau') \big)
    & =  2 \intne \Cov \big( \I(X_{\xi}^{(\tau)} = j), \I(Y_{\xi}^{(\tau')} = j') \big) \, \di \xi \\
    & =  2 \intne H_{j,j'}^{(\tau,\tau')}(\xi) \, \di \xi - 2 p^{(\tau)}(j) p^{(\tau')}(j'),
    \end{align*}
%    where, for $\xi\in[0,1]$, 
    where  $X_{\xi}^{(\tau)} = Y_{\xi}^{(\tau)} = N_E^{(\tau)}$ in distribution with joint probability mass function
    \begin{align*}
    H_{j,j'}^{(\tau,\tau')}(\xi)
    &= \Pro \big( X_{\xi}^{(\tau)}  = j, Y_{\xi}^{(\tau')} =  j' \big)  \\
    &= \! \sum_{l=0}^j \sum_{r=j-l}^{j '}p^{(\xi\tau)}(l) p^{(\xi\tau')}(j'-r) p_2^{((1-\xi)\tau',(1-\xi)\tau)}(r,j-l). 
    \end{align*}
\end{thm}

\begin{proof}
The result is a consequence of the next two lemmas.
\end{proof}

It is worthwhile to mention that one may add the classical tail empirical process $\bar e_n$ as an $(m+2)$th-coordinate to $E_{n,m}^\slb$ (just as in Theorem 4.2 in \citealp{Rob09}). Additional conditions as in that reference would be necessary then, including a moment bound on the increments of $\tau \mapsto N_n^{\scs (\tau)}$ and adapted mixing conditions. Details are omitted for the sake of brevity.

Further, it is interesting to note that in specific cases the asymptotic variance of the sliding blocks process can be seen to be smaller than that of its disjoint blocks counterpart. For instance, some tedious but straightforward calculations show that, for $\tau =1$,
\[ 
\Var(e_1^{\slb}(\tau)) = 2e^{-2}(2e-5) \approx 0.1182, \quad 
\Var(e_2^{\slb}(\tau)) = e^{-2}(5e-13) \approx 0.0800, \]
which are substantially smaller than
\[ 
\Var(e_1^{\djb}(\tau)) = e^{-1} - e^{-2} \approx 0.2325, \quad 
\Var(e_2^{\djb}(\tau)) = \frac{1}{2e} - \frac{1}{4e^2} \approx 0.1501, \]
where $e_j^{\djb}$ denotes the disjoint blocks limit from Theorem 4.1 in \cite{Rob09}.

\begin{lem}[Tightness.] \label{tightness} Under the conditions of Theorem \ref{weak_conv}, and 
	for any $0 < \phi < \infty$ and $m \in \N_{\ge 0}$, the process $(E_{n,m}^{\slb})_{n \in \N}$ is asymptotically  tight in $D([0,\phi])^{m+1}$.
\end{lem}

\begin{lem}[Fidis-convergence.] \label{fidis}
	Suppose that Condition~\ref{cond}(i)--(ii) are met. Then, for $m \in N_{\ge0}$ and $\tau_1,\ldots,\tau_r \geq 0, r \in \N_{\ge 1}$,  we have
	\[ \big( E_{n,m}^{\slb}(\tau_1),\ldots, E_{n,m}^{\slb}(\tau_r) \big) \wto \big( E_{m}^{\slb}(\tau_1),\ldots, E_{m}^{\slb}(\tau_r) \big). \] 
\end{lem}

%%%%%%%%%%%%%%%%%%%%%%%%%%%%%%%%%%%%%%%%%%%%%%%%%%%%%%%%%%%%%%%%%%%%%%%%%%%%%%%%%%%%%%%%%%%%%%%%%%%%

\section{Proofs  for Section~\ref{sec:comp}} \label{sec:proofcomp}

\begin{proof}[Proof of Lemma \ref{tightness}]
Since marginal asymptotic tightness implies joint asymptotic tightness, it is sufficient to show asymptotic tightness of $e_{n,j}^{\slb}$ for fixed $j \in \N_{\ge 0}$. Subsequently, we omit the upper index $\slb$. 

 For sufficiently large $n$, the summands making up  $(e_{n,j} (\tau))_{\tau \in [0,\phi]}$ are only depending on $U_s^{\eps_2}=U_s\I(U_s>1-\eps_2)$, whence the beta-mixing coefficients based on the $\mathcal B^{\eps_2}_{\cdot:\cdot}$-sigma  fields become available; in particular, we may use that $\sqrt {k_n} \beta_{\eps_2}(b_n)=o(1)$.

Let $b_n' = 2 b_n$ and $\mathcal{K}_n= (n-b_n+1)/b_n' = O(n/b_n)$. For simplicity we assume that $\kn$ is an integer. For $k\in\{1, \dots, \mathcal K_n\}$, define 
	\begin{align*}
	A_k &= \{2(k-1)b_n'+1, \ldots, 2(k-1)b_n'+b_n'\}, \\
	B_k &= \{(2k-1)b_n'+1,\ldots, (2k-1)b_n'+b_n'\} \ 
	\end{align*}
	such that $|A_k|=|B_k|=l_n$ and $A_1 \cup B_1 \cup \ldots \cup A_{\kn} \cup B_{\kn} = \{ 1,\ldots, n-b_n+1\}$. By the coupling lemma in \cite{Ber79},  
	we can inductively construct an array $\{(\tilde U_s)_{s \in I_i^{\slb}}: i=1,\ldots, n-b_n+1 \}$ such that 
	\begin{align}
\mathrm{(i)} & \quad \forall\ k \in\{1, \dots, \mathcal K_n\}: \big\{ (\tilde U_s)_{s \in I_i^{\slb}}: i \in C_k \big\} \stackrel{D}{=} \big\{ ( U_s)_{s \in I_i^{\slb}}: i \in C_k \big\}, \nonumber \\
\mathrm{(ii)} & \quad \forall\ k \in\{1, \dots, \mathcal K_n\}: \nonumber \\ 
& \hspace{1cm}\Pro \Big( \big\{ (\tilde U_s)_{s \in I_i^{\slb}} : i \in C_k \big\} \neq \big\{ (U_s)_{s \in I_i^{\slb}} : i \in C_k \big\} \Big) \leq \beta_{\varepsilon_2}(b_n),\nonumber       \\
\mathrm{(iii)} & \quad \big\{ (\tilde U_s)_{s \in I_i^{\slb}}:i \in C_k \big\}_{k=1,\ldots,\kn} \ \textrm{ is (row-wise) independent,}\qquad    \label{coupling}
	\end{align}
	where $C_k \in \{A_k,B_k\}$. Next, to simplify the notation, define 
	\[ 
	N_i^{(\tau)} = N_{b_n,i}^{(\tau),\slb} = \sum\nolimits_{s \in I_i^{\slb}} \I \big( U_s > 1-\tau/b_n\big) \] 
and its version based on $(\tilde U_s)_s$ as 
\[ 
	\tilde N_i^{(\tau)} = \sum\nolimits_{s \in I_i^{\slb}} \I \big( \tilde U_s > 1-\tau/b_n\big).
	\] 
Further, let $\tilde e_{n,j}(\tau) = \sqrt{k_n} \{ \tilde p_n^{(\tau)} (j) - \varphi_{n,j}(\tau) \}$ where 
\[
\tilde p_n^{(\tau)}(j) = \frac{1}{n-b_n+1} \sum_{i=1}^{n-b_n+1} \I(\tilde N_i^{(\tau)}=j).
\]
We begin by showing that 
\begin{align} 
\sup_{\tau \in [0, \phi] } | e_{n,j}(\tau)-\tilde e_{n,j}(\tau)  | = \op. \label{sup_tilde} 
\end{align} 
	By Item~(i) in (\ref{coupling}), we have $\Pro(N_i^{(\tau)}=m) = \Pro(\tilde N_i^{(\tau)}=m)$, which implies 
	\begin{align*}
	e_{n,j}(\tau) - \tilde e_{n,j}(\tau)
	&= 
	\frac{\sqrt{k_n}}{n-b_n+1} \sum_{i=1}^{n-b_n+1} \big\{ \I(N_i^{(\tau)}=j) - \I(\tilde N_i^{(\tau)}=j) \big\} \\
	&= 
	\frac{\sqrt{k_n}}{n-b_n+1} \sum_{k=1}^{\kn} \sum_{i \in A_k \cup B_k} \big\{ \I(N_i^{(\tau)}=j) - \I(\tilde N_i^{(\tau)}=j) \big\}.
	\end{align*}
	For fixed $k \in \{1,\ldots,\kn\}$, we obtain 
	\begin{align*}
	& \phantom{{}={}} \Big|  \sum_{i \in A_k} \I(N_i^{(\tau)}=j) - \I(\tilde N_i^{(\tau)}=j) \Big| \\
	& \leq 2b_n \times \I \big( \big\{ N_i^{(\tau)}:i \in A_k \big\} \neq \big\{ \tilde N_i^{(\tau)} : i \in A_k \big\} \big) \\
	& \leq 2b_n  \times \I \big( \big\{ (U_s)_{s \in I_i^{\slb}}:i \in A_k \big\} \neq \big\{ (\tilde U_s)_{s \in I_i^{\slb}} : i \in A_k \big\} \big),
	\end{align*}
	The same holds true if $A_k$ is replaced by $B_k$. Hence, by Item~(ii) in (\ref{coupling}),
	\begin{align*}
	\Exp \Big[ \sup_{\tau \in [0, \phi] } | e_{n,j}(\tau)-\tilde e_{n,j}(\tau)  |  \Big] 
	\leq 
	\frac{\sqrt{k_n} \kn}{n-b_n+1} 4 b_n \beta_{\varepsilon_2}(b_n) = 2 \sqrt{k_n} \beta_{\varepsilon_2}(b_n),
	\end{align*}
	which converges to zero by assumption. Markov's inequality implies (\ref{sup_tilde}).

As a consequence of (\ref{sup_tilde}), it is sufficient to show that the process $(\tilde e_{n,j})_{n \in \N}$ is tight. Write $\tilde e_{n,j}(\tau) = A_{n,j}(\tau) + B_{n,j}(\tau)$, where 
	\begin{align*}
	A_{n,j}(\tau) &= \frac{1}{\sqrt{\kn}} \sum_{k=1}^{\kn} \big\{ \bar A_{n,j,k}(\tau) - \Exp[\bar A_{n,j,k}(\tau)] \big\},
	\end{align*}
	with
	\begin{align*}
	\bar A_{n,j,k}(\tau) &= \frac{\sqrt{k_n \kn}}{n-b_n+1} \sum_{i\in A_k} \I(\tilde N_i^{(\tau)}=m), 
	\end{align*}
	and where $B_{n,j}$ is defined analogously, but with $A_k$ replaced by $B_k$.
	Since finite sums of asymptotically tight processes are asymptotically tight, it is sufficient to show tightness of $A_{n,j}$ and $B_{n,j}$. We only treat $A_{n,j}$. For that purpose, note that $(\tilde U_t : t \in I_i^{\slb})_{i \in A_k}$ only depends on 
	\[ 
	\tilde U_n^{(k)} := 
	\big( \tilde U_{2(k-1)b_n'+1}, \ldots, \tilde U_{2(k-1)b_n'+b_n'+b_n-1} \big) \in \R^{3b_n-1}
	\] 
	 by the definition of $I_i^\slb$. Write 
	 \begin{align*}
	 \bar A_{n,j,k}(\tau) 
	 &= 
	 \frac{\sqrt{k_n \kn}}{n-b_n+1} \sum_{i \in A_k} \I \Big( \sum_{t \in I_i^{\slb}} \I \big( \tilde U_t > 1-\tau/b_n \big) = j \Big)  \\
	 &=
	 h_{n,j}^{(\tau)}(\tilde U_n^{(k)}) - h_{n,j-1}^{(\tau)}(\tilde U_n^{(k)}), 
	 \end{align*}
	where 
	\[ 
	h_{n,j}^{(\tau)} : \R^{3b_n-1} \to \R, \ 
	u \mapsto \frac{\sqrt{k_n \kn}}{n-b_n+1} \sum_{i=1}^{2b_n} \I \Big( \sum_{t \in I_i^{\slb}} \I \big(  u_t > 1-\tau/b_n \big) \leq j \Big), 
	\] 
As a consequence, we may write
	\begin{align*}
	A_{n,j}(\tau) &=  C_{n,j}(\tau) - C_{n,j-1}(\tau). 
	\end{align*}
	where
	\begin{align*}
	C_{n,j}(\tau) = \frac{1}{\sqrt{\kn}} \sum_{k=1}^{\kn} \big\{ h_{n,j}^{(\tau)}(\tilde U_n^{(k)}) - \Exp[h_{n,j}^{(\tau)}(\tilde U_n^{(k)})] \big\}  
	\end{align*}
	for $j\ge 0$ and $C_{n,-1}=0$.
	It is hence sufficient to show asymptotic tightness of $C_{n,j}$ for fixed $j\in \N_{\ge 0}$. Note that by the Item~(i) and (iii) in (\ref{coupling}), $(\tilde U_n^{\scs (k)})_{k=1,\ldots,\kn}$ is a row-wise i.i.d.\ triangular array. Let $P_n$ denote the distribution of $\tilde U_n^{\scs (1)}$ on $\R^{3b_n-1}$. Introducing the  empirical process
	\[
	\mathbb H_n = \frac{1}{\sqrt{\kn}} \sum_{k=1}^{\kn} (\delta_{\tilde U_n^{(k)}} - P_n),
	\]
	we may write $C_{n,j}(\tau) = \mathbb H_{n} h_{n,j}^{(\tau)}$, which may now be investigated by standard methods from empirical process theory \citep{VanWel96}. 
	
	For $\delta>0$, define classes of functions
	\begin{align*}
	\mathcal{F}_n &= \{ h_{n,j}^{(\tau)}: \tau \in [0,\phi] \}, \\
	\mathcal{F}_{n,\delta} &= \{ h_{n,j}^{(\tau)} - h_{n,j}^{(\tau')} : \tau,\tau' \in [0,\phi], |\tau - \tau'| \leq \delta \}
	\end{align*}
	Further, for a real-valued functional $H$  defined on a function class $\mathcal G$, let $\|H\|_{\mathcal{G}} = \sup_{g \in \mathcal{G}} |Hg|$. Clearly,
	\begin{align*}
	& \sup_{\tau, \tau' \in [0,\phi]:  |\tau - \tau'| \leq \delta} |C_{n,j}(\tau)-C_{n,j}(\tau')| = \| \mathbb H_{n}\|_{\mathcal{F}_{n,\delta}},
	\quad 
	\sup_{\tau \in [0,\phi]} |C_{n,j}(\tau)| = \| \mathbb H_{n} \|_{\mathcal{F}_n}.
	\end{align*}  
	Therefore, it suffices to prove  that
	\begin{align}
	\lim_{\delta \downarrow 0} \limsup_{n \to \infty} \Exp \big[ \| \mathbb H_{n} \|_{\mathcal{F}_{n,\delta}} \big] = 0, \quad  \limsup_{n \to \infty} \Exp \big[ \|  \mathbb H_{n}  \|_{\mathcal{F}_n} \big] < \infty. \label{key}
	\end{align}
	We show this by applying Theorem 2.14.2 in \cite{VanWel96}. We begin by constructing an envelope function $E_n$  for $\mathcal{F}_n$ which satisfies $|f| \leq E_n$ for all $f \in \mathcal{F}_n$ and all sufficiently large $n$.
	For that purpose note that, for any $\tau \in [0,\phi]$, 
	\begin{align*}
	\sup_{u \in \R^{3b_n-1}} |h_{n,j}^{(\tau)}(u)| 
	&= 
	\sup_{u \in \R^{3b_n-1}} \Big| \frac{\sqrt{k_n \kn}}{n-b_n+1} \sum_{i=1}^{2b_n} \I\Big( \sum_{t \in I_i^{\slb}} \I \big( u_t > 1-\tau/b_n \big) \leq j \Big) \Big| \\
	& \leq \frac{\sqrt{k_n \kn}}{n-b_n+1} 2b_n = \sqrt{2 \frac{n}{n-b_n+1}} \leq 2
	\end{align*}
	since $n-b_n+1 \geq n/2$ for sufficiently large $n$ by Condition \ref{cond}(ii). Hence, we may choose $E_n=E= 2$. Further note that $2E=4$ is an envelope function for $\mathcal{F}_{n,\delta}$.
	
	Next, let $\|\cdot\|_{n,2}$ be the norm $\| f \|_{n,2} = \Exp[f(\tilde U_n^{\scs (1)})^2]^{1/2}$ 
	 and define, for $\eta >0$, 
	\begin{align}
	\nonumber
	a_n(\eta) &= \eta  \| 2E_n \|_{n,2} / \sqrt{1+ \log N_{[ \, ]} (\eta \| 2E_n\|_{n,2}, \mathcal{F}_{n,\delta}, ||\cdot||_{n,2}) } \\
	 &= 4\eta  / \sqrt{1+ \log N_{[ \, ]} (4 \eta , \mathcal{F}_{n,\delta}, ||\cdot||_{n,2}) }
	 \label{eq:aneta}
	\end{align}
	where $N_{[ \, ]}$ denotes the bracketing number as in Definition 2.1.6 in \cite{VanWel96}. 
	
	Next, we prove the subsequent inequality: for any $\tau,\tau' \in [0,\phi+1]$, 
	\begin{align}
	\| h_{n,j}^{(\tau)}-h_{n,j}^{(\tau')}||_{n,2} \leq2  |\tau-\tau'|^{1/2}. \label{ineq}
	\end{align}
	Indeed, by Jensen's inequality 
	\begin{align}
	\| h_{n,j}^{(\tau)}- h_{n,j}^{(\tau')} \|_{n,2}^2 
	&= 
	\frac{k_n \kn}{(n-b_n+1)^2} \Exp \Big[ \Big( \sum_{i=1}^{2b_n} \{ \I(\tilde N_i^{(\tau)} \leq j ) - \I(\tilde N_i^{(\tau')} \leq j ) \} \Big)^2 \Big] \nonumber \\
	& \leq 
	\frac{k_n \kn(2b_n)^2}{(n-b_n+1)^2} \Exp \big[ \{ \I(\tilde N_1^{(\tau)} \leq j )- \I(\tilde N_1^{(\tau')} \leq j ) \}^2 \big] \nonumber \\
	& \leq 4 
	\Exp \big[ \{ \I(\tilde N_1^{(\tau)} \leq j )- \I(\tilde N_1^{(\tau')} \leq j ) \}^2 \big] \label{expec}
	\end{align}
	for sufficiently large $n$. Without loss of generality, let $\tau \leq \tau'$. Since $z \mapsto \I(\tilde N_1^{(z)} \leq j)$ is monotonically decreasing, one has 
	\begin{align*}
	\{ \I(\tilde N_1^{(\tau)} \leq j) - \I(\tilde N_1^{(\tau')} \leq j)\}^2
	&= 
	\I(N_1^{(\tau)} \leq j) - \I( N_1^{(\tau')} \leq j) \\
	&=
	\I(N_1^{(\tau)} \leq j < N_1^{(\tau')})  \\
	&\le 
	\I  ( N_1^{(\tau')} - N_1^{(\tau)} \ge 1)
	\end{align*}
	Hence, by (i) in (\ref{coupling}), the expression on the right-hand side of (\ref{expec}) can be bounded by
	\begin{align*} 
 4 \cdot \Pro(N_1^{(\tau')} - N_1^{(\tau)} \ge 1) \leq 4 \cdot \Exp[N_1^{(\tau')}-N_1^{(\tau)}] = 4 (\tau'-\tau)
	\end{align*} 
as asserted in \eqref{ineq}.

Next, \eqref{ineq} implies that, for any $\tau, \tau' \in [0,\phi]$ with $|\tau-\tau'| \leq \delta$, we have $\| h_{n,j}^{\scs (\tau)} - h_{n,j}^{\scs (\tau')} ||_{n,2} \leq2  |\tau-\tau'|^{1/2} < \delta^{1/2} \| 2E \|_{n,2}$, such that 
\[ 
\sup_{f \in \mathcal{F}_{n,\delta}}  \| f \|_{n,2} < \delta^{1/2} \|2E\|_{n,2} = 4 \delta^{1/2} . 
\] 
Hence, the condition in Theorem 2.14.2 in \cite{VanWel96} is met and we obtain
	\begin{multline}
	\Exp \big[ \| \mathbb H_{n} \|_{\mathcal{F}_{n,\delta}} \big] 
	 \lesssim 
	4 \int_0^{\delta^{1/2}} \sqrt{1+ \log N_{[ \, ]}(4 \varepsilon, \mathcal{F}_{n,\delta}, ||\cdot ||_{n,2}) } \ \mathrm{d} \varepsilon \\
	 + 4 \sqrt{\kn}  \Prob(4 > \sqrt{\kn } a_n(\delta^{1/2})) \label{exp_f_delta}    	
	\end{multline}
	and by the last part of the theorem
	\begin{align}
	\Exp \big[ ||H_{\mathcal{K}}||_{\mathcal{F}} \big] 
	& \lesssim 
	4 \int_0^{1} \sqrt{1+ \log N_{[ \, ]}(4 \varepsilon, \mathcal{F}_n, ||\cdot ||_{n,2}) } \ \mathrm{d} \varepsilon. \label{exp_f}
	\end{align}
	
	It remains to bound the bracketing numbers $N_{[ \, ]}$ appearing in the previous two displays. For that prurpose, we construct a cover of $\mathcal{F}_n$. For $\varepsilon \in(0,1)$ and $a \in \N_{\ge 1}$ let $D_{\varepsilon,a} = [(a-1) \varepsilon^{2}/4, a \varepsilon^{2}/4]$. Then
	\[ 
	[0,\phi] 
	\subset \bigcup_{a \in \{1,2,\ldots,M_{\varepsilon}\}} D_{\varepsilon,a} 
	\subset [0,\phi+1], 
	\qquad M_{\varepsilon} = \lfloor 4(\phi+1)/{\varepsilon^{2}} \rfloor.
	\] 
	Now, since $\tau \mapsto h_{n,j}^{(\tau)}$ is monotonically decreasing, we may choose, for any $\tau\in[0,\phi]$, an integer $a\in \{1, \dots, M_\eps\}$ such that
\[
h_{n,j}^{(a\eps^{1/2}/4)} \leq h_{n,j}^{(\tau)} \leq h_m^{((a-1)\eps^{2}/4)}.
\]
Moreover, for any $a \in \{1,\ldots,M_{\varepsilon}\}$, we have, by (\ref{ineq}) 
\[ 
\| h_{n,j}^{((a-1)\eps^{2}/4)} - h_{n,j}^{(a\eps^{2}/4)}\|_{n,2} 
\leq 
2 (\varepsilon^{2}/4)^{1/2}=\eps. 
\] 
Hence, the class $\mathcal{F}_n$ is covered by the collection of $\varepsilon$-brackets
	\[ 
	\left\{ \big[ h_m^{(a \varepsilon^{2}/4)}, h_m^{((a-1) \varepsilon^{2}/4)} \big]: a \in \{1,2,\ldots, M_{\varepsilon}\} \right\} 
	\] 
	which implies
	$
	N_{[ \, ]} (\varepsilon, \mathcal{F}_n, \| \cdot \|_{n,2}) 
	\leq 
	M_{\varepsilon} \lesssim \varepsilon^{-2}.  
	$ 
	Moreover, since $\mathcal{F}_{n,\delta} \subset \mathcal{F}_n - \mathcal{F}_n$, we obtain 
	\begin{align*}
	N_{[ \, ]}(\varepsilon,\mathcal{F}_{n,\delta}, \| \cdot \|_{n,2}) 
	&\leq 
	N_{[ \, ]} (\varepsilon, \mathcal{F}_n-\mathcal{F}_n, \|\cdot \|_{n,2})  \\
	&\leq 
	N_{[ \, ]}(\varepsilon/2, \mathcal{F}_n, \| \cdot \|_{n,2})^2 
	\lesssim
	\eps^{-4}.
	\end{align*} 
	
	The bounds on the covering numbers imply that, by (\ref{exp_f_delta}),
	\begin{align*}
	\Exp \big[ \| \mathbb H_{n} \|_{\mathcal{F}_{n,\delta}} \big] 
	& \lesssim 
	4 \int_0^{\delta^{1/2}} \sqrt{1+ \log (\eps^{-4}) } \ \mathrm{d} \varepsilon 
	 + 4 \sqrt{\kn}  \Prob(4 > \sqrt{\kn } a_n(\delta^{1/2})).
	\end{align*}
	For fixed $\delta >0$, we further obtain that $a_n(\delta^{1/2})$ from \eqref{eq:aneta} is bounded away from $0$, uniformly in $n$, such that the second summand in the previous display is eventually $0$ for large $n$, because $\kn \to \infty$ as $ n \to \infty$. Since the integral term is finite, we obtain the first assertion in (\ref{key}). Finally, the second assertion in (\ref{key}) follows from
	\begin{align*}
	\Exp \big[ \| \mathbb H_{n} \|_{\mathcal{F}_n} \big]  \lesssim 4 \int_0^{1} \sqrt{1+ \log (\eps^{-2})}  \ \mathrm{d} \varepsilon  < \infty
	\end{align*}  
	by (\ref{exp_f}).
\end{proof}

\begin{proof}[Proof of Lemma \ref{fidis}]
	By the Cramér-Wold device it suffices to show that 
	\begin{align} \label{eq:dnd}
	& D_n = \sum_{l=1}^r \sum_{j=0}^m \lambda_{l,j} \ e_{n,j}^{\slb}(\tau_l) 
	\wto \sum_{l=1}^r \sum_{j=0}^m \lambda_{l,j} \ e_{j}^{\slb}(\tau_l)  =D
	\end{align}
	for any $\lambda_{l,j} \in \R$. Throughout the proof, let $I_i=I_i^{\slb}$  and write
	\begin{align*}
D_n
	&= \sum_{j=1}^{k_n-1} \sum_{s \in I_j} \sum_{l=1}^r \sum_{j=0}^m 
	\lambda_{l,j} \frac{\sqrt{k_n}}{n-b_n+1} \big\{ \I(N_{b_n,s}^{(\tau_l),\slb}=j) - \varphi_{n,j}(\tau) \big\}  + \op. 
	\end{align*}
Let $k_n\ho{\ast} < k_n$ be an integer sequence with $k_n\ho{\ast} \to \infty$ and $k_n\ho{\ast} = 
	o(k_n^{1/4})$.
	For $q_n\ho{\ast} = \lfloor k_n/(k_n\ho{\ast}+2)\rfloor \to \infty$ and 
	$p=1,\ldots,q_n\ho{\ast}$, define
	\[ 
	J_p\ho{+} = \bigcup_{i=(p-1)(k_n\ho{\ast}+2)+1}\ho{p(k_n\ho{\ast}+2)-2} I_i^\djb , 
	\qquad
	J_p\ho{-} = I^\djb_{p(k_n\ho{\ast}+2)-1} \cup I^\djb_{p(k_n\ho{\ast}+2)}. 
	\]
	Thus, we have decomposed the observation period into $q_n\ho{\ast}$ `big blocks' $J_p^{\scs +}$ of size $k_n\ho{\ast}b_n$, which are separated by `small blocks'  $J_p^{\scs -}$ of size $2b_n$. We may hence rewrite  $D_n = V_n^+ + V_n^- + \op$, where 
	\[ 
	V_n^{\pm} = \frac{1}{\sqrt{q_n^{\ast}}} \sum_{p=1}^{q_n^{\ast}} T_{np}^{\pm} 
	\] 
	and, for $p \in\{1,\ldots,q_n^{\ast}\}$,
	\begin{align*}
	T_{np}^{\pm} &= \sqrt{\frac{q_n^{\ast}}{k_n}} \sum_{s \in J_p^{\pm}} \sum_{l=1}^r \sum_{j=0}^m \lambda_{l,j} \frac{n}{n-b_n+1} \frac{1}{b_n} \big\{ \I(N_{b_n,s}^{(\tau_l), \slb}=j) - \varphi_{n,j}(\tau_l) \big\}. 
	\end{align*}
	
	Let us show that $V_{n}^-=o_\Prob(1)$. For that purpose, 
	take $\varepsilon_1 \in (0,1)$ from Condition \ref{cond}.  Observe that, for sufficiently large $n$, $T_{np}^-$ only depends on $U_s^{\varepsilon_1}=U_s \I(U_s > 1- \varepsilon_1)$ with $s \in \{ (p(k_n^{\ast}+2)-2)b_n+1,\ldots, p(k_n^{\ast}+2)b_n+b_n-1 \}$. 
	Now, since $\Exp[V_n^-] =0$, it is enough to prove $\Var(V_n^-) = o(1)$. By stationarity, 
	\begin{align}
	\Var(\tilde V_n^-) 
	\leq 
	3 \Var(T_{n1}^{-}) + 
	2 \sum_{p=2}^{q_n^{\ast}} \big| \Cov(T_{n1}^{-}, T_{n,p+1}^{-}) \big|. \label{var_vn_minus1}
	\end{align}

Observing that $|J_1^-|=2b_n$ and $n/(n-b_n+1) \le 2$ for sufficiently large $n$, we have
	\begin{align}
	|T_{n1}^{-}|
	& \leq 
	4 \sqrt{\frac{q_n^{\ast}}{k_n}} \sum_{l=1}^r \sum_{j=0}^m |\lambda_{l,j}|
	= O\Big( \sqrt{\frac{q_n^{\ast}}{k_n}} \Big) = O \Big( \frac{1}{\sqrt{k_n^{\ast}}} \Big) =o(1), \label{Tn1_minus1}
	\end{align}
	which implies $\Var(T_{n1}^{-})=o(1)$ as well. 
	Next, by Lemma 3.9 in \cite{DehPhi02}, Condition \ref{cond}(ii) and since $T_{n,p}^{-}$ is bounded, we obtain
	\begin{align*}
	\sum_{p=2}^{q_n^{\ast}} \big| \Cov(T_{n1}^{ -},T_{n,p+1}^{ -}) \big| 
	& \leq 
	 4 \| T_{n1}^{ -} \|_{\infty} \sum_{p=2}^{q_n^{\ast}} \alpha_{\varepsilon_1}(pk_n^{\ast}b_n) \\	
	& \lesssim o(1) \sum_{p=2}^{q_n^{\ast}} (pk_n^{\ast}b_n)^{-\eta}  =o(1), 
	\end{align*}
	such that altogether $\Var(V_n^-) =o(1)$ by (\ref{var_vn_minus1}).

	It remains to show that $V_n^+$ converges in distribution to $D$ from \eqref{eq:dnd}. 
	 Since $T_{np}^{\scs +}$ and $T_{np'}^{\scs +}$ are based on $U_s^{\eps_1}$ observations that  are at least $b_n$ observations apart for $p \neq p'$, and since $q_n^{\ast} \alpha_{\eps_1}(b_n) \leq k_n \alpha_{\eps_1}(b_n) \lesssim k_n b_n^{-\eta}$ by Condition \ref{cond}(ii), a standard argument based on characteristic functions implies that $(T_{np}^{+})_{p=1,\ldots,q_n^{\ast}}$ may be considered independent, which is assumed from now on. As in (\ref{Tn1_minus1}), we obtain that $|T_{np}^+|=O(\sqrt{k_n^{\ast}})$, whence
	\[ 
	\frac{\sum_{p=1}^{q_n^{\ast}}  \Exp\big[ |T_{np}^+|^{3} \big] }{\big\{ \sum_{p=1}^{q_n^{\ast}} \Var(T_{np}^+) \big\}^{3/2}} = O \big( k_n^{-1/2} (k_n^{\ast})^{2} \big) =o(1), 
	\]
	provided that $\lim_{n \to \infty} \Var(T_{n1}^+)$ exists.  In this case, the Ljapunov condition is satisfied and the central limit theorem implies that $V_{n}^+$ converges in distribution  to a centered normal distribution with variance $\lim_{n \to \infty}\Var(T_{n1}^+)$. Note that 
	\[ 
	T_{n1}^+ = \sum_{l=1}^r \sum_{j=0}^m \lambda_{l,j} \ e_{n^*,j}^{\slb}(\tau_l) + R_n,
	\] 
     where $R_n \to 0$ in $L_2(\Prob)$, 
	with $n^{\ast} = k_n^{\ast} b_n$ and that our assumptions in Condition \ref{cond} still hold if $n$ and $k_n$ are substituted by $n^{\ast}$ and $k^{\ast}_n$. The limiting variance  of the above expression is calculated in Lemma \ref{Sliding_Cov} below and is seen to be of the required form. 
\end{proof}

\begin{lem} \label{Sliding_Cov}
	Suppose that Condition~\ref{cond}(i)--(ii) are met. Then, for $0 \leq \tau \leq \tau'$ and $j,j' \in \N_{\ge 0}$, we have
	\begin{align*}
	\lim_{n\to\infty} 
	\Cov (e_{n,j}^{\slb}(\tau), e_{n,j'}^{\slb}(\tau'))
	 &=  2 \intne \Cov \big( \I(X_{\xi}^{(\tau)} = j), \I(Y_{\xi}^{(\tau')} = j') \big) \, \di \xi \\
	&=  2 \intne H_{j,j'}^{(\tau,\tau')}(\xi) \, \di \xi - 2 p^{(\tau)}(j) p^{(\tau')}(j'),
	\end{align*}
	where  $X_{\xi}^{(\tau)} = Y_{\xi}^{(\tau)} = N_E^{(\tau)}$ in distribution with joint probability mass function 
	\begin{align} 
	H_{j,j'}^{(\tau,\tau')}\!(\xi) 
	&= 
	\Pro \big( X_{\xi}^{(\tau)}  =  j, Y_{\xi}^{(\tau')}  =  j' \big)  \nonumber \\ 
	&=  
	\sum_{l=0}^j \sum_{r=j-l}^{j'} p^{(\xi\tau)}(l) p^{(\xi\tau')}(j'-r) p_2^{((1-\xi)\tau',(1-\xi)\tau)}(r,j-l).
	\label{eq:hjj}
	\end{align}
\end{lem}

\begin{proof}[Proof of Lemma \ref{Sliding_Cov}]
	Fix $0 \leq \tau \leq \tau'$ and $j,j'  \in \N_{\ge 0}$. 
	 Note that we may replace $U_s$ by $U_s^{\eps_1}=U_s\I(U_s>1-{\eps_1})$ for $n$ large enough, where $\eps=\varepsilon_1$ is from Condition \ref{cond}(ii).	
	Write 
	\begin{align*}
	r_n(\tau, \tau') &\equiv 
	\Cov (e_{n,j}^{\slb}(\tau), e_{n,j'}^{\slb}(\tau')) \\
	&= 
	\frac{k_n}{(n-b_n+1)^2} \sum_{s,t=1}^{n-b_n+1}\Cov (\I(N_{b_n,s}^{(\tau),\slb}=j), \I(N_{b_n,t}^{(\tau'),\slb}=j') ) \\
	&= 
	\frac{k_n}{(n-b_n+1)^2} \sum_{i,i'=1}^{k_n-1} \sum_{s \in I_i} \sum_{t \in I_{i'}} \Cov(A_s,B_t) +o(1),
	\end{align*}
	where $A_s = \I(N_{b_n,s}^{(\tau),\slb}=j), B_t=\I(N_{b_n,t}^{(\tau'),\slb}=j')$ and $I_i = I_i^\djb$. By stationarity, we may further write
	\begin{align}
	r_n(\tau, \tau') 
	&=  \nonumber
	\frac{k_n (k_n-1)}{(n-b_n+1)^2} \Cov \Big( \sum_{s \in I_1} A_s, \sum_{t \in I_1} B_t \Big)  \\
	& \hspace{.7cm}  \nonumber
	+ \frac{k_n}{(n-b_n+1)^2} \sum_{i=2}^{k_n-1} (k_n-i) \bigg\{ \Cov \Big( \sum_{s \in I_1} A_s, \sum_{t \in I_i} B_t \Big)  \\
		& \hspace{5.2cm} \nonumber
		+ \Cov \Big( \sum_{s \in I_q} A_s, \sum_{t \in I_1} B_t \Big) \bigg\} +o(1)  \\
	&=   \label{eq:tnd}
	T_{n1} + T_{n2} + T_{n3} + T_{n4} +o(1),
	\end{align}
	where
	\begin{align*}
	T_{n1} 
	&=  
	\frac{k_n(k_n-1)}{(n-b_n+1)^2} \Cov \Big( \sum_{s \in I_1} A_s, \sum_{t \in I_1} B_t \Big)  \\
	T_{n2} 
	&= 
	\frac{k_n (k_n-2)}{(n-b_n+1)^2} \bigg\{ \Cov \Big( \sum_{s \in I_1} A_s, \sum_{t \in I_2} B_t \Big) + \Cov \Big( \sum_{s \in I_2} A_s, \sum_{t \in I_1} B_t \Big) \bigg\}  \\
	T_{n3} 
	&=
	 \frac{k_n (k_n-3)}{(n-b_n+1)^2} \bigg\{ \Cov \Big( \sum_{s \in I_1} A_s, \sum_{t \in I_3} B_t \Big) + \Cov \Big( \sum_{s \in I_3} A_s, \sum_{t \in I_1} B_t \Big) \bigg\} \\
	T_{n4} 
	&=
	 \frac{k_n}{(n-b_n+1)^2} \sum_{i=4}^{k_n-1} (k_n-i) \bigg\{ \Cov \Big( \sum_{s \in I_1} A_s, \sum_{t \in I_i} B_t \Big) \\
	& \hspace{6.8cm} 
	+ \Cov \Big( \sum_{s \in I_i} A_s, \sum_{t \in I_1} B_t \Big) \bigg\}.
	\end{align*}
	Next, we show that 
	\begin{align} \label{eq:tn34}
	T_{n3} =o(1), \quad T_{n4} = o(1).
	\end{align}
	For that purpose note that $\sum_{s \in I_1} A_s$ and $\sum_{t \in I_i} B_s$ are at least $(i-3)b_n$ observations apart. 
	By Lemma 3.9 in \cite{DehPhi02} we obtain
	\begin{align*}
	\big| \Cov \Big( \sum_{s \in I_1} A_s, \sum_{t \in I_i} B_t \Big) \big| 
	& \leq 4 \ b_n^2 \alpha_{\eps_1}((i-3)b_n), 
	\end{align*}
	such that 
	\begin{align*}
	|T_{n4}| 
	& \leq 
	\frac{8 \ k_n^2 b_n^2}{(n-b_n+1)^2} \sum_{i=4}^{k_n-1} \alpha_{\eps_1}((i-3)b_n) 
	\lesssim 
	\frac{n^2  b_n^{-\eta}}{(n-b_n+1)^2} \sum_{i=1 }^{k_n-4} i^{-\eta} =o(1)
	\end{align*}
	since $\eta>1$ by Condition \ref{cond}(ii). 
	Regarding $T_{n3}$, note that 
	\begin{align*}
	\big| \Cov \Big( \sum_{s \in I_1} A_s, \sum_{t \in I_3} B_t \Big) \big| 
	& \leq 
	\sum_{t=2b_n+1}^{3b_n} \big| \Cov \Big( \sum_{s=1}^{b_n} A_s,B_t \Big) \big| \\
	& \leq 
	4 \, b_n \sum_{t=2b_n+1}^{3b_n} \alpha_{\eps_1}(t-2b_n) = 4 \, b_n \sum_{t=1}^{b_n} \alpha_{\eps_1}(t)
	\end{align*}
	by Lemma 3.9 in \cite{DehPhi02}, which implies 
	\begin{align*}
	|T_{n3}| \leq 8 \frac{k_n(k_n-3)b_n}{(n-b_n+1)^2} \sum_{t=1}^{b_n} \alpha_{\eps_1}(t) \lesssim \frac{k_n^2 b_n}{(n-b_n+1)^2} \sum_{t=1}^{b_n} t^{-\eta} = O(b_n^{-1}).
	\end{align*}
	Hence, \eqref{eq:tn34} is shown.
	
	Next, consider $T_{n1}$.  Since $k_n (k_n-1)/(n-b_n+1)^2 = 1/b_n^2 + o(1)$ and $\Exp[A_s]  \to p^{(\tau)}(j)$ and $\Exp[B_t] \to p^{(\tau')}(j')$, we may write
	\begin{align*}
T_{n1} = \frac{1}{b_n^2} \sum_{s,t=1}^{b_n} \! \Exp[A_s B_t] - p^{(\tau)}(j) p^{(\tau)}(j') + o(1),
	\end{align*}
	Next, we have $b_n^{-2} \sum_{s,t=1}^{b_n} \Exp[A_s B_t] = \intne f_n(\xi) \ \mathrm{d}\xi$, where, for $\xi \in (0,1)$, 
	\begin{align*}
	f_n(\xi) 
	&= 
	\frac{1}{b_n} \sum_{s,t=1}^{b_n} \Exp[A_s B_t] \I \Big( \xi \in \Big[ \frac{t-1}{b_n},\frac{t}{b_n} \Big) \Big) \\
	&=
	 \frac{1}{b_n} \sum_{s=1}^{b_n} \Exp[A_s B_{\lfloor b_n \xi \rfloor +1}] = \intne \varphi_n(\xi,z) \ \mathrm{d}z,
	\end{align*}
where, for $z \in (0,1)$,
	\begin{align}
 \varphi_n(\xi,z) 
 &= \label{eq:defphi}
 \sum_{s=1}^{b_n} \Exp[A_s B_{\lfloor b_n \xi \rfloor +1}] \I \Big( z \in \Big[ \frac{s-1}{b_n},\frac{s}{b_n} \Big) \Big)   \\
 &= \nonumber
 \Exp \big[ A_{\lfloor b_n z \rfloor +1} B_{\lfloor b_n \xi \rfloor +1} \big] 
 = 
  \Pro \big( N_{b_n,\lfloor b_n z \rfloor +1}^{(\tau),\slb} = j, N_{b_n,\lfloor b_n \xi \rfloor +1}^{(\tau').\slb} = j' \big) 
	\end{align}
	For $0 < z \leq \xi < 1$, we may rewrite	
	\begin{align}
	 \varphi_n(\xi,z)  
	 &=
	 \sum_{l=0}^{j} \sum_{r=0}^{j'}  \Pro \Big( 
	 N_{\lfloor b_n z \rfloor + 1 : \lfloor b_n \xi \rfloor}(\tau) =l, 
	 N_{\lfloor b_n \xi \rfloor + 1 : \lfloor b_n z \rfloor + b_n}(\tau)= j-l, \nonumber \\ 
	 & N_{\lfloor b_n \xi \rfloor + 1 : \lfloor b_n z \rfloor + b_n}(\tau')= r,
	 N_{\lfloor b_n z \rfloor + b_n + 1 : \lfloor b_n \xi \rfloor + b_n}(\tau')= j'-r\Big).  \label{phi_n}
	\end{align}
	where, for $s,s' \in \N_{\ge 1}$ with $s \le s'$ and $\tau\ge 0$,
	\[
	N_{s:s'}(\tau) = \sum_{t=s}^{s'} \I \Big(U_t>1-\frac{\tau}{b_n}\Big).
	\]
	We will next argue that the first, the intersection of the second and the third and the fourth  of the four events in each summand in \eqref{phi_n} may be considered independent.
	Indeed, for any fixed $y>0$ and any integer sequence $q_n$ converging to infinity with $q_n =o(b_n)$, we have 
	\begin{align*}
	\Pro \big(N_{1:q_n}(y) = 0 \big) \geq 1-q_n \Pro \big( U_1 > 1-\frac{y}{b_n} \big) = 1-\frac{yq_n}{b_n} \to 1
	\end{align*}
	As a consequence, we may intersect the events inside the sum in \eqref{phi_n} with
	\begin{align} \label{eq:nnn}
	\{N_{\lfloor b_n \xi \rfloor-q_n:\lfloor b_n \xi \rfloor}(\tau) =0, N_{\lfloor b_n z \rfloor + b_n + 1 : \lfloor b_n z \rfloor + b_n + 1 + q_n}(\tau')= 0\}.
	\end{align}
	at the expense of a $O(q_n/b_n)$-term. On the intersected event, we must then have $N_{\lfloor b_n z \rfloor +1:\lfloor b_n \xi \rfloor-q_n}(\tau)=l$ and  $N_{\lfloor b_n z \rfloor + b_n + q_n : \lfloor b_n \xi \rfloor + b_n}(\tau)= j'-r$. After discarding the events in \eqref{eq:nnn} again, we are left with an intersection of three events that are based on  observations that are at least $q_n$ observations apart. As a consequence, at the expense of an $\alpha_{\eps_1}(q_n)$-error, they may be considered independent. Finally, we may sneak in the omitted observations once again at the expense of an additional $O(q_n/b_n)$-term, and we arrive at
	\begin{align}
	\varphi_n(\xi, z) & = 
	\sum_{l=0}^{j} \sum_{r=0}^{j'} \
	\Pro \Big( N_{\lfloor b_n z \rfloor + 1 : \lfloor b_n \xi \rfloor}(\tau) =l\Big) \nonumber  \\
	 &\hspace{1cm} \times \Pro\Big( N_{\lfloor b_n \xi \rfloor + 1 : \lfloor b_n z \rfloor + b_n}(\tau)= j-l,
	 N_{\lfloor b_n \xi \rfloor + 1 : \lfloor b_n z \rfloor + b_n}(\tau')= r\Big) \nonumber \\
  &\hspace{1cm} \times \Pro\Big( N_{\lfloor b_n z \rfloor + b_n + 1 : \lfloor b_n \xi \rfloor + b_n}(\tau')= j'-r\Big)\nonumber   \\  
  &\hspace{.5cm}  + O(\alpha_{\eps_1}(q_n)) + O(q_n/b_n) \label{manipulation}
	\end{align}
 which converges to
 	\begin{multline*}
	H(\xi-z)  	
	=
	H_{j,j'}^{(\tau, \tau')}(\xi-z)\\
	= \sum_{l=0}^j \sum_{r=j-l}^{j'} 
		p^{((\xi-z)\tau)}(l)  p^{((\xi-z)\tau')}(j'-r) 
		p_2^{((1-\xi+z)\tau',(1-\xi+z)\tau)}(r,j-l)
	\end{multline*}
	by Condition \ref{cond}(i), where $H_{j,j'}^{(\tau, \tau')}$ is defined in \eqref{eq:hjj}.
	Changing the roles of $z$ and $\xi$, we obtain 
	\[
	\varphi_n(\xi,z) \to H(\xi-z)\I(z \leq \xi) + H(z-\xi) \I(z > \xi). 
	\]
	For fixed $\xi \in (0,1)$, $\sup_{n \in \N} ||\varphi_n(\xi,\cdot)||_{\infty} \leq 1$, such that the dominated convergence theorem implies 
	\[ 
	f_n(\xi) = \intne \varphi_n(\xi,z) \ \mathrm{d}z \to \int_0^{\xi} H(\xi-z) \ \mathrm{d}z + \int_{\xi}^1 H(z-\xi) \ \mathrm{d}z. 
	\] 
	Moreover, since $||f_n||_{\infty}  \leq 1$, dominated convergence also implies that
	\begin{align}
	\lim_{n\to\infty} T_{n1} 
	&= \intne \int_0^{\xi} H(\xi-z) \, \mathrm{d}z + \int_{\xi}^{1} H(z-\xi) \, \mathrm{d}z \, \mathrm{d}\xi - p^{(\tau)}(j) p^{(\tau')}(j') \nonumber \\
	& = 2 \intne \int_0^{\xi} H(\xi-z) \, \mathrm{d}z \, \mathrm{d}\xi - p^{(\tau)}(j) p^{(\tau')}(j') \nonumber \\ 
	& = 2 \int_{0}^{1} (1-\xi) H(\xi) \, \di \xi - p^{(\tau)}(j) p^{(\tau')}(j'), \label{eq:tn1}
	\end{align}
	where the last step is due to Fubini's theorem.

	It remains to treat $T_{n2}$ in \eqref{eq:tnd}, which consists of two summands, say $T_{n2,1}$ and $T_{n2,2}$. By similar arguments as for $T_{n1}$, the first summand  $T_{n2,1}$ can be written as
	\begin{align*}
T_{n2,1}
	&=  \frac{1}{b_n^2} \sum_{s=1}^{b_n} \sum_{t=b_n+1}^{2b_n} \Exp[A_s B_t] - p^{(\tau)}(j) p^{(\tau')}(j') + o(1) \nonumber \\
	&=  \int_0^1 \int_0^1 \psi_n(\xi,z) \, \di z  \, \di \xi - p^{(\tau)}(j) p^{(\tau')}(j') +o(1) 
	\end{align*}
	where
	\begin{align*}
	\psi_n(\xi,z) 
	&=  
	\Exp \big[ A_{\lfloor b_n \xi \rfloor +1} B_{\lfloor (z+1)b_n \rfloor +1} \big] \\
	&= \nonumber
	\Pro \Big(N_{\lfloor b_n \xi \rfloor +1: \lfloor b_n \xi \rfloor +b_n}(\tau)=j,
	N_{\lfloor b_n (z+1) \rfloor +1: \lfloor  b_n (z+1) \rfloor +b_n}(\tau')=j'\Big).
	\end{align*}
	If $\xi \leq z$, then $\lfloor b_n \xi \rfloor +b_n \leq \lfloor b_n (1+z) \rfloor +1$ and we can manipulate the above probability as in (\ref{manipulation}), such that it equals
	\begin{align*} 
		\psi_n(\xi,z)  &= \Pro \Big(N_{\lfloor b_n \xi \rfloor +1: \lfloor b_n \xi \rfloor +b_n}(\tau)=j \Big) \Pro \Big( N_{\lfloor b_n (z+1) \rfloor +1: \lfloor  b_n (z+1) \rfloor +b_n}(\tau')=j' \Big) \\
	&\qquad + O(\alpha_{\eps_1}(q_n)) + O(q_n/b_n),
	\end{align*}
	which converges to $p^{(\tau)}(j)p^{(\tau')}(j')$. In the case $z \leq \xi$, we again need to separate the sums as in (\ref{manipulation}) and obtain that $\psi_n(\xi,z)$ equals 
		\begin{align*}
	&\sum_{l=0}^{j} \sum_{r=0}^{j'} \
	\Pro \Big( N_{\lfloor b_n \xi \rfloor + 1 : \lfloor b_n (z+1) \rfloor}(\tau) =l\Big) \nonumber  \\
	 &\hspace{1cm} \times 
	 \Pro\Big( N_{ \lfloor b_n (z+1) \rfloor + 1: \lfloor b_n \xi \rfloor + b_n}(\tau)= j-l,
	 N_{ \lfloor b_n (z+1) \rfloor + 1: \lfloor b_n \xi \rfloor + b_n}(\tau')= r\Big) \nonumber \\
  &\hspace{1cm} \times 
  \Pro\Big( N_{ \lfloor b_n \xi \rfloor + b_n + 1: \lfloor b_n (z+1) \rfloor + b_n}(\tau')= j'-r\Big)\nonumber   \\  
  &\hspace{.5cm}  + O(\alpha_{\eps_1}(q_n)) + O(q_n/b_n) 
	\end{align*}
	which converges to
	\begin{multline*}
		H(1-(\xi-z)) = H_{j,j'}^{(\tau, \tau')}(1-(\xi-z))  \\
		= 
		\sum_{l=0}^j \sum_{r=j-l}^{j'} p^{((1-\xi+z)\tau)}(l)   
		p^{((1-\xi+z)\tau')}(j'-r)
		p_2^{((\xi-z)\tau', (\xi-z)\tau)}(r,j-l).
	\end{multline*}
Since $\|\psi_n\|_{\infty} \leq 1$, dominated convergence implies
	\begin{align*}
\lim_{n\to\infty} T_{n2,1} &=
\intne \int_0^{\xi} H(1-(\xi-z)) \, \mathrm{d}z  + \int_\xi^1 p^{(\tau)}(j) p^{(\tau')}(j') \, \di  z \, \di \xi \\ 
& \hspace{4cm}
- p^{(\tau)}(j) p^{(\tau')}(j') \\
&= \intne 
\xi H(\xi) \, \mathrm{d}\xi - \frac12 p^{(\tau)}(j) p^{(\tau')}(j')
	\end{align*}
	as $n \to \infty$. 
	By symmetry, the second summand in $T_{n2}$ has the same limit,
	such that  
	\begin{align} 
	\lim_{n\to\infty} T_{n2}  
	& = 2 \intne \xi H(\xi) \, \mathrm{d}\xi   - p^{(\tau)}(j) p^{(\tau')}(j'),
	\label{eq:tn2}
	\end{align}
	where the last equation follows as in (\ref{eq:tn1}).
	Altogether, by \eqref{eq:tn34}, \eqref{eq:tn1} and \eqref{eq:tn2}, we have
	\begin{align*}
	\lim_{n\to\infty} \Cov \big( e_{n,j}^{\slb}(\tau), e_{n,j'}^{\slb}(\tau') \big)
	& = 2 \intne H_{j,j'}^{(\tau, \tau')}(\xi) \, \di \xi - 2 p^{(\tau)}(j) p^{(\tau')}(j')
	\end{align*}
	as asserted.
\end{proof}

\appendix
\renewcommand{\thesection}{\Alph{section}}

%\section{Further proofs} \label{sec:app}

\section{Auxiliary lemmas - Disjoint blocks}

Throughout, assume that Condition \ref{cond} is met. All convergences are for $n \to \infty$ if not stated otherwise.

\begin{lem} \label{dj_ejn_integral}
	For any $j \in \N_{\ge 1}$, 
	\[ 
	\intnu e_{n,j}^{\djb}(\tau) \ \mathrm{d}(\hat H_{n}^{\djb}-H)(\tau) = \op. 
	\]
\end{lem}

\begin{proof}[Proof of Lemma \ref{dj_ejn_integral}]
Throughout the proof, we omit the upper index $\djb$ at all instances of $\hat H_n^{\djb}, e_{n,j}^\djb$ and $Z_{ni}^\djb$. For any $\delta >0$ and $\ell\in\N_{\ge1}$, we have
\begin{multline*}
	\Pro \Big( \Big| \intnu e_{n,j}(\tau) \ \mathrm{d}(\hat H_{n}-H)(\tau) \Big| > 3 \delta \Big) \\
	\leq  
	\Prob( |A_{n,\ell}|>\delta) +  \Prob( |B_{n,\ell,1}|>\delta)  + \Prob( |B_{n,\ell,2}|>\delta),
	\end{multline*}
	where
	\begin{align} \label{eq:anl1}
		A_{n, \ell} &= \int_{0}^{\ell} e_{n,j}(\tau) \ \mathrm{d}(\hat H_{n}-H)(\tau)
		\end{align}
		and
		\begin{align} \label{eq:bnlv1}
		B_{n,\ell,1} &=  \int_{\ell}^{\infty} e_{n,j}(\tau) \ \mathrm{d}\hat H_{n}(\tau),  \quad
		B_{n,\ell,2} = \int_{\ell}^{\infty} e_{n,j}(\tau) \ \mathrm{d}H(\tau) .
	\end{align}
The proof is finished once we have shown that 
\begin{align} \label{eq:anl2}
\forall\, \ell \in \N_{\ge1}: \quad	A_{n, \ell} =  \op,
\end{align}
and that, for $v\in\{1,2\}$,
\begin{align}  \label{eq:bnlv2}
 \lim_{\ell \to \infty} \limsup_{n \to \infty} 
 \Pro \big( |B_{n,\ell,v} | > \delta \big) = 0.
 \end{align}

We start by showing \eqref{eq:anl2}. Fix $\ell\in\N_{\ge 1}$.
From the proof of Lemma 9.2 in \cite{BerBuc18}, we have
\[ 
\sup_{\tau \in [0,\ell]} \big| \hat H_{n}(\tau)-H(\tau) \big| = \op.
\] 
Next, some thoughts reveal that the proof of Theorem 4.1 in \cite{Rob09} in fact allows for setting $\sigma =0$ in his notation, such that 
\[
\{e_{n,j}(\tau)\}_{\tau \in [0,\ell]} \wto \{e_j(\tau)\}_{\tau \in [0,\ell]}
\] 
in $D([0,\ell])$, for some centered Gaussian process $e_{j}$ (see also Theorem~\ref{weak_conv} for an analogous result for the sliding blocks version $e_{\scs n,j}^{\scs \slb}$). The previous two displays  imply \eqref{eq:anl2} by Lemma C.8 in \cite{BerBuc17}.

Next, consider \eqref{eq:bnlv2} with $v=1$. We have
\begin{align*} 
B_{n,\ell,1}
&= 
k_n^{-3/2} \sum_{i,i'=1}^{k_n} \Big\{ \! \I\big( N_{b_n,i'}^{(Z_{ni})} \! = j\big) - \varphi_{n,j}(Z_{ni}) \Big\} \I(Z_{ni} \geq  \ell)   \\
&= T_{n,\ell} \! 
	+S_{n,\ell,1} \! +S_{n,\ell,2},
	\end{align*}
	where 
	\begin{align*}
	T_{n,\ell} &=  k_n^{-3/2} \sum_{i=1}^{k_n} \sum_{i' \in \{i-1,i,i+1\}} \Big\{ \I\big( N_{b_n,i'}^{(Z_{ni})} = j\big) - \varphi_{n,j}(Z_{ni}) \Big\} \I(Z_{ni} \geq \ell),\\
	S_{n,\ell,1} &=  k_n^{-3/2} \sum_{i=3}^{k_n} \sum_{i'=1}^{i-2} \Big\{ \I\big( N_{b_n,i'}^{(Z_{ni})} = j\big) - \varphi_{n,j}(Z_{ni}) \Big\} \I(Z_{ni} \geq \ell), \\
	S_{n,\ell,2} &=  k_n^{-3/2} \sum_{i=1}^{k_n-2} \sum_{i'=i+2}^{k_n} \Big\{ \I\big( N_{b_n,i'}^{(Z_{ni})} = j\big) - \varphi_{n,j}(Z_{ni}) \Big\} \I(Z_{ni} \geq \ell).
	\end{align*}	
Clearly,  $|T_{n,\ell}| \le 3k_n^{-1/2}=o(1)$. 
Next, write  $\eps=\eps_1\in(0,1)$ and $c>1-\eps$ from Condition~\ref{cond}(iii) as $c=1-\kappa \eps$  for some $\kappa \in (0,1)$, and let
\begin{align*} 
C_n = C_n(\varepsilon) 
= \big\{ \max_{i=1,\ldots,k_n} Z_{ni} < \kappa \varepsilon b_n \big\}
= \big\{ \min_{i=1,\ldots,k_n} N_{ni} > 1 - \kappa \varepsilon \big\},
\end{align*}
where $N_{ni}= \max\{U_s: s \in I_i^{\djb}\}$. We obtain $\Pro(C_n) \to 1$ as $n \to \infty$ by Condition \ref{cond}(iii). As a consequence, \eqref{eq:bnlv2} with $v=1$ follows once we have  shown that
\begin{align} \label{eq:snc}
	\lim_{\ell \to \infty} \limsup_{n \to \infty} \Pro \big( |S_{n,\ell,w} \I_{C_n} | > \delta \big) = 0, \quad w \in \{1,2\}.
\end{align} 

We only prove this for the term $S_{n,\ell,1}$, as $S_{n,\ell,2}$ can be treated analogously. 
Define $N_{b_n,j,\varepsilon}^{\scs (\tau)}$ as $N_{b_n,j}^{\scs (\tau)}$ and $Z_{ni}^{\varepsilon}$ as $Z_{ni}$, but with $U_s$ substituted by $U_s^{\varepsilon} = U_s \I(U_s > 1- \varepsilon)$, respectively. Then, $Z_{ni} < \eps \kappa b_n$ iff $Z_{ni}^{\eps \kappa} < \eps \kappa b_n$, and in that case we have 
\begin{compactenum}
\item[(1)] $Z_{ni}=Z_{ni}^{\eps \kappa}$,
\item[(2)] $U_s>1-Z_{ni}^{\eps \kappa}$ iff $U_s^\eps>1-Z_{ni}^{\eps \kappa}$.
\end{compactenum}
As a consequence, $S_{n,\ell,1} \I_{C_n} = S_{n,\ell,1}^{\varepsilon} \I_{C_n}$,
where 
	\begin{align*}
	S_{n,\ell,1}^{\varepsilon} 
	&= 
	\frac{1}{k_n} \sum_{i=3}^{k_n} f_{n,i-2}(Z_{ni}^{\varepsilon \kappa}) \ \I \big( \varepsilon \kappa  b_n > Z_{ni}^{\varepsilon \kappa} \geq \ell \big) 
	\end{align*}
	and where
	\begin{align} \label{eq:fn}
	f_{n,i-2}(\tau) &= k_n^{-1/2} \sum_{i'=1}^{i-2} \Big\{ \I \big( N_{b_n,i',\varepsilon}^{(\tau)} = j \big) -\varphi_{n,j}(\tau) \Big\}.
	\end{align} 
	We may further write $f_{n,i-2}(\tau)= h_{n,i-2,j}(\tau)-h_{n,i-2,j-1}(\tau)$, where 
	\begin{align} \label{eq:hnip}
	h_{n,i-2,p}(\tau)= k_n^{-1/2} \sum_{i'=1}^{i-2} \I \big( N_{b_n,i',\varepsilon}^{(\tau)} \leq p \big) - \Pro \big( N_{b_n,i'}^{(\tau)} \leq p \big), \quad p \in \N_{\ge0}. 
	\end{align}
	Next, we apply Bradley's coupling lemma (see Lemma~\ref{Bradley} in the appendix) 
with $X= (U_s^{\varepsilon})_{s \in I_1 \cup \dots \cup I_{i-2}},  Y=Z_{ni}^{\varepsilon \kappa}$ and $q=q_n = ||Z_{n1}^{\varepsilon \kappa}||_{\gamma}/(\sqrt{k_n} b_n)$ for some $\gamma >0$. 
		We obtain the  existence of a random variable $Y^{\ast} = Z_{ni}^{\ast \varepsilon \kappa}$, which is independent of $(U_s^{\varepsilon})_{s \in I_1 \cup \dots \cup I_{i-2}}$, has the same distribution as $Z_{ni}^{ \varepsilon \kappa}$ and satisfies
	\[ 
	\Pro (|Z_{ni}^{\varepsilon \kappa}-Z_{ni}^{\ast \varepsilon \kappa}|>q) \leq 18 \ (\sqrt{k_n}b_n)^{\frac{\gamma}{2 \gamma +1}} \alpha_{\eps}(b_n)^{\frac{2 \gamma}{2 \gamma +1}}.
	\] 
	Thus, we obtain the bound
	\begin{align}\nonumber
	\Exp \big[ |S_{n,\ell,1}^{\varepsilon}| \big] 
	&\leq 
	\frac{1}{k_n} \sum_{i=3}^{k_n} \sum_{p \in \{j-1,j\}} \Exp \Big[ |h_{n,i-2,p}(Z_{ni}^{\varepsilon \kappa})|  \I \big( \varepsilon b_n \kappa > Z_{ni}^{\varepsilon \kappa} \geq \ell \big) \\
	&\hspace{6cm} 
	\times  \I \big( |Z_{ni}^{\varepsilon \kappa} - Z_{ni}^{\ast \varepsilon \kappa}| < q \big) \Big] \nonumber \\
	&\phantom{{}\le{}} 
	+ 36 \ \frac{1}{k_n}  \sum_{i=3}^{k_n} k_n^{-1/2} i (\sqrt{k_n}b_n)^{\frac{\gamma}{2 \gamma +1}} \alpha_{\eps}(b_n)^{\frac{2 \gamma}{2 \gamma +1}, \label{exp_snl1}
    }
	\end{align}
	where the second sum is of the order
	\[ 
	O\Big( k_n^{\frac{1}{2}+\frac{\gamma}{4 \gamma +2}} b_n^{\frac{\gamma (1-2 \eta)}{2 \gamma +1}} \Big) 
	= O \Big( \big( \big(k_n b_n^{-\frac{2 \gamma (2 \eta -1)}{3 \gamma +1}} \big) \big)^{\frac{3 \gamma +1}{4 \gamma +2}} \Big) = o(1) 
	\] 
	by Condition \ref{cond}(ii), choosing $\gamma = \eta/(\eta-2) >0$.
	To bound the first sum, note that for all $x,y \geq 0$ with $y-a \leq x \leq y+a$ for some $a > 0$, we have, for any $p\in\N_{\ge 0}$,
	 \begin{align}
	 \label{eq:hni2}
	 |h_{n,i,p}(x)| 
	 \leq 
	 \max \big\{ |h_{n,i,p}(y+a)|, |h_{n,i,p}((y-a)_+) | \big\} + 2  a \sqrt{k_n}
	 \end{align}
	 where $z_+=\max(z,0)$, 
	 which follows from monotonicity arguments. Indeed, $\tau \le \tau'$ implies $N^{\scs (\tau)}_{b_n,1} \le N^{\scs (\tau')}_{b_n,1}$, whence, for $y+a \ge x\ge y-a\ge 0$,
	\begin{align*}
 0 < h_{n,i,p}(x) 
 & \leq 
 h_{n,i,p}(y-a) +  \sqrt{k_n} \Prob(  N_{b_n,1}^{(y-a)}\le p <  N_{b_n,1}^{(y+a)}) \\
  & \leq 
 h_{n,i,p}(y-a) +  \sqrt{k_n} \Prob(  N_{b_n,1}^{(y+a)} - N_{b_n,1}^{(y-a)}  \ge 1) \\
	& \leq  
	h_{n,i,p}(y-a) +  \sqrt{k_n} \Exp [ N_{b_n,1}^{(y+a)} - N_{b_n,1}^{(y-a)}  ]\\
	&= h_{n,i,p}(y-a) +  2a\sqrt{k_n} ,
	\end{align*}
	where we have used the facts that $N^{\scs (\tau)}_{b_n,1}$ is integer-valued. A similar inequality to the bottom implies \eqref{eq:hni2}.
	As a consequence of  \eqref{eq:hni2}, we may bound the first sum on the right-hand side of \eqref{exp_snl1} by
	\begin{align*}
	\frac{1}{k_n} \sum_{i=3}^{k_n} \sum_{p \in \{j-1,j\}} \Exp \Big[ 
	& 
	\Big\{ |h_{n,i-2,p}(Z_{ni}^{\ast \varepsilon \kappa}+q_n) | + |h_{n,i-2,p}((Z_{ni}^{\ast \varepsilon \kappa}-q_n)_+)| \\
	&+ 2 \|Z_{n1}^{\varepsilon \kappa}\|_{\gamma}/b_n \Big\} \I \big( \varepsilon b_n \kappa  + q_n > Z_{ni}^{\ast \varepsilon \kappa} \geq \ell -q_n \big) \Big].
	\end{align*}
	Now, since $Z_{n1}^\eps/b_n \le 1$ and $q_n \to 0$, we have
	\[ 
	\limsup_{n \to \infty} \|Z_{n1}^{\varepsilon \kappa}\|_{\gamma}/b_n \ \Pro \big( \varepsilon b_n \kappa >Z_{ni}^{\ast \varepsilon \kappa} \geq \ell -q_n \big) 
	\leq \limsup_{n \to \infty} \Pro \big( Z_{ni}^{\varepsilon \kappa} \geq \ell \kappa \big)
	\]
	which converges to 0 as $\ell \to \infty$. Hence, for proving \eqref{eq:snc} with $w=1$, it remains to treat, for $p \in \{j-1,j\}$, 
	\begin{align}
	\frac{1}{k_n} \sum_{i=3}^{k_n} \Exp \Big[ \big\{ |h_{n,i-2,p}((Z_{ni}^{\ast \varepsilon \kappa} \pm q)_+) | \I 
	\big( \varepsilon b_n \kappa  + q_n > Z_{ni}^{\ast \varepsilon \kappa} \geq \ell-q_n \big) \Big]. \label{h_plusminus}
	\end{align}
	We only consider the case with the plus sign. After conditioning on $Z_{ni}^{\scs \ast \varepsilon \kappa}$ we need to bound $\Exp[|h_{n,i-2,v}(x)|]$ for $\ell \le x \le \eps b_n$ (note that $Z_{ni}^{\scs \ast \varepsilon \kappa} + q_n \le \eps b_n \kappa + 2q_n \le  \eps b_n$ for large $n$, since $q_n$ converges to zero). Write $h_{n,i-2,p} = h_{n,i-2,p}^{even} + h_{n,i-2,p}^{odd}$, where $h_{n,i-2,p}^{even}$ and $h_{n,i-2,p}^{odd}$ correspond to the sum over the even and odd blocks in \eqref{eq:hnip}, respectively. Set 
	\[ 
	V_j(x) = \big\{ \I \big( N_{b_n, 2j, \varepsilon}^{(x)} \leq p \big) - \Pro \big( N_{b_n, 2j}^{(x)} \leq p \big) \big\}, 
	\] 
	such that $h_{n,i-2,v}^{even} = k_n^{-1/2} \sum_{j=1}^{\lfloor i/2 \rfloor -1} V_j $. Note that $V_j$ is centered for $\ell \le x \le \eps b_n$. Recursive application of Bradley's coupling lemma (see Lemma~\ref{Bradley}) 
	with some $\gamma > 0, V_1^{\ast}=V_1$ and, in the $j$-th step, $X=(V_1^{\ast},\ldots,V_j^{\ast}), Y=V_{j+1}$ and $q'=q_n' =1/\sqrt{k_n}$
	(note that $\alpha(\sigma(V_j),\sigma(V_{j+1}) \leq \alpha_{\eps}(b_n)$) in combination with Theorem 5.1 in \cite{Bra05} lets us construct an i.i.d. sequence $(V_j^{\ast})_{j \geq 1}$, such that $V_j^{\ast}$ has the same distribution as $V_j$ and 
	\[ 
	\Pro (|V_j-V_j^{\ast}| \geq q_n') \leq 18 \ k_n^{\frac{\gamma}{4 \gamma +2}} \alpha_{\eps}(b_n)^{\frac{2 \gamma }{2 \gamma +1}}.
	\] 
	Note that the i.i.d. sequence $(V_j^{\ast})_{j \geq 1}$ is centered with $|V_j^\ast|\le 1$. As a consequence, 
	\begin{align}
	\Exp [|h_{n,i-2,v}^{even}(x)|]
	& \leq 
	k_n^{-1/2} \Exp \big[ \big| \textstyle{\sum_{j=1}^{\lfloor i/2 \rfloor -1}} V_j^{\ast} \big| \big] + i k_n^{-1/2} \Exp \big[ |V_j-V_j^{\ast}| \big] \nonumber \\
	& \leq 
	(i/k_n)^{1/2}  + ik_n^{-1/2} \big\{ q_n' +  18 \, k_n^{\frac{\gamma}{4 \gamma +2}} \alpha_{\eps}(b_n)^{\frac{2 \gamma }{2 \gamma +1}} \big\} \nonumber \\
	& \leq  
	(i/k_n)^{1/2} + i k_n^{-1} + 18  C \, k_n^{\frac{1}{2}+\frac{\gamma}{4 \gamma +2}} b_n^{-\eta \frac{2 \gamma }{2 \gamma +1}}  \label{h_even}
	\end{align} 
	A similar bound can be obtained for the sum over the odd blocks. Assembling terms, the expression in (\ref{h_plusminus}) can be bounded by 
	\begin{align*}
	&\phantom{{}={}} 
	\Pro \big( Z_{n1}^{\ast \varepsilon \kappa} \geq \ell-q_n \big) 
	\frac{1}{k_n} \sum_{i=3}^{k_n} \Big[ (i/k_n)^{1/2} + i k_n^{-1} + 18 C \, k_n^{\frac{1}{2}+\frac{\gamma}{4 \gamma +2}} b_n^{-\eta \frac{2 \gamma }{2 \gamma +1}}   \Big] \\
	&\lesssim   
	\Pro \big( Z_{n1}\geq \ell/2 \big) 
	\Big\{ 1+ k_n^{\frac{1}{2}+\frac{\gamma}{4 \gamma +2}} b_n^{-\eta \frac{2 \gamma }{2 \gamma +1}}   \Big\},
	\end{align*}
	where 
	\[
	k_n^{\frac{1}{2}+\frac{\gamma}{4 \gamma +2}} b_n^{-\eta \frac{2 \gamma }{2 \gamma +1}} = \big(k_n b_n^{-\frac{4 \eta \gamma}{3 \gamma +1}})^{\frac{3 \gamma +1}{4 \gamma +2}} =o(1) 
	\]
	by Condition \ref{cond}(ii), after setting $\gamma =1$.
	Hence, since $\lim_{n\to\infty}\Prob(Z_{n1} \ge \ell/2)=e^{-\theta\ell/2} \to 0$ for $\ell \to\infty$, we obtain \eqref{eq:snc} and hence \eqref{eq:bnlv2} with $v=1$.
Next, consider \eqref{eq:bnlv2} with $v=2$. By Markov's inequality 
	\begin{align*}
\Prob(|B_{n,\ell,2}|>\delta)
	\leq \delta^{-1} \! \int_{\ell}^{\infty}  \Exp[|e_{n,j}(\tau)|] \ \mathrm{d}H(\tau).
	\end{align*}
	Split the integral on the right-hand side into two integrals over $[\ell, \eps b_n]$ and $(\eps b_n, \infty)$. For $\tau \in [\ell, \eps b_n]$, we have $e_{n,j}(\tau) = f_{n,k_n}(\tau)$, with $f_{n,k_n}$ from \eqref{eq:fn}. Hence, similar as for the treatment of \eqref{h_plusminus}, see in particular relation (\ref{h_even}), we have $\Exp[|f_{n,k_n}(\tau)|] \lesssim  1+o(1)$, where the upper bound is uniform in $\tau$. As a consequence, the integral on the right-hand side of the previous display can be bounded by	
	\[(1+o(1)) \int_{\ell}^{\eps b_n} \,\di H(\tau)  + \sqrt{k_n} \int_{\eps b_n}^\infty \,\di H(\tau),
	\] 
	which converges to zero for $n \to \infty$ followed by $\ell \to \infty$. This proves \eqref{eq:bnlv2} with $v=2$.
\end{proof}

\begin{lem} \label{dj_weakconv}
	For any $m \in \N_{\ge 1}$, 
	\[
	\frac{1}{\sqrt{k_n}} \sum_{i=1}^{k_n} \big(W_{n,i}^{\djb}(1), \ldots, W_{n,i}^{\djb}(m)\big) 
	\wto 
	(s_1^{\djb},\ldots,s_m^{\djb}) \sim \Nor_{m}(0,\Sigma_{m}^{\djb}), 
	\]
	where $W_{n,i}^{\djb}(j)$ and $\Sigma_{m}^\djb = (d_{j,j'}^\djb)_{1\leq j,j' \leq m}$ are defined in \eqref{eq:wni} and \eqref{eq:sigmadb}, respectively. 
\end{lem}

\begin{proof}[Proof of Lemma \ref{dj_weakconv}]
Throughout the proof, we omit the upper index $\djb$ at all instances of $\hat H_n^{\djb}, e_{n,j}^\djb$ and $Z_{ni}^\djb$. 
	Define 
	\[ 
	B_{n,j} = \frac{1}{\sqrt{k_n}} \sum_{i=1}^{k_n} \Big\{ \varphi_{n,j}(Z_{ni}) - \Exp[\varphi_{n,j}(Z_{ni})] \Big\}, \quad j \in \N_{\ge1}. 
	\]
	Decompose each block $I_i =I_i^\djb = I_i^+ \cup I_i^-$, $i = 1, \ldots,k_n$, into a big block $I_i^+ = \{ (i-1)b_n+1,\ldots, ib_n-\ell_n \}$ and a small one $I_i^-=\{ib_n-\ell_n+1,\ldots,ib_n\}$, where $\ell_n$ is from Condition \ref{cond}(ii), and define $Z_{ni}^+ = b_n(1-N_{ni}^+)$ with $N_{ni}^+ = \max\{U_s: s \in I_i^+\}$. 
	Set
	\[ 
	B_{n,j} ^+ = \frac{1}{\sqrt{k_n}} \sum_{i=1}^{k_n} \varphi_{n,j}(Z_{ni}^+) - \Exp[\varphi_{n,j}(Z_{ni}^+)], \quad j \in \N_{\ge1},
	\]
	and write
	 \[ 
	 B_{n,j} ^- = B_{n,j} - B_{n,j} ^+ = \frac{1}{\sqrt{k_n}} \sumki Y_{ni} - \Exp[Y_{ni}],
	 \]
	 where $Y_{ni} = \varphi_{n,j}(Z_{ni}) - \varphi_{n,j}(Z_{ni}^+)$.

	 The same arguments as in the proof of Lemma~9.3 in  \cite{BerBuc18} yield
	 \begin{align} \label{eq:bnjm}
	 B_{n,j} ^- = \op.
	 \end{align} 
    Further, define
	\begin{equation*}
	e_{n,j}^+(\tau) = \sqrt{k_n} \{ p_n^{(\tau),+}(j)- \Pro(N_{b_n,1}^{(\tau),+}=j) \}, \quad j \in \N_{\ge 1},
	\end{equation*}
	where
	\[ 
	p_n^{(\tau),+}(j) = \frac{1}{k_n} \sum_{i=1}^{k_n} \I(N_{b_n,i}^{(\tau),+}=j), \quad 
	N_{b_n,i}^{(\tau),+} = \sum\nolimits_{s \in I_i^+} \I(U_s > 1-\tau/b_n). 
	\]
	A slight adaptation of the proof of Lemma 6.6 in \cite{Rob09} (invoking Lemma 3.9  in \citealp{DehPhi02} instead of Lemma 6.3 in \citealp{Rob09}) shows that, for any $\tau >0$, $\Exp[| e_{n,j}(\tau) - e_{n,j}^+(\tau) |^2] \lesssim (\tau \ell_n/b_n)$ such that 
	\[ 
	\Exp \Big[ \Big| \intnu e_{n,j}(\tau) - e^+_{n,j}(\tau) \, \di H(\tau) \Big|^2 \Big] 
	\lesssim 
	(\ell_n/b_n) \intnu \tau  \, \di H(\tau),  
	\]
	which converges to 0 by Condition \ref{cond}(ii).  This implies that 
	\begin{align} \label{eq:enjm} 
	\intnu e_{n,j}(\tau) \ \mathrm{d}H(\tau) = \intnu e_{n,j}^+(\tau) \ \mathrm{d}H(\tau) + \op. 
	\end{align}
As a consequence of \eqref{eq:bnjm} and \eqref{eq:enjm}, we have
	\begin{align}  \label{eq:bem}
	& k_n^{-1/2} \sum_{i=1}^{k_n} W_{n,i}(j)  
	= \intnu e^+_{n,j}(\tau) \, \di H(\tau) + B^+_{n,j} + \op. 
	\end{align}
	
	Next, define $A_n^+ = \big\{ \min_{i=1,\ldots,k_n} N_{ni}^+ > 1- \varepsilon \big\}$ with $\eps=\eps_1$ from Condition \ref{cond}(ii), such that $\lim_{n\to\infty}\Pro(A_n^+) =1$ by Condition \ref{cond}(iii). Hence,  by \eqref{eq:bem} and the Cram\'er-Wold-device, the lemma is shown once we prove that
	\begin{align}
	& \sum_{j=1}^m \lambda_j  \bigg\{ \intnu e_{n,j}^+(\tau) \, \di H(\tau) + B_{n,j} ^+  \bigg\} \I_{A_n^+} \wto \sum_{j=0}^m \lambda_j s_{j} \label{CramerWold}
	\end{align}
	for arbitrary $\lambda_j \in \R$. For that purpose, rewrite the left-hand side of (\ref{CramerWold})  as $k_n^{-1/2} \sum_{i=1}^{k_n} f_{i,n} \I_{A_n^+}$, where 
	\begin{multline*}
	f_{i,n} = \sum_{j=0}^m \lambda_j \bigg\{ \intnu \I(N_{b_n,i}^{(\tau),+}=j) - \Pro(N_{b_n,1}^{(\tau),+}=j) \, \di H(\tau) \\
	+ \varphi_{n,j}(Z_{ni}^+) - \Exp[\varphi_{n,j}(Z_{ni}^+)] \bigg\}.
	\end{multline*}
	By the definition of $A_n^+$, we have 
	\[ 
	k_n^{-1/2} \sum_{i=1}^{k_n} f_{i,n} \I_{A_n^+} = k_n^{-1/2} \sum_{i=1}^{k_n} \tilde f_{i,n} + \op, 
	\] 
	where $\tilde f_{i,n} = f_{i,n} \I(Z_{ni}^+ < \varepsilon b_n)$. Observing that $\tilde f_{i,n}$ is $\mathcal{B}_{\{(i-1)b_n+1\}:\{ib_n-\ell_n\}}^{\varepsilon}$-measurable, a standard argument based on characteristic functions shows that $\{\tilde f_{i,n}: i=1, \dots, k_n\}$ may be considered independent in the remaining part of this proof.
To obtain asymptotic normality, we apply Ljapunov's central  limit theorem. First, note that $|\tilde f_{1,n}| \le 2 \sum_{j=1}^{m} |\lambda_j| < \infty$. 
	This implies, by stationarity,  for any $p>2$,
	\begin{align*}
	\frac{\sum_{i=1}^{k_n} \Exp[|\tilde f_{i,n}|^p]}{\big\{ \sum_{i=1}^{k_n} \Var(\tilde f_{i,n}) \big\}^{p/2}} 
	&= 
	k_n^{1-p/2}  \frac{\Exp[|\tilde f_{1,n}|^p]}{\Exp[|\tilde f_{1,n}|^2]^{p/2}} 
	\lesssim
	  k_n^{1-p/2} \Exp[\tilde f_{1,n}^2]^{-p/2},
	\end{align*}
	which converges to zero provided that  $\lim_{n \to \infty} \Exp[\tilde f_{1,n}^2]$ exists. The central limit theorem then implies that $k_n^{\scs -1/2} \sum_{i=1}^{k_n} \tilde f_{i,n}$ converges in distribution to a centered normal distribution with  variance $\lim_{n \to \infty} \Exp[\tilde f_{1,n}^2]$,
whence it remain to calculate the latter limit.	
	
	 For that purpose, note that $\lim_{n \to \infty} \Exp[\tilde f_{1,n}^2] = \lim_{n \to \infty} \Exp[f_{1,n}^2]$. Set 
	\begin{align*}
	 C_{n,j} &= \intnu \I(N_{b_n,1}^{(\tau),+}=j) - \Pro(N_{b_n,1}^{(\tau),+}=j) \, \di H(\tau), \\ 
	 D_{n,j} &= \varphi_{n,j}(Z_{n1}^+) - \Exp[\varphi_{n,j}(Z_{n1}^+)],
	\end{align*}
	and note that 
	\begin{align*}
		\Exp[f_{1,n}^2] = \sum_{j,j'=1}^{m} \lambda_j \lambda_{j'} \Exp[(C_{n,j}+D_{n,j}) (C_{n,j'}+D_{n,j'})],
	\end{align*}
	which we need to show to converge to $ \sum_{j,j'=1}^{m} \lambda_j \lambda_{j'} \Exp[s_js_{j'}]$.
	Similar arguments as in the proof of \eqref{eq:bnjm} and \eqref{eq:enjm} allow us to replace $I_1^+$ by $I_1$.
	
 We start by considering the product of the $C_{n,j}$-terms. 
	Invoking the dominated convergence theorem and 
	\[ 
	\Pro(N_{b_n,1}^{(\tau)}=j, N_{b_n,1}^{(\tau')}=j') \to \Pro(N_E^{(\tau)}=j, N_E^{(\tau')}=j') 
	\]
	with $(N_E^{(\tau)}, N_E^{(\tau')})$ as defined in Theorem~\ref{theo:djb} (following from Condition~\ref{cond}(i)), we obtain that
	\begin{align*}
	   \lim_{n\to\infty}
	   \Exp[C_{n,j}C_{n,j'}] 
	   &=  \intnu \intnu \Cov[\I(N_E^{(\tau)}=j) ,\I(N_E^{(\tau')}=j')] \, \di H(\tau) \di H(\tau').
	\end{align*}
		
	Second, we consider the product of the $D_{n,j}$-terms. For this purpose, we first show that $\varphi_{n,j}(Z_{n1})$ converges weakly to $p^{(Z)}(j)$, for $Z \sim \mathrm{Exp}(\theta)$, which in turn is a consequence of weak convergence of $Z_{n1}$ to $Z$ and the extended continuous mapping theorem. For the latter, one needs to prove that $\varphi_{n,j}(x_n) \to p^{(x)}(j)$ for any $x_n \to x$, which follows from
	\begin{align*}
	|\varphi_{n,j}(x_n)-\varphi_{n,j}(x)| & \leq \Exp \big[ | \I(N_{b_n,1}^{(x_n)}=j) - \I(N_{b_n,1}^{(x)}=j)| \big] \\
	& \leq \Exp \big[ \I( |N_{b_n,1}^{(x_n)} - N_{b_n,1}^{(x)}| \geq 1 ) \big] 
	\leq \Exp \big[ |N_{b_n,1}^{(x_n)} - N_{b_n,1}^{(x)}| \big] \\
	&= \Exp \big[ N_{b_n,1}^{(x_n\vee x)} - N_{b_n,1}^{(x_n \wedge x)} \big]
	= |x_n - x|.
	\end{align*}
	Likewise, $\varphi_{n,j}(Z_{n1})\varphi_{j',n}(Z_{n1})$ weakly converges to $p^{(Z)}(j)p^{(Z)}(j')$. Since $|\varphi_{n,j}| \leq 1$, Theorem 2.20 in \cite{Van98} implies convergence of the corresponding moments, i.e., 
	\begin{align*}
		\Exp[D_{n,j}D_{n,j'}] 
		& = \Cov\big( \varphi_{n,j}(Z_{n1}), \varphi_{n,j'}(Z_{n1}) \big) 
		 = \Cov \big( p^{(Z)}(j), p^{(Z)}(j') \big) + o(1).
	\end{align*}
	
	With regard to the mixed $C_{n,j}$- and $D_{n,j'}$-terms, note that, for $j\in\N_{\ge0}$ and $\mu\ge 0$,
	\begin{align*}
	\Pro(N_{b_m,1}^{(\tau)} =j, Z_{n1} >\mu) &= \Pro(N_{b_n,1}^{(\tau)} =j, N_{b_n,1}^{(\mu)} =0)  \\
	&\to 
		\begin{cases}
	p_2^{(\tau,\mu)}(j,0)  & ,\tau \geq \mu \ge 0 \\
	e^{-\theta \mu} \ind(j=0) & ,\mu> \tau \ge 0,
	\end{cases}  
	\end{align*}
	such that $(N_{b_n,1}^{(\tau)},Z_{1:n}) \wto (N_E^{(\tau)},Z)$ with $(N_E^{(\tau)},Z)$ as specified in Theorem~\ref{theo:djb}. The extended continuous mapping theorem and boundedness, $|\varphi_{n,j}| \le 1$, implies
	\begin{align*}
		\Exp [C_{n,j}D_{n,j'}] 
		&= \intnu \Cov\big\{ \I(N_{b_n,1}^{(\tau)}=j), \varphi_{n,j'}(Z_{n1}) \big\} \, \di H(\tau) \\
		&= \ \intnu \Cov\big\{ \I(N_{E}^{(\tau)}=j), p^{(Z)}(j') \big\} \, \di H(\tau)  +o(1).
	\end{align*}
	
	The last three paragraphs imply
	\begin{align*}
	 \lim_{n\to\infty} \Exp[(C_{n,j}+D_{n,j}) (C_{n,j'}+D_{n,j'})]
	 =
	 d_{j,j'}
	\end{align*}
	with $d_{j,j'}=d_{j,j'}^\djb$ from \eqref{eq:sigmadb}, which finalizes the proof.
\end{proof}

%%%%%%%%%%%%%%%%%%%%%%%%%%%%%%%%%%%%%%%%%%%%%%%%%%%%%%%%%%%%%%%%%%%%%%%%%%%%%%%%%%%%%%%%%%%%%%%%%%%%

\section{Auxiliary lemmas - Sliding blocks}

Throughout, we assume that Condition \ref{cond} is met and that, additionally, $\sqrt{k_n} \beta_{\varepsilon_2}(b_n)=o(1)$ for some $\varepsilon_2 >0$. All convergences are for $n \to \infty$ if not stated otherwise. We will also occasionally omit the upper index $\slb$ at $\hat H_n^{\slb}, e_{n,j}^\slb$ and $Z_{ni}^\slb$.

\begin{lem} \label{sl_ejn_integral}
	For any $j \in \N_{\ge 1}$,
	\[ \intnu e_{n,j}^{\slb}(\tau) \ \mathrm{d}(\hhutsl -H)(\tau) = \op. \]
\end{lem}

\begin{proof}[Proof of Lemma \ref{sl_ejn_integral}]
The proof is very similar to the one of Lemma \ref{dj_ejn_integral}. In fact, we need to show that
\eqref{eq:anl2} and \eqref{eq:bnlv2} is met (for $v=1,2$), where $A_{n,\ell}$ and $B_{n,\ell,v}$ are defined as in \eqref{eq:anl1} and \eqref{eq:bnlv1}, but with $\hat H_n= \hat H_n^{\slb}$ and $e_{n,j}=e_{n,j}^\slb$. 

	Invoking Theorem~\ref{weak_conv} instead of Theorem 4.1 in \cite{Rob09}, the proof of \eqref{eq:anl2} is the same as in the proof  Lemma \ref{dj_ejn_integral}.
	
Regarding \eqref{eq:bnlv2} with $v=1$, write
	\begin{align*}
	B_{n,\ell,1} 
	&=  \frac{\sqrt{k_n}}{(n-b_n+1)^2} \sum_{i,i'=1}^{n-b_n+1} \Big\{ \I \big( N_{b_n,i'}^{(Z_{ni})} = j \big)  - \varphi_{n,j}(Z_{ni}) \Big\} \I (Z_{ni} \geq \ell) \\
	&=  
	k_n^{-3/2} \! \sum_{i,i'=1}^{k_n-1} b_n^{-2} \sum_{s \in I_i} \sum_{s' \in I_{i'}} \Big\{ \I \big( N_{b_n,s'}^{(Z_{ns})} = j \big)  - \varphi_{n,j}(Z_{n1}) \Big\}  \\
	&\hspace{7cm} \times  \I (Z_{ns} \geq \ell)+o(1) \\
	&=  V_{n,\ell,1} + V_{n,\ell,2} + o(1),
	\end{align*}
	where $V_{n,\ell,v}$ is made up from the same summands as in the line before, with the only difference that for $v=1$
	the sum over $i'$ ranges from 1 to $i-3$, and for $v=2$ it goes from $i+3$ to $k_n-1$ (c.f.\ the proof of Lemma~\ref{dj_ejn_integral}).
	
	Now, for $\varepsilon=\eps_1$ from Condition \ref{cond}(ii), write $c$ from Condition~\ref{cond}(iii) as $c=1-\eps \kappa$ with $\kappa \in (0,1)$, and let $C_n = \{ \min_{i=1,\ldots,n-b_n+1} N_{ni} > 1 - \varepsilon \kappa \}$, such that $\Pro(C_n) \to 1$ as $n \to \infty$ by Condition \ref{cond}(iii). As in the proof of Lemma \ref{dj_ejn_integral}, one can now show that 
	\[ 
	\lim_{\ell \to \infty} \limsup_{n \to \infty} \Pro \big( |V_{n,\ell,v} \I_{C_n}| > \delta \big) = 0
	\]
	for $v\in\{1,2\}$. 
	Likewise, as for the process $e_{n,j}^\djb$ in the disjoint blocks setting,  we obtain the bound $\Exp[|e_{n,j}^{\slb}(\tau)|]=O(1)$ uniformly in $\tau$, such that 
	\[ 
	\lim_{\ell \to \infty} \limsup_{n \to \infty} \Exp \Big[ \Big| \int_{\ell}^{\infty} e^{\slb}_{j,n}(\tau) \ \mathrm{d}H(\tau) \Big| \Big] 
	\lesssim \lim\limits_{\ell \to \infty}  e^{-\theta \ell} = 0. 
	\] This concludes the proof. 
\end{proof}

\begin{lem} \label{sl_weakconv}
		For any $m \in \N_{\ge 1}$, 
	\[
	\frac{\sqrt{k_n}}{n-b_n+1} \sum_{i=1}^{n-b_n+1} (W^{\slb}_{n,i}(1), \ldots, W^{\slb}_{n,i}(m)\big) 
	\wto 
	(s_1^{\slb},\ldots,s_m^{\slb}) \sim \Nor_{m}(0,\Sigma_{m}^{\slb}), 
	\]
	where $W_{n,i}^{\slb}(j)$ is defined as in \eqref{eq:wni} but with $\djb$ replaced by $\slb$, and where $\Sigma_{m}^\slb = (d_{j,j'}^\slb)_{1\leq j,j' \leq m}$ is defined in \eqref{eq:sigmasl}.
	\end{lem}

\begin{proof}[Proof of Lemma \ref{sl_weakconv}]	By the Cramér-Wold device it suffices to show that 
	\begin{align*}
	& \sum_{j=1}^m \lambda_{j} \frac{\sqrt{k_n}}{n-b_n+1} \sum_{i=1}^{n-b_n+1} W^{\slb}_{n,i}(j) 
	\wto \sum_{j=1}^m \lambda_{j} s_j^{\slb}
	\end{align*}
	for arbitrary $\lambda_{j} \in \R$. Write the right-hand side as
	\begin{align*}
  	\sum_{i=1}^{k_n-1} \sum_{s \in I_i} \sum_{j=1}^m \lambda_{j} \frac{\sqrt{k_n}}{n-b_n+1} \Big\{ 
  	& \intnu \I(N_{b_n,s}^{(\tau)}=j) - \varphi_{n,j}(\tau) \, \di H(\tau) \nonumber \\
  	& + \varphi_{n,j}(Z_{ns}) - \Exp\big[ \varphi_{n,j}(Z_{ns}) \Big\} + \op, 
	\end{align*}
	where the small $o_\Prob(1)$ term is due to the fact that a negligible  number of summands has been omitted. To take care of the serial dependence of the sliding blocks, we apply a similar construction as in the proof of Lemma \ref{fidis}. Using the same notation as in that proof, write $V_n^{\pm} = (q_n^{\ast})^{-1/2} \sum_{i=1}^{\scs q_n^{\ast}} T_{ni}^{\pm}$ with
	\begin{multline*}
	T_{ni}^{\pm} = \sqrt{\frac{q_n^{\ast}}{k_n}} \sum_{s \in J_i^{\pm}} \sum_{j=1}^m \lambda_{j} \frac{n}{n-b_n+1} \frac{1}{b_n} 
	 \Big\{  \intnu \I(N_{b_n,s}^{(\tau)}=j)  - \varphi_{n,j}(\tau) \, \di H(\tau) \\
	 + \varphi_{n,j}(Z_{ns}) - \Exp\big[ \varphi_{n,j}(Z_{ns}) \big] \Big\}. 
	\end{multline*}
	Since 
	\[ 
	\Big| \intnu \I(N_{b_n,s}^{(\tau)}=j) - \varphi_{n,j}(\tau)\, \di H(\tau) \Big|
	+ 
	\big| \varphi_{n,j}(Z_{ns}) - \Exp\big[ \varphi_{n,j}(Z_{ns}) \big] \big| \leq 2, 
	\]
	we still obtain the upper bound in (\ref{Tn1_minus1}), and the remaining proof is the same as in Lemma \ref{fidis}.
   In particular, note that 
   \[ 
   T_{n1}^+ = \sum_{j=1}^m \lambda_{j} \frac{\sqrt{k^{\ast}_n}}{n^{\ast}-b_n+1} \sum_{i=1}^{n^{\ast}-b_n+1} W_{n^{\ast},i}(j) + R_n,
   \] 
   where $R_n \to 0$ in $L_2(\Prob)$ and $n^{\ast} = k_n^{\ast} b_n$, and that our assumptions in Condition \ref{cond} still hold if $n$ and $k_n$ are substituted by $n^{\ast}$ and $k^{\ast}_n$. The assertion then follows from Lemma~\ref{Sliding_OtherCov} below. 
\end{proof}

\begin{lem} \label{Sliding_OtherCov}
    For any $j,j' \in \N_{\ge 1}$, we have 
    \begin{align*}
   \lim_{n\to\infty}
    \Cov \Big( \frac{\sqrt{k_n}}{n-b_n+1} \sum_{i=1}^{n-b_n+1} W_{n,i}^{\slb}(j), \frac{\sqrt{k_n}}{n-b_n+1} \sum_{i=1}^{n-b_n+1} W_{n,i}^{\slb}(j') \Big) = d_{j,j'}^\slb,
    \end{align*}
    where $d_{j,j'}^\slb$ is defined in \eqref{eq:sigmasl}.
\end{lem}

\begin{proof}[Proof of Lemma \ref{Sliding_OtherCov}]	
Assume that all $U_s$ are $\mathcal{B}^{\varepsilon}_{s:s}$-measurable with $\eps=\eps_1$ from Condition \ref{cond}; the general case can be treated by multiplying with suitable indicator functions as in the previous proofs.
	Write 
		\begin{align}
		&\phantom{{}={}} \Cov \Big( \frac{\sqrt{k_n}}{n-b_n+1} \sum_{i=1}^{n-b_n+1} W_{n,i}^{\slb}(j), \frac{\sqrt{k_n}}{n-b_n+1} \sum_{i=1}^{n-b_n+1} W_{n,i}^{\slb}(j') \Big) \nonumber\\
		&= C_{n1} + C_{n2} + C_{n3} + C_{n4}, \label{covariances}
				\end{align}
				where
	\begin{align*}
		C_{n1} &=  
		\Cov \Big( \intnu e_{n,j}(\tau) \, \di H(\tau), \intnu e_{n,j'}(\tau) \, \di H(\tau) \Big)	\\
		C_{n2} &=  
		\frac{k_n}{(n-b_n+1)^2} \sum_{i,i'=1}^{n-b_n+1} \Cov \big( \varphi_{n,j}(Z_{ni}), \varphi_{n,j'}(Z_{ni'}) \big)\nonumber	\\
		C_{n3} &=  
		\frac{k_n}{(n-b_n+1)^2} \sum_{i,i'=1}^{n-b_n+1} \Cov \Big( \intnu \I (N_{b_n,i}^{(\tau)} = j ) \, \di H(\tau), \varphi_{n,j'}(Z_{ni'}) \Big) \nonumber \\
		C_{n4} &= 
		\frac{k_n}{(n-b_n+1)^2} \sum_{i,i'=1}^{n-b_n+1} \Cov \Big( \intnu \I (N_{b_n,i}^{(\tau)} = j' ) \, \di H(\tau), \varphi_{n,j}(Z_{ni'}) \Big). \nonumber			
	\end{align*}
	By Lemma \ref{Sliding_Cov}, the first term $C_{n1}$ satisfies
	\begin{multline} \label{eq:ccn1}
	\lim_{n\to\infty}C_{n1}
	 =   
	 2 \intne   \intnu \intnu   \Cov \big( \I(X_{1,\xi}^{(\tau)}=j), \\ 
		\I(Y_{1,\xi}^{(\tau')}=j') \big) \, \di H(\tau) \di H(\tau') 
	\end{multline}
	As at the beginning of the proof of Lemma \ref{Sliding_Cov}, the second term can be shown to satisfy $C_{n2}=T_{n1}+T_{n2}+o(1)$, where
	\begin{align*}
		T_{n1} &= \frac{1}{b_n^2} \sum_{s,t \in I_1} \Cov \big( \varphi_{n,j}(Z_{ns}), \varphi_{n,j'}(Z_{nt}) \big), \\
		T_{n2} &= \frac{1}{b_n^2} \sum_{s \in I_1} \sum_{t \in I_2} \Cov \big( \varphi_{n,j}(Z_{ns}), \varphi_{n,j'}(Z_{nt}) \big) \\
		&\hspace{3cm }+  \frac{1}{b_n^2} \sum_{s \in I_2} \sum_{t \in I_1} \Cov \big( \varphi_{n,j}(Z_{ns}), \varphi_{n,j'}(Z_{nt}) \big).
	\end{align*}
	Let us start with $T_{n1}$. We know that $\Exp[\varphi_{n,j'}(Z_{nt})] \to \bar{p}(j')$ by the proof of Lemma \ref{dj_weakconv}, which implies 
	\[ 
	T_{n1} = \frac{1}{b_n^2} \sum_{s,t=1}^{b_n} \Exp [ \varphi_{n,j}(Z_{ns}) \varphi_{n,j'}(Z_{nt}) ] - \bar{p}(j) \bar{p}(j') + o(1). 
	\]
	As in the proof of Lemma \ref{Sliding_Cov} we can write 
	\begin{align*}
		\frac{1}{b_n^2} \sum_{s,t=1}^{b_n} \Exp [ \varphi_{n,j}(Z_{ns}) \varphi_{n,j'}(Z_{nt}) ] &= \intne \intne g_n(\xi,z) \, \di z \di \xi, 
	\end{align*}
	where
	\begin{align*}
		g_n(\xi,z) &= \Exp [ \varphi_{n,j}(Z_{n, \lfloor b_nz \rfloor +1}) \varphi_{n,j'}(Z_{n, \lfloor b_n \xi \rfloor +1}) ].
	\end{align*}
	Let $z \leq \xi$. Set $r_n = \lfloor b_n \xi \rfloor - \lfloor b_n z \rfloor$. For $x,y >0$ consider
	\begin{align*}
		&\phantom{{}={}} 
		\Pro \big( Z_{n, \lfloor b_nz \rfloor +1} > x, Z_{n, \lfloor b_n \xi \rfloor +1}> y \big) \\
		&=  
		\Pro \big( N_{1:b_n} < 1-\tfrac{x}{b_n}, \ N_{r_n+1:r_n+b_n} < 1-\tfrac{y}{b_n} \big) \\
		&=  
		\Pro \big( N_{1:r_n} < 1-\tfrac{x}{b_n}, \ N_{r_n+1:b_n} < 1- \tfrac{x \vee y}{b_n}, \ N_{b_n+1:r_n+b_n} < 1-\tfrac{y}{b_n} \big)
	\end{align*}
	where $N_{s:t} = \max(U_s, \dots, U_t)$ for  $s, t \in \N_{\ge 1}$ with $s\le t$.
	Note that $\Pro(N_{1:q_n}>1-z/b_n) \leq zq_n/b_n \to 0$ for any integer sequence $q_n=o(b_n)$ that  is converging to infinity. Similar as in the step (\ref{manipulation}) in the proof of Lemma \ref{Sliding_Cov}, this implies that the expression in the previous display equals
	\begin{multline*} 
	\Pro \big( N_{1:r_n} < 1-\tfrac{x}{b_n}\big) \ \Pro \big( N_{r_n+1:b_n} < 1-  \tfrac{x \vee y}{b_n} \big) \ \Pro \big( N_{b_n+1:r_n+b_n} < 1-\tfrac{y}{b_n} \big) \\
	 +  O(\alpha_{\eps}(q_n)) + O((x \vee y)q_n/b_n),
	\end{multline*}
	which by (\ref{z_conv}) converges to 
	\begin{align*}
		H_{\xi-z}(x,y) := \exp \big( -\theta \{ (x \wedge y) (\xi-z) + (x \vee y) \} \big).
	\end{align*}
	As a consequence, by the definition of $(X_{2,\xi-z},  Y_{2,\xi-z})$ in Theorem~\ref{theo:slb}, 
\[ 
(Z_{n,\lfloor b_n z \rfloor +1}^{\slb}, Z_{n,\lfloor b_n \xi \rfloor +1}^{\slb}) \wto (X_{2,\xi-z},Y_{2,\xi-z}),
\]
As in the proof of Lemma \ref{dj_weakconv}, the extended continuous mapping theorem and Theorem 2.20 in \cite{Van98} imply 
\[
\lim_{n\to\infty} g_n(z,\xi) = \Exp[p^{(X_{2,\xi-z})}(j) p^{(Y_{2,\xi-z})}(j')]
\] 
for $z \leq \xi$. By symmetry, for $z > \xi$,
\[
\lim_{n\to\infty} g_n(z,\xi) = \Exp[p^{(X_{2,z-\xi})}(j) p^{(Y_{2,z-\xi})}(j')]
\] 
A simple calculation then shows that 
\[ 
\lim_{n\to\infty} \intne \intne g_n(\xi,z) \, \di z \di \xi 
=
2 \intne (1-\xi) \Exp [p^{(X_{2,\xi})}(j) p^{(Y_{2,\xi})}(j')] \, \di \xi,  
\] 
Altogether, we have that
\[ 
\lim_{n\to\infty} T_{n1} =
2 \intne (1-\xi) \Exp \big[ p^{(X_{2,\xi})}(j) p^{(Y_{2,\xi})}(j') \big] \, \di \xi - \bar{p}(j) \bar{p}(j'). 
\] 
Analogously, the term $T_{n2}$ can be seen to satisfy
\[ 
\lim_{n\to\infty} T_{n2} 
= 
2 \intne \xi \Exp \big[ p^{(X_{2,\xi})}(j) p^{(Y_{2,\xi})}(j') \big] \, \di \xi - \bar{p}(j) \bar{p}(j'), 
\] 
such that 
\begin{align} 
	C_{n2} = T_{n1} + T_{n2} +o(1) 
& \nonumber
	\to 2 \intne \! \Exp \big[ p^{(X_{2,\xi})}(j) p^{(Y_{2,\xi})}(j') \big] \, \di \xi - 2 \bar{p}(j) \bar{p}(j') \\
&= \label{eq:ccn2}
	2 \intne  \Cov \big( p^{(X_{2, \xi})}(j), p^{(Y_{2, \xi})}(j') \big) \, \di \xi.
\end{align}

	Next, consider $C_{n3}$ in (\ref{covariances}), which may be written as
$
	C_{n3} = S_{n1} + S_{n2} + o(1),
$
	where
\begin{align*}
	S_{n1} 
&= 
	\frac{1}{b_n^2} \sum_{s,t \in I_1} \Cov \bigg( \intnu \I(N_{b_n,s}^{(\tau)}=j) \, \di H(\tau), \varphi_{n,j'}(Z_{nt}) \bigg) \\
	S_{n2}
&=  
	\frac{1}{b_n^2} \bigg\{ \sum_{s \in I_1} \sum_{t \in I_2} \Cov \bigg( \intnu \I(N_{b_n,s}^{(\tau)}=j) \, \di H(\tau), \varphi_{n,j'}(Z_{nt}) \bigg) \\
	&  \hspace{1.5cm} + \sum_{s \in I_2} \sum_{t \in I_1} \Cov \bigg( \intnu \I(N_{b_n,s}^{(\tau)}=j) \, \di H(\tau), \varphi_{n,j'}(Z_{nt}) \bigg) \bigg\}.
	\end{align*}
By similar arguments as before, we obtain
	\begin{multline*} 
	S_{n1}  
	= 
	\intne \intne  \intnu \Exp \Big[ \I\big( N_{b_n,\lfloor b_n z \rfloor +1}^{(\tau)} = j \big) \varphi_{n,j'}(Z_{n,\lfloor b_n \xi \rfloor+1}) \Big]  \, \di H(\tau) \, \di z \, \di \xi \\
	- \bar{p}(j)\bar{p}(j') + o(1). 
	\end{multline*} 

	To analyze the convergence of the product moment in the previous display we start by showing that  	
	\[ 
	\big( N_{b_n,\lfloor b_n z \rfloor +1}^{(\tau)}, Z_{n,\lfloor b_n \xi \rfloor+1} \big) 
	\wto 
	\big( X_{3,|\xi-z|}^{(\tau)}, Y_{3,|\xi-z|} )
	\] 
	where $(X_{3,\zeta}^{(\tau)}, Y_{3,\zeta})$ is defined in Theorem~\ref{theo:slb}. For $x>0, j \in \N_{\ge 0}$ and $0\le z \leq \xi \le 1$, write
	\begin{align*}
	&\phantom{{}={}} 
	\Pro \big(N_{b_n,\lfloor b_n z \rfloor +1}^{(\tau)}= j, Z_{n,\lfloor b_n \xi \rfloor+1} > x \big) 
	=  
	\Pro \big(N_{b_n,\lfloor b_n z \rfloor +1}^{(\tau)}= j, N_{b_n,\lfloor b_n \xi \rfloor +1}^{(x)}
	=0 \big)
	\end{align*}
	which is exactly of the form of $\varphi_n$ in \eqref{eq:defphi} and hence converges to 
	\begin{align*}
&\phantom{{}={}} H_{0,j}^{(x,\tau)}(\xi-z)\I(x \le \tau)  + 	H_{j,0}^{(\tau, x)}(\xi-z) \I(x > \tau) \\
	& =
	\sum_{l=0}^j p^{(\tau(\xi-z))}(l) p^{(x (\xi-z))}(0) p_2^{((1-\xi+z)\tau, (1-\xi+z)x)}(j-l,0)  \I(x \leq \tau) \\
	& \hspace{3cm} + p^{(\tau(\xi-z))}(j) e^{-\theta x} \I(x > \tau)\\
	&=
	\Pro \big( X_{3,\xi-z}^{(\tau)} =j, Y_{3,\xi-z} >x \big)  
	\end{align*}
	by the proof of Lemma \ref{Sliding_Cov}, where the last equation follows from the definition of   $(X_{3,\zeta}^{\scs (\tau)}, Y_{3,\zeta})$ in Theorem~\ref{theo:slb}. 
 The same arguments as before in combination with the dominated convergence theorem implies
	\begin{multline*}  
	\lim_{n\to\infty} 
	\Exp \Big[ \I\big( N_{b_n,\lfloor b_n z \rfloor +1}^{(\tau)} = j \big) \varphi_{n,j'}(Z_{n,\lfloor b_n \xi \rfloor+1}) \Big] \\                                     
	  = \Exp \Big[ \I \big( X_{3,\xi-z}^{(\tau)} = j \big) p^{(Y_{3,\xi-z})}(j') \Big].
	\end{multline*} 
	For the case $z > \xi$, one obtains the same limiting expression, but with $z$ and $\xi$ interchanged. As a consequence, by similar arguments as for $C_{n2}$, 
	\begin{align*}
	&\phantom{{}={}} \lim_{n\to\infty} S_{n1}  \\
	&= 
	\intne \intne  \intnu \Exp \Big[ \I \big( X_{3,\xi-z}^{(\tau)} = j \big) p^{(Y_{3,\xi-z})}(j') \Big]   \, \di H(\tau) \, \di z \, \di \xi - \bar{p}(j)\bar{p}(j') \\
	&=
	 2 \intne \int_0^\infty (1-\xi) \Exp \Big[\I\big( X_{3\xi}^{(\tau)} = j \big) p^{(Y_{3\xi})}(j') \Big]  \, \di H(\tau) \, \di \xi- \bar{p}(j) \bar{p}(j').
	\end{align*} 
	A similar argumentation for $S_{n2}$ finally implies
	\begin{align}
	C_{n3} 
	&= \nonumber
	S_{n1}+S_{n2}  + o(1) \\
	&\to  \nonumber
	2 \intne \int_0^\infty  \Exp \Big[\I\big( X_{3\xi}^{(\tau)} = j \big) p^{(Y_{3\xi})}(j') \Big]  \, \di H(\tau) \, \di \xi- \bar{p}(j) \bar{p}(j') \\
	&= \label{eq:ccn3}
	2 \intne \intnu \Cov \Big( \I\big( X_{3,\xi}^{(\tau)} = j \big), p^{(Y_{3,\xi})}(j') \Big)  \, \di H(\tau) \, \di \xi, 
	\end{align} 
	where we have used that $X_{3,\xi} \sim p^{(\tau)}$ and $Y_{3,\xi} \sim \mathrm{Exponential}(\theta)$. 
	The assertion is a consequence of (\ref{covariances}) and \eqref{eq:ccn1}, \eqref{eq:ccn2}, \eqref{eq:ccn3}, and the fact that $C_{n4}$ has the same limit as $C_{n3}$, but with interchanged roles of $j$ and $j'$.
\end{proof}

%%%%%%%%%%%%%%%%%%%%%%%%%%%%%%%%%%%%%%%%%%%%%%%%%%%%%%%%%%%%%%%%%%%%%%%%%%%%%%%%%%%%%%%%%%%%%%%%%%%%%

\section{Further auxiliary results}

\begin{lem}[\citealp{Bra83}] \label{Bradley}
	If $X$ and $Y$ are two random variables in some Borel space $S$ and $\R$, respectively, if $U$ is uniform on $[0,1]$ and independent of $(X,Y)$ and if $q>0$ and $\gamma >0$ are such that $q \leq ||Y||_{\gamma} = \Exp[|Y|^{\gamma}]^{1/\gamma}$, then there exists a measurable function $f$ such that $Y^{\ast}=f(X,Y,U)$ has the same distribution as $Y$, is independent of $X$ and satisfies \[ \Pro (|Y-Y^{\ast}| \geq q) \leq 18 (||Y||_{\gamma}/q)^{\gamma/(2\gamma +1)} \alpha(\sigma(X),\sigma(Y))^{2\gamma/ (2\gamma +1)}. \]
\end{lem}

\begin{lem}[\citealp{Ber79}] %\label{Berbee}
	If $X$ and $Y$ are two random variables in some Borel spaces $S_1$ and $S_2$, respectively, then there exists a random variable $U$ independent of $(X,Y)$ and a measurable function $f$ such that $Y^{\ast} = f(X,Y,U)$ has the same distribution as $Y$, is independent of $X$ and satisfies $\Pro(Y \neq Y^{\ast}) = \beta(\sigma(X),\sigma(Y))$. 
\end{lem}

\section{Further simulation results} \label{more_figures}

This section contains simulation results for the ARMAX- and AR-model described in Section~\ref{sec:sim}, see Figure~\ref{Fig:maxAR_Var}-\ref{Fig:AR_MSE_all}.

\begin{figure} [t!]
	\begin{center}
		\includegraphics[width=\textwidth]{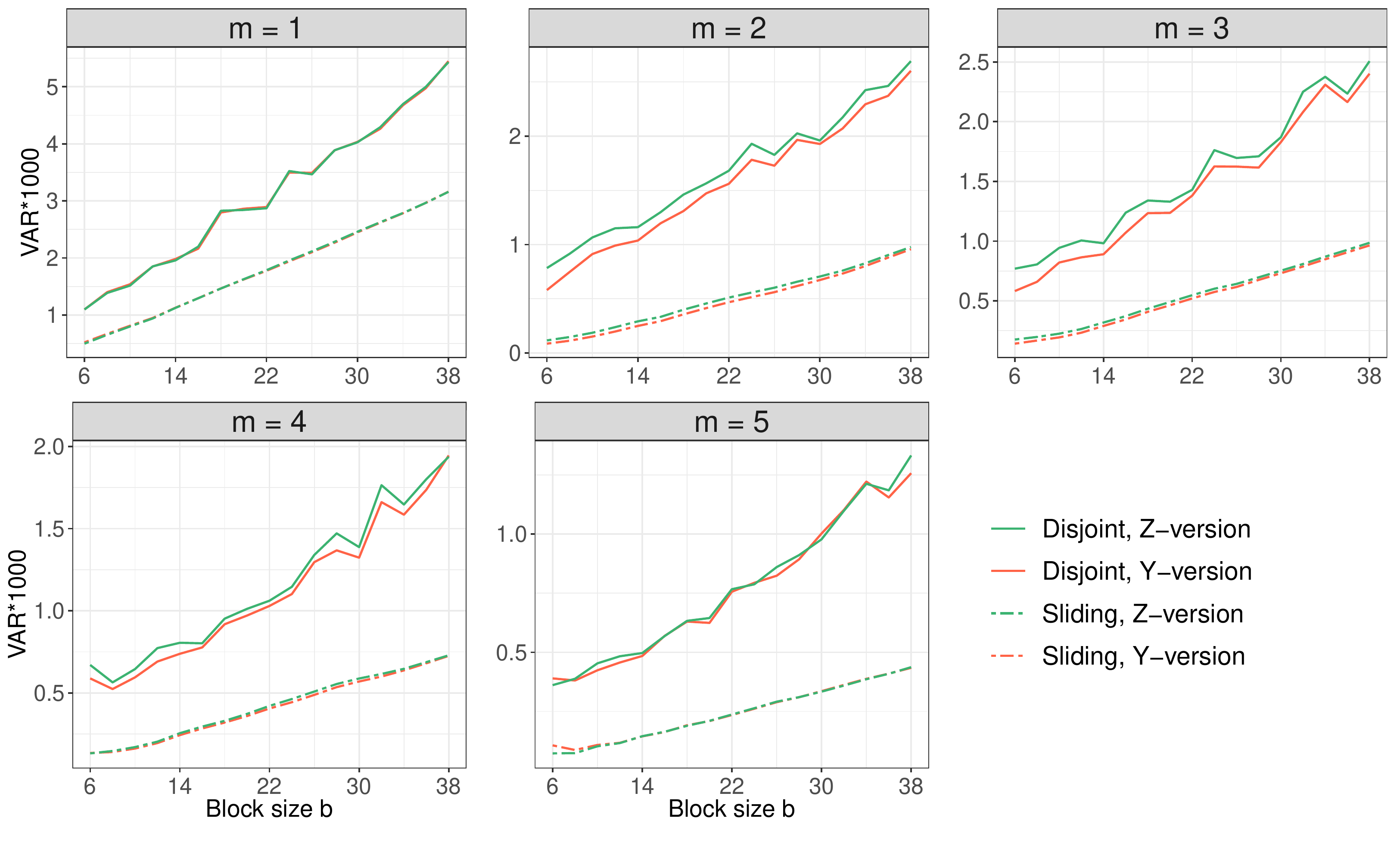} \vspace{-.8cm}
		\caption{Variance multiplied by $10^3$ for the estimation of $\pi(m)$ within the ARMAX-model for $m=1,\ldots,5$.} 
		\vspace{-.3cm} 
		\label{Fig:maxAR_Var}
	\end{center}
\end{figure}

\begin{figure} [t!]
	\begin{center}
		\includegraphics[width=\textwidth]{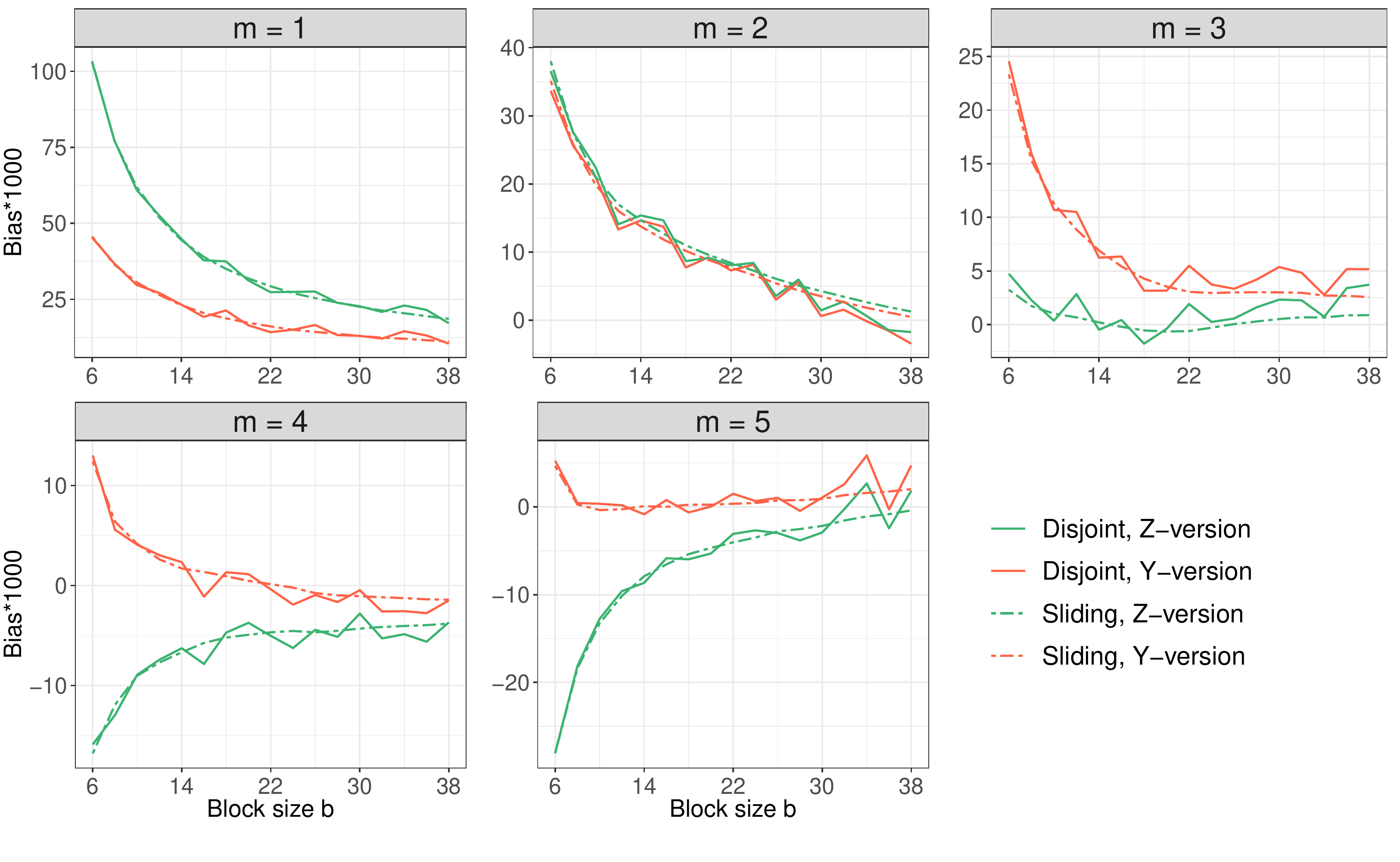} \vspace{-.8cm}
		\caption{Bias multiplied by $10^3$ for the estimation of $\pi(m)$ within the ARMAX-model for $m=1,\ldots,5$.} 
		\vspace{-.3cm} 
		%\label{Fig:maxAR_Bias}
	\end{center}
\end{figure}

\begin{figure} [t!]
	\begin{center}
		\includegraphics[width=\textwidth]{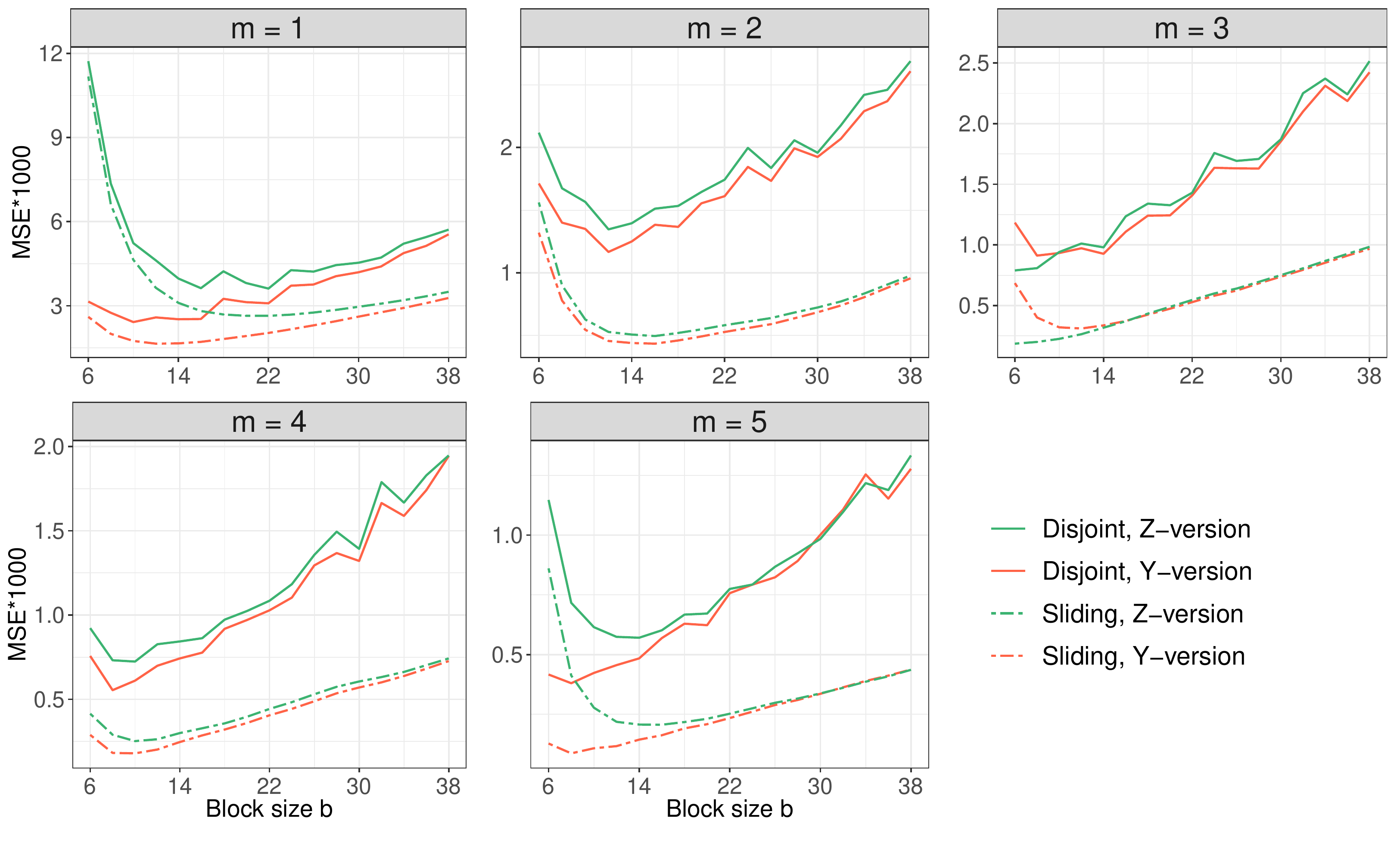} \vspace{-.8cm}
		\caption{Mean squared error multiplied by $10^3$ for the estimation of $\pi(m)$ within the ARMAX-model for $m=1,\ldots,5$.} 
		\vspace{-.3cm} 
		%\label{Fig:maxAR_MSE}
	\end{center}
\end{figure}

\begin{figure} [t!]
	\begin{center}
		\includegraphics[width=\textwidth]{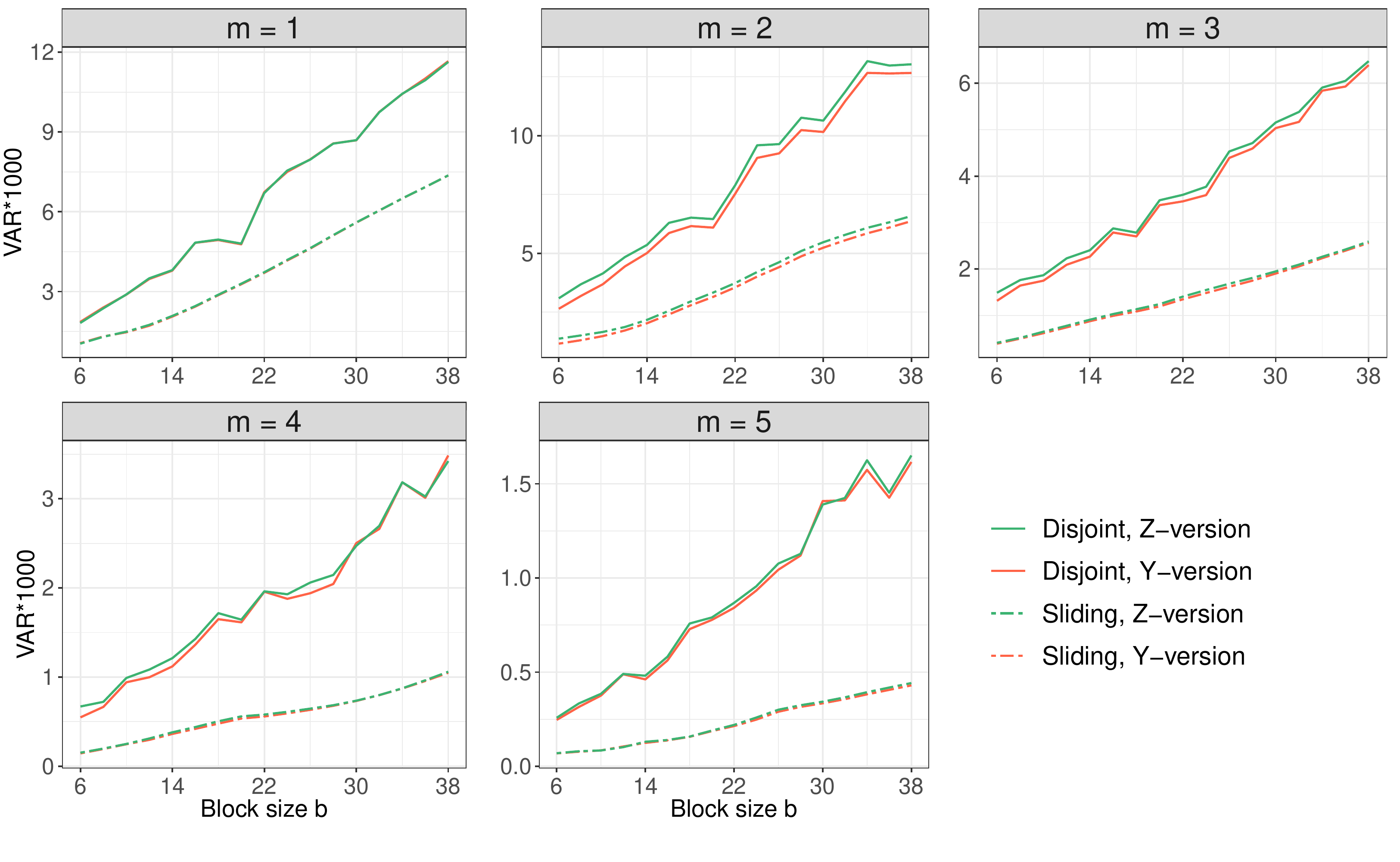} \vspace{-.8cm}
		\caption{Variance multiplied by $10^3$ for the estimation of $\pi(m)$ within the AR-model for $m=1,\ldots,5$.} 
		\vspace{-.3cm} 
		%\label{Fig:AR_Var}
	\end{center}
\end{figure}

\begin{figure} [t!]
	\begin{center}
		\includegraphics[width=\textwidth]{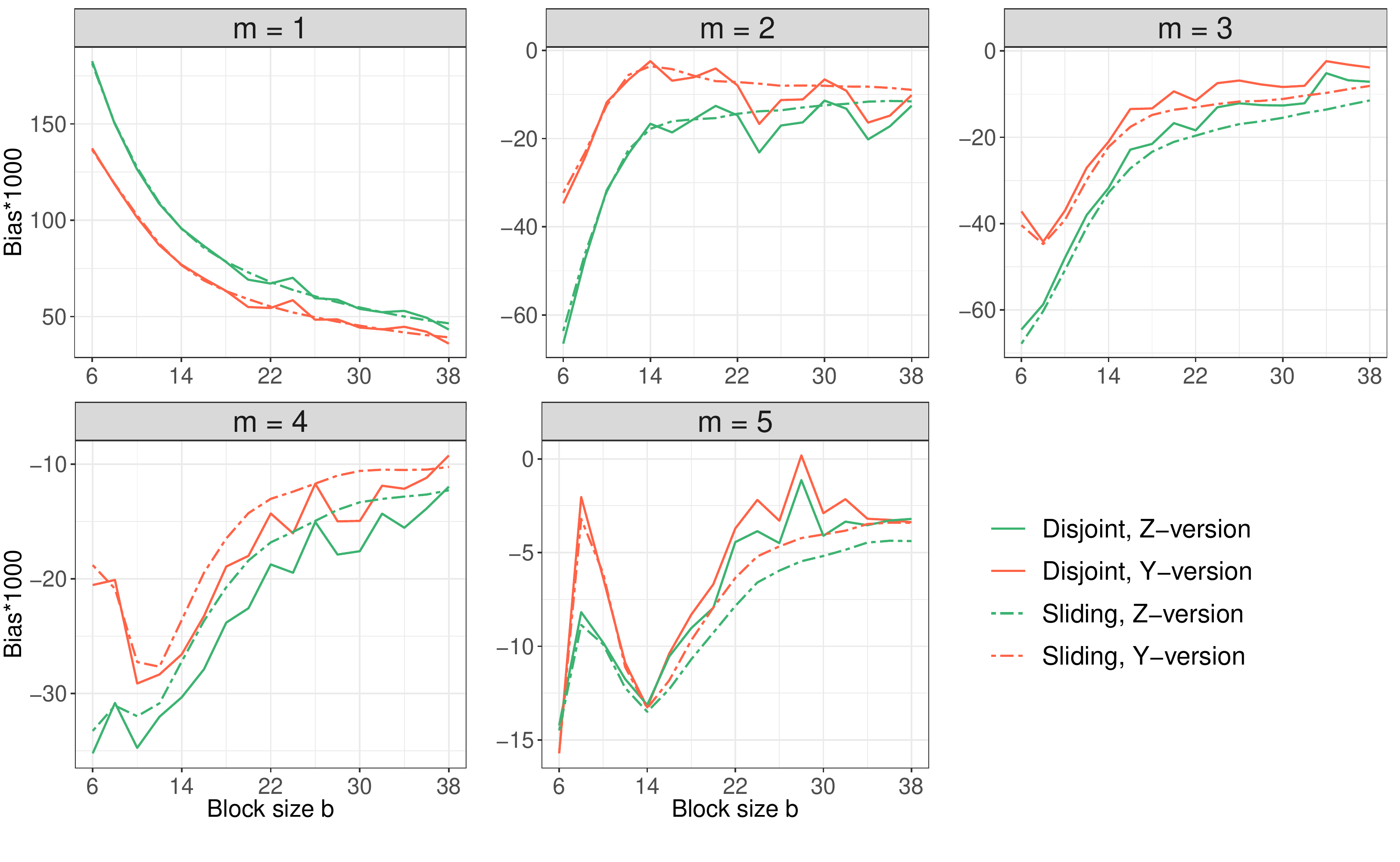} \vspace{-.8cm}
		\caption{Bias multiplied by $10^3$ for the estimation of $\pi(m)$ within the AR-model for $m=1,\ldots,5$.} 
		\vspace{-.3cm} 
		%\label{Fig:AR_Bias}
	\end{center}
\end{figure}

\begin{figure} [t!]
	\begin{center}
		\includegraphics[width=\textwidth]{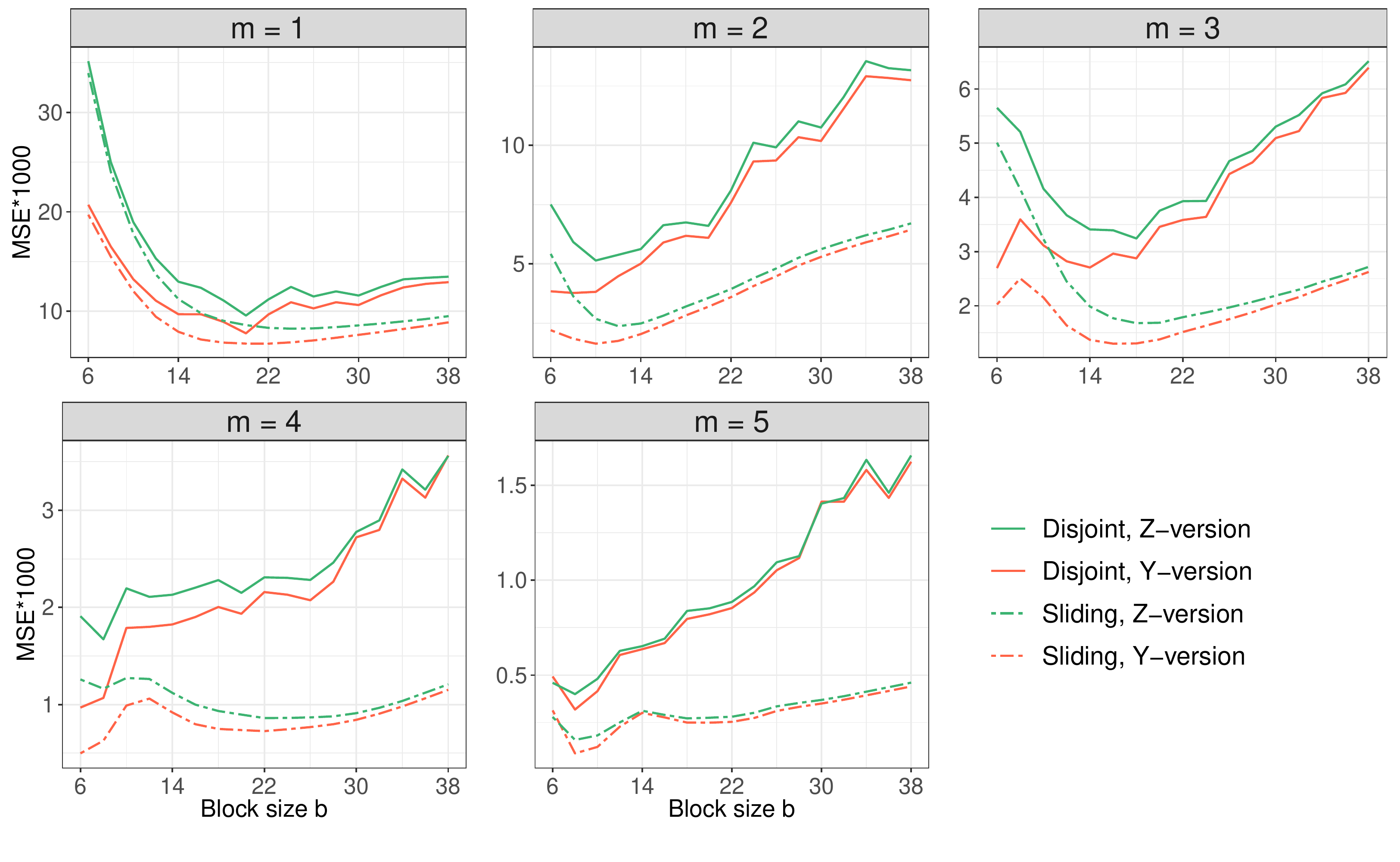} \vspace{-.8cm}
		\caption{Mean squared error multiplied by $10^3$ for the estimation of $\pi(m)$ within the AR-model for $m=1,\ldots,5$.} 
		\vspace{-.3cm} 
		%\label{Fig:AR_MSE}
	\end{center}
\end{figure}

\begin{figure} [t!]
	\begin{center}
		\includegraphics[width=\textwidth]{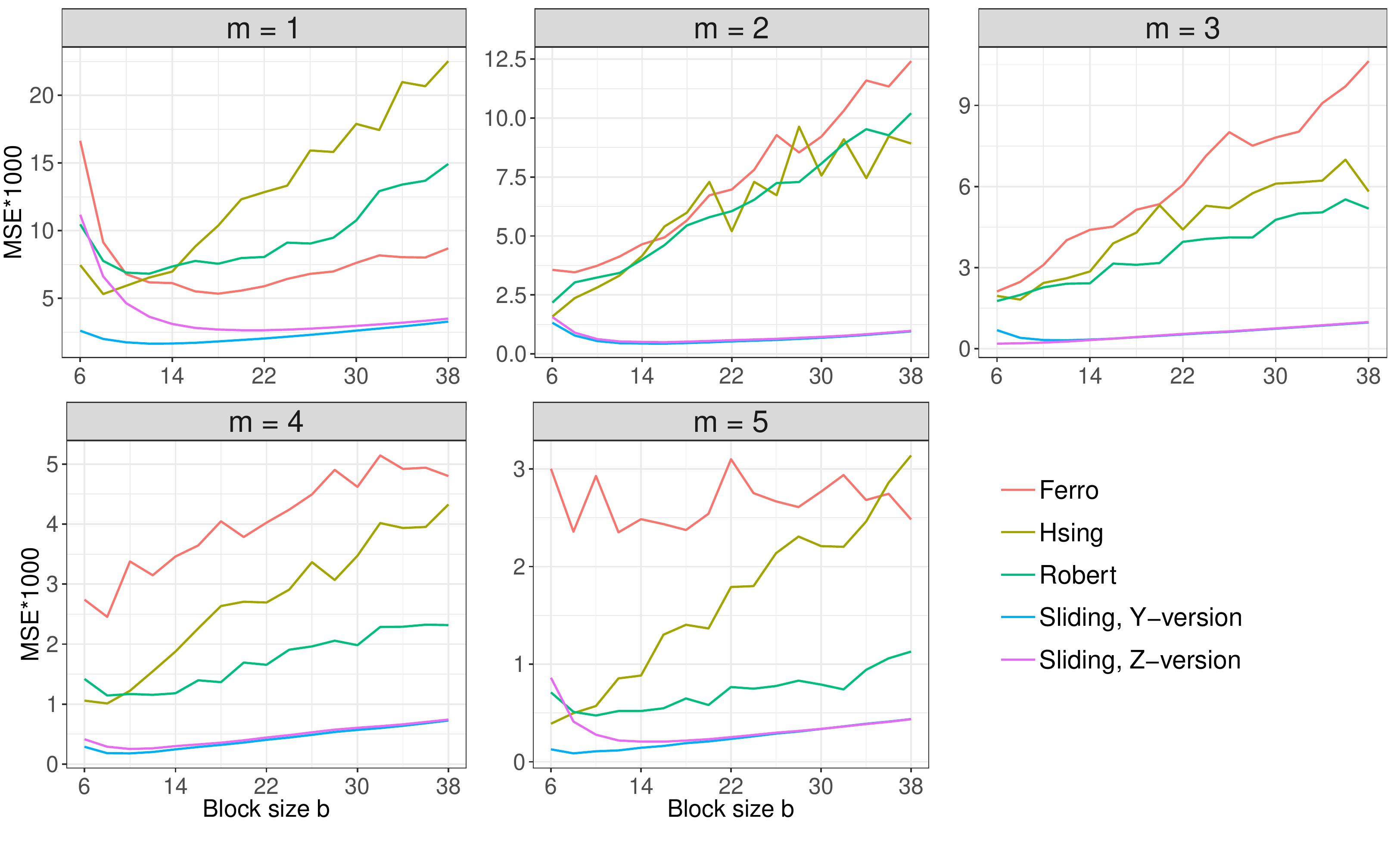} \vspace{-.8cm}
		\caption{Mean squared error multiplied by $10^3$ for the estimation of $\pi(m)$ within the ARMAX-model for $m=1,\ldots,5$.} 
		\vspace{-.3cm} 
		%\label{Fig:maxAR_MSE_all}
	\end{center}
\end{figure}

\begin{figure} [t!]
	\begin{center}
		\includegraphics[width=\textwidth]{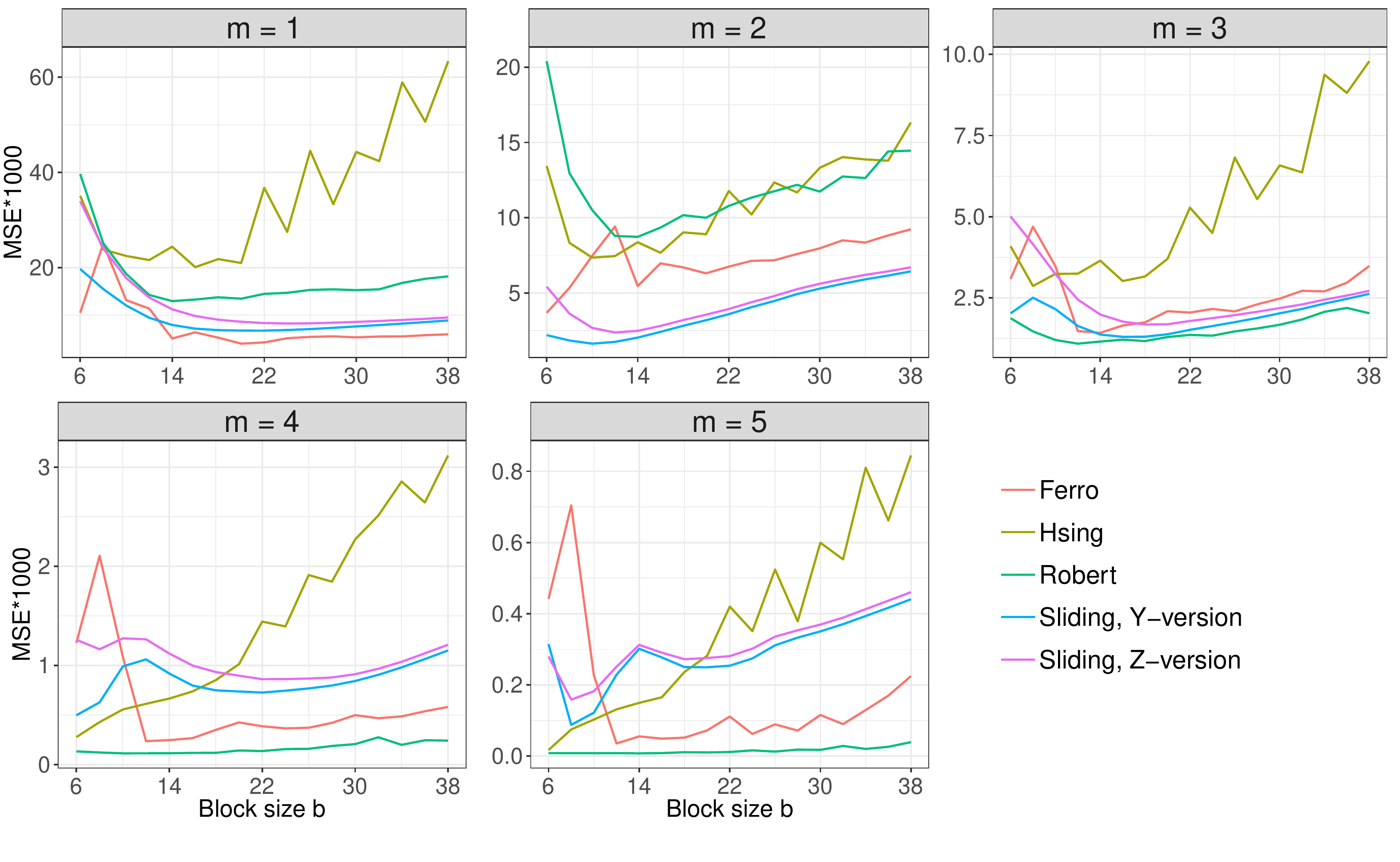} \vspace{-.8cm}
		\caption{Mean squared error multiplied by $10^3$ for the estimation of $\pi(m)$ within the AR-model for $m=1,\ldots,5$.} 
		\vspace{-.3cm} 
		\label{Fig:AR_MSE_all}
	\end{center}
\end{figure}

 \section*{Acknowledgements}
This work has been supported by the Collaborative Research Center ``Statistical modeling of nonlinear dynamic processes'' (SFB 823) of the German Research Foundation, which is gratefully acknowledged.

\bibliographystyle{chicago}
\bibliography{biblio}

\end{document}